\documentclass[10pt, a4paper]{article}
\usepackage{amsmath, amsthm, amssymb, mathtools}
\usepackage[svgnames]{xcolor}
\usepackage[left=1.5cm, right=1.5cm, top=1.5cm, bottom=2cm]{geometry}
\usepackage{hyperref}
\usepackage[notref, notcite, final]{showkeys}
\usepackage{bbm}
\usepackage{enumitem}
\usepackage{float}
\usepackage[left]{lineno}

\usepackage{caption}
\usepackage{graphicx}
\usepackage[backend=bibtex, sorting=nyt, style=numeric-comp, url=false, isbn=false, eprint=false, maxbibnames=4, giveninits=true]{biblatex}
\renewbibmacro{in:}{\ifentrytype{article}{}{\printtext{\bibstring{in}\intitlepunct}}}
\DeclareFieldFormat
[article,inbook,incollection,inproceedings,patent,thesis,unpublished]
{title}{#1\isdot}
\DeclareFieldFormat[article]{volume}{\mkbibbold{#1}}
\DeclareFieldFormat[article]{number}{\mkbibbold{#1}}

\usepackage{xr}

\newcommand\dd{\mathrm{d}}
\newcommand\supp{\mathrm{supp}\,}

\newcommand\essinf{\mathrm{ess\,inf}\,}
\newcommand\R{\mathbb R}
\newcommand\ep{\varepsilon}

\theoremstyle{plain}
\newtheorem{theorem}{Theorem}[section]
\newtheorem{lemma}[theorem]{Lemma}
\newtheorem{proposition}[theorem]{Proposition}
\newtheorem{corollary}[theorem]{Corollary}
\theoremstyle{definition}
\newtheorem{definition}[theorem]{Definition}
\newtheorem{remark}[theorem]{Remark}
\newtheorem{assumption}[theorem]{Assumption}
\newtheorem{claim}[theorem]{Claim}

\counterwithin{equation}{section}

\newcommand{\writefoot}[1]{
\renewcommand{\thefootnote}{}
\footnotetext{\hspace{-16.5pt}\scriptsize#1}
\renewcommand{\thefootnote}{\arabic{footnote}}
}

\author{Jean-Baptiste Burie, Arnaud Ducrot, Quentin Griette\footnote{Corresponding author. e-mail: \href{mailto:quentin.griette@u-bordeaux.fr}{\texttt{quentin.griette@u-bordeaux.fr}}}}

\begin{document}
\writefoot{\small \textbf{AMS subject classifications (2020).} 34D05, 92D25, 37L15, 37N25 \smallskip}
\writefoot{\small \textbf{Keywords.} ordinary differential equations, asymptotic behavior, population dynamics, Radon measure,  evolution.\smallskip}

\writefoot{\small \textbf{Acknowledgements:} 
J.-B. Burie and A. Ducrot are supported by the ANR project ArchiV ANR-18-CE32-0004. 
Q. Griette was partially supported by ANR grant  ``Indyana'' number  ANR-21-CE40-0008. 
A.D. and Q.G. acknowledge the support of the Math AmSud program for project 22-MATH-09. 
}

\begin{center}
	\hypersetup{hidelinks}
	\vspace{1cm}
	\renewcommand{\thefootnote}{\fnsymbol{footnote}}
	\begin{minipage}{0.9\textwidth}
		\centering
		\LARGE{\bf {Epidemic  models in measure spaces}: \\persistence,  concentration {and oscillations}}\bigskip
	\end{minipage}

	\Large
	Jean-Baptiste Burie$^{a}$, Arnaud Ducrot$^{b}$, and Quentin Griette$^{b,}$\footnote{Corresponding author. e-mail: \href{mailto:quentin.griette@univ-lehavre.fr}{\texttt{quentin.griette@univ-lehavre.fr}}}\medskip \\
	\medskip

	\today
	\bigskip

	\normalsize
	{\it $^a$ Institut de Math\'ematiques de Bordeaux, Universit\'e de Bordeaux, \\
	CNRS, IMB, UMR 5251, \\ 
	351, cours de la Lib\'eration, F-33400 Talence, France.}\\
	\medskip

	{\it $^b$ Universit\'e Le Havre Normandie, Normandie Universit\'e, LMAH, 76600 Le Havre, France. 
	}
	\hypersetup{hidelinks=false}
\end{center}
\bigskip

\begin{abstract}
	We investigate  the long-time dynamics of a SIR epidemic model in the case of a population of pathogens infecting a homogeneous host population. 
	The pathogen population is structured by a genotypic variable.
	When the initial mass of the maximal fitness set is positive, we give a precise description of the convergence of the orbit, including a formula for the asymptotic distribution. 
	We also investigate precisely the case of a finite number of regular global maxima and show that the initial distribution may have an influence on the support of the eventual distribution. 
	In particular,  the natural process of competition is not always selecting a unique species, but several species may coexist as long as they maximize the fitness function. 
	In many cases it is possible to compute the eventual distribution of the surviving competitors. 
	In some configurations, species that maximize the fitness may still get extinct depending on the shape of the initial distribution and some other parameter of the model, and we provide a way to characterize when this unexpected extinction happens.
	Finally, we provide an example of a pathological situation in which the distribution never reaches a stationary distribution but oscillates forever around the set of fitness maxima.
\end{abstract}
\bigskip

\section{Introduction}
In this article we investigate the large time behavior of the SIR epidemic model
\begin{subequations}\label{eq:main}
	\begin{equation}
		\left\{\begin{aligned}\relax
			&S_t(t) 		= \Lambda -\theta S(t) - S(t)\int_{{X}} \beta(x)I(t,  \dd x), \\ 
			&I_t(t,\dd x)		= \beta(x)S(t)I(t, \dd x) - \gamma(x)I(t, \dd x), \\
			&R_t(t) 		= \int_{X}\gamma(x)I(t, \dd x), 
		\end{aligned}\right.
	\end{equation}
	with the initial data
	\begin{equation}
		S(0)=S_0\in(0, +\infty), \qquad I(0, \dd x)=I_0(\dd x)\in\mathcal{M}_+(X), \qquad  R(0)= R_0\in (0, +\infty),
	\end{equation}
\end{subequations}
where $X$ is a Polish space and $\mathcal{M}_+(X)$ is the set of nonnegative Borel measures on $X$.

This model  describes the evolution of a population of hosts that can be, at any time $t>0$, either free of infection ($S(t)$, the susceptible population), infected by a pathogen of type $x\in X$ ($I(t, \dd x)$, the infected population of type $x$) or removed from the system ($R(t)$, the recovered population), the latter class including two possible outcomes of the infection,  complete immunity or death. The parameter $\Lambda>0$ models a constant influx of susceptible hosts, $\theta>0$ the death rate of the hosts in the absence of infection, $\beta(x)$ the transmission parameter of the pathogen of type $x$, and $\gamma(x)$ the recovery rate of a pathogen of type $x$. Both $\beta(x)$ and $\gamma(x)$ are bounded continuous functions.
SIR models are ubiquitous in the literature concerning mathematical epidemiology and have been extensively studied. Without the pretention of reconstructing the entire history of the model, let us cite the works of \textcite{Ker-McK-1927} that might well be its first occurrence in the literature, and was immediately applied to a plague outbreak in the island of Bombay.

In this article we consider that the phenotype of the pathogen (that is, the values of $\beta(x)$ and $\gamma(x)$) depends on an underlying variable $x\in X$, where $X$ is a set of attainable values that possesses a few mathematical properties. This variable $x$ may be a collection of quantitative phenotypic traits involved the mechanism of {transmission, reproduction or replication of the pathogen} (expression of surface protein at the cellular or viral level, impact on the host's behavior, ...) or the underlying genes that determine the values of $\beta(x)$ and $\gamma(x)$. We do not specify the particular mechanisms that link the underlying variable $x\in X$ and the phenotype $(\beta(x), \gamma(x))$ but focus on the dynamics of the population under \eqref{eq:main} conditionally to the knowledge of these mechanisms. By analogy, because it stands for a hidden process that determines an observable quantity, $x$ will be called the \textit{genotypic variable}, even if it could well stand for hidden quantitative phenotypic variables.  We do not take mutations into account and consider that the pathogen is asexual; therefore, \eqref{eq:main} can be considered a pure competition model where the pathogens compete for a single resource (the susceptible hosts).

{When $I_0(\dd x)$ is a finite collection of Dirac masses, 
\begin{equation*}
	I_0(\dd x) = \sum_{i=1}^n I_0^i\delta_{x_i}(\dd x), 
\end{equation*}
our problem is reduced to a system of ordinary differential equations, 
\begin{subequations}\label{eq:SI-finite}
	\begin{equation}
		\left\{\begin{aligned}\relax
			&\frac{\dd \phantom{t}}{\dd t}S(t) = \Lambda -\theta S(t) - S(t)\big(\beta_1 I^1(t)+\beta_2 I^2(t)+\ldots + \beta_n I^n(t)\big)\\ 
			&\frac{\dd \phantom{t}}{\dd t}I^1(t)=\beta_1 S(t)I^1(t) - \gamma_1I^1(t)  \\
			& \qquad\vdots \\ 
			&\frac{\dd \phantom{t}}{\dd t}I^n(t)=\beta_n S(t)I^n(t) - \gamma_nI^n(t), 
		\end{aligned}\right.
	\end{equation}
	with the initial data
	\begin{equation}
		S(0)=S_0\in(0, +\infty),\qquad   I^1(0)=I_0^1\in(0, +\infty), \qquad \ldots, \qquad I^n(0)=I_0^n\in(0, +\infty),
	\end{equation}
\end{subequations}
where $\beta_i=\beta(x_i)$, $\gamma_i=\gamma(x_i)$ and $I(t, \dd x)=\sum_{i=1}^n I^i(t)\delta_{x_i}(\dd x)$. In this context, \textcite{Hsu-Hub-Wal-77, Hsu-78} showed for closely related systems of ordinary differential equations that the solution eventually converges to an equilibrium which may not be unique but is always concentrated on the equations that maximize the fitness $\beta_i/\gamma_i$. One of the objectives of the current article is to establish an equivalent result in the context of measure-valued initial conditions that are not necessarily finite sums of Dirac masses (this will be given by Theorem \ref{thm:measures}). 
}

In a recent work \cite{Burie-Ducrot-Griette-2023} we clarified the asymptotic behavior of an extension of \eqref{eq:SI-finite} in the case of an infinite number of equations. This discrete setting is a particular case of the results that we present here, and we provide a number of illuminating examples that give a glimpse of the diversity of different behaviors that can be expected for solutions to \eqref{eq:main}. In particular, we show that, when the assumptions of Proposition \ref{prop:local-survival} are not satisfied, it is possible to construct exotic inital data for which the total mass of pathogen $\int_{X}I(t, \dd x)$ does not converge to a limit but oscillates between distinct values. We refer to \cite{Burie-Ducrot-Griette-2023} for details.  

The problem of several species competing for a single resource has received a lot of attention in the literature.  In this context, the ``Competitive exclusion principle''  states that ``Complete competitors cannot coexist'', which in particular means that given a number of species competing for the same resource in the same place, only one can survive in the long run. This idea was already present to some extent in the book of Darwin, and is sometimes referred to as Gause's law \parencite{Har-60}.
This problem of survival of competitors has attracted the attention of mathematicians since the '70s and many studies have proved this property in many different contexts -- let us mention the seminal works of 
\textcite{Hsu-Hub-Wal-77, Hsu-78}
followed by 
\textcite{Arm-McG-80, But-Wol-85, Wol-Lu-92, Hsu-Smi-Wal-96, Wol-Xia-97, Li-99, Rap-Ver-19}, 
to cite a few -- and also disproved in other contexts, for instance in fluctuating environments, see 
\textcite{Cus-80} and 
\textcite{Smi-81}. 
\textcite{Ack-All-03} study the competitive exclusion in an epidemic model with a finite number of strains, and describe how different species can coexist in some cases.

In our model, the fitness of a pathogen with genotype $x$ is given by the formula $\mathcal{R}_0(x) = \frac{\Lambda \beta(x)}{\theta \gamma(x)}=\frac{\Lambda}{\theta}\alpha(x)$, where $\alpha(x):=\frac{\beta(x)}{\gamma(x)}$; the competitive exclusion principle implies that the only genotypes that {eventually remain} are the ones that maximize $\mathcal{R}_0(x)$. When $\mathcal{R}_0(x)$ has a unique maximum then it is clear that $I(t, x)$ converges to a Dirac distribution concentrated at the maximum. But if $\mathcal{R}_0(x)$ attains its maximum at a more complex set -- from two isolated maxima to an entire line segment -- then the eventual weight of each fitness maximum in the population is less clear. We will give a partial description of this distribution here, that depends on the initial repartition of infected as well as the repartition of the phenotypic value $\gamma(x)$ in the vicinity of the set of fitness maxima.
We will show in particular that, while it is true that the species have to maximize the fitness function in order to survive, the natural process of competition is not selecting a unique genotypic value but several may coexist as long as they maximize the fitness function. In many cases it is possible to compute the eventual repartition of the surviving competitors. In some cases, species that maximize the fitness may still get extinct if the initial population is not sufficient, and we provide a method to characterize when this unexpected extinction happens.

Considering a situation where $\mathcal R_0$ has more than one maximum at the same exact level may appear artificial but is not without biological interest.  Indeed, the long-time behavior that we observe in these borderline cases can persist in transient time  upon perturbing the function $R_0$.
For example in the epidemiological context of \textcite{DayGandon2007}, it has been observed that a  strain 1 with a higher value of $\gamma$ and a slightly lower $\mathcal R_0$ value than a strain 2 may nevertheless be dominant for some time {(we reproduce such a behavior numerically in Figure \ref{Fig8})}. These borderline cases shed light on our understanding of the transient dynamics, see also \textcite{BDDD2020} where we explicit transient dynamics {for a related evolutionary model} depending on the {local flatness of the fitness function}.

{Quantitative traits} such as the virulence or the transmission rate of a pathogen, the life expectancy of an individual and more generally any observable feature such as height, weight, muscular mass, speed, size of legs, etc. are naturally represented using  continuous variables. Such a description of a population seems highly relevant and has been used mostly in modelling studies involving some kind of evolution \parencite{Mag-00, Mag-02, Bar-Per-07,Des-Mis-Jab-Rao-08,Bar-Mir-Per-09,  Bou-Cal-Meu-Mir-Per-Rao-12, Jab-Rao-11,Rao-11, Lor-Per-14, Lor-Per-Tai-17,Gri-19, Ducrot-Magal-Liu-Griette}. In this context, and this has been remarked before  \parencite{Des-Mis-Jab-Rao-08, Lorenzi-Pouchol-2020}, concentration on the maximum level set of the fitness function $\mathcal{R}_0(x)$ means that the classical mathematical framework of functions is not sufficient to describe accurately the dynamics of the solutions to \eqref{eq:main}. In this article, we will therefore extend our analysis to the case of Radon measures. Note that it is also natural to consider measures as initial data in epidemic models with an age of infection structure to model cohorts of patients, see \textcite{Demongeot-Griette-Maday-Magal}.

{{When $X=\mathbb{R}^N$,} System \eqref{eq:main} arises naturally as the limit of a mutation-selection model of spore-producing pathogen proposed by \parencite{Iacono2012} and studied mathematically by \textcite{Djidjou2018, Dji-Duc-Fab-17,BDDD2020, Burie2018b,BDGR2020} when the dynamics of the spores is very fast.
The system \eqref{eq:main} corresponds to the case of no mutations at all or, equivalently, the case of a fully concentrated kernel (equal to a Dirac mass at 0). 

{Despite our efforts, we were unable to find a precise description of the behavior of the solutions of \eqref{eq:main} in the literature when the initial condition $I_0(\dd x)$ is a Radon measure}. {Here we remark that the vector field of \eqref{eq:main} is locally Lipschitz continuous in the space $\mathbb R\times \mathcal{M}_+(X)\times\mathbb R$ (when $\mathcal M(X)$ is equipped with the total variation norm), so the existence of solutions is not the main difficulty. The solution $I(t, \dd x)$ can be written as 
\begin{equation*}
	I(t, \dd x) = e^{\beta(x)\int_0^t S(s)\dd s-t\gamma(x)}I_0(\dd x), 
\end{equation*}
so the solution is always a bounded continuous function multiplied by the initial data $I_0(x)$ at any finite time $t>0$. But to describe what happens as $t\to+\infty$ is not at all trivial. In Theorem \ref{thm:measures}, we distinguish two typical situations. When there is a positive initial mass on the set of maximal fitness ($\int_{\alpha(x)=\alpha^*}I_0(\dd x)>0$, where $\alpha^*=\sup_{x\in \supp I_0}\alpha(x)$ and we recall that $\alpha(x)=\beta(x)/\gamma(x)$), then we can show that the distribution of pathogens converges to a stationary distribution that we can compute explicitly. This is done with the help of a Lyapunov function that is essentially the same as the one used by \textcite{Hsu-78}. This is point i) of Theorem \ref{thm:measures}. The case when there is no initial mass  on the set of maximal fitness ($\int_{\alpha(x)=\alpha^*}I_0(\dd x)=0$) is less clear. We compactify the orbits by using the weak-* topology of measures and use this compactness to show the uniform persistence of the population thanks to a general argument from \textcite{Mag-Zha-05}. Then we show that the population on the sets of high fitness always grows faster than the one on sets of low fitness, and this allows us to control uniformly the Kantorovitch-Rubinstein distance between the solution $I(t, \dd x)$ and the space of measures that are concentrated on the set of maximal fitness, $\mathcal{M}_+\big(\{\alpha(x)=\alpha^*\}\big)$. Thus in this case also we can prove that the solution eventually concentrates on the set that maximizes the fitness. This is point  ii) of  Theorem \ref{thm:measures}.
}

In general, it is not true that the distribution $I(t, \dd x)$ eventually reaches a stationary distribution. We construct a counterexample in Section \ref{sec:no-convergence}. By carefully choosing the initial data and the fitness function $\alpha(x)=\beta(x)/\gamma(x)$, we construct a solution of \eqref{eq:main} with $I(t, \dd x\dd y)$ that approaches the unit circle of $\mathbb R^2$ but never stops turning  around it. This fact is illustrated numerically in  Section \ref{sec:numerical-oscillations}. We prove in Claim \ref{claim:osc} that the $\omega$-limit set of the integral of $I(t, \dd x \dd y)$ on the upper-half plane,  $\int_{\mathbb R\times\mathbb{R}^+}I(t, \dd x\dd y)$, contains at least two values, therefore $I(t, \dd x\dd y)$ does not converge to a stationary distribution. We also refer to \cite{Burie-Ducrot-Griette-2023} where we construct an exemple in a discrete setting where the total mass $\int_{X}I(t, \dd x)$ does not converge to a single value but oscillates between several values. 

{With some additional assumption we improve the description of the asymptotic behavior of $I(t, \dd x)$ compared to Theorem \ref{thm:measures} in case ii). In Assumption \ref{as:reg-bound} we impose a condition on the disintegration of $I_0(\dd x)$ with respect to $\alpha(x)$ to impose that the distribution of $I_0(\dd x)$ is uniformly positive around the maximum of $\gamma(x)$, $\gamma^*:=\sup_{\alpha(x)=\alpha^*}\gamma(x)$. Under this assumption, in Proposition \ref{prop:local-survival}, we refine the localization of the asymptotic concentration set of $I(t, \dd x)$ and we prove that the total mass of $I(t, \dd x)$, $\int_{X}I(t, \dd x)$, converges to a limit value. We also focus on the special case when the fitness function $\alpha(x)$ attains a finite number of interior regular maxima in the interior of $\supp I_0$, and when $X=\mathbb R^N$. In Theorem \ref{THEO-eta}, we show that the initial distribution around the maxima of the fitness function plays a crucial role in the asymptotic behavior of the solution. We show that the fitness maxima that keep a non-zero asymptotic population are the ones that maximize an ad-hoc score that involves the value of $\gamma(x)$ but also the dimension of the Euclidean space and the polynomial decay of the initial data around the fitness maximum.}

The structure of the paper is as follows. In section \ref{sec:main} we present our main results. More precisely, we state our results for persistence and concentration in section \ref{sec:main-general},
we give precise statements of our results concerning fitness functions with a finite number of regular maxima in section \ref{sec:dep-initial-data}, and in section \ref{sec:no-convergence} we provide a counterexample to the convergence of the distribution $I(t, \dd x)$ when the initial mass of fitness maxima is negligible. In section \ref{sec:numerics} we illustrate our results with numerical simulations. The corresponding figures are added at the end of the article.
In section \ref{sec:measure} we prove our results concerning general measure initial data (corresponding to the statements in section \ref{sec:main-general}). 
In section \ref{sec:finite-maxima} we prove our statements on the systems with a fitness function $\alpha(x)$ having a finite number of regular maxima (corresponding to the statements in section \ref{sec:dep-initial-data}).

\paragraph{Data availability} Data sharing not applicable to this article as no datasets were generated or analysed during the current study.

\section{Main results}
\label{sec:main}

Without loss of generality, the system \eqref{eq:main} can be rewritten as
\begin{subequations}\label{eq:SI-no-mut}
	\begin{equation}\label{eq:SI-no-mut-a}
		\left\{\begin{aligned}\relax
			&S_t(t) = \Lambda -\theta S(t) - S(t)\int_{{X}} \alpha(x)\gamma(x)I(t, \dd x), \\ 
			&I_t(t, \dd x)= \big(\alpha(x)S(t) - 1\big) \gamma(x)I(t, \dd x),\; x \in {X},
		\end{aligned}\right.
	\end{equation}
	with the initial data
	\begin{equation}
		S(0)=S_0\in (0, +\infty), \qquad I(0, \dd x)\in \mathcal{M}_+(X), 
	\end{equation}
\end{subequations}
by setting $\alpha(x):=\beta(x)/\gamma(x)$, and removing the third equation, which has no impact on the dynamics of the system. In the rest of the article we will study the system \eqref{eq:SI-no-mut} instead of \eqref{eq:main}. 

Before going to our results, we introduce some notations that will be used along this work. We work on a Polish space $X$ (\textit{i.e.} a metrizable space which is separable and complete for at least one metric) equipped with a complete distance $d$}. 
We denote by $\mathcal M(X)$ the set of finite signed Radon measures on $X$. Recall that $\mathcal M(X)$ is a Banach space when endowed with the {\it total variation norm} given by:
\begin{equation*}
	\Vert \mu\Vert_{TV}=|\mu|(X)=\int_{X}|\mu|(\dd x),\;\forall \mu\in\mathcal M(X).
\end{equation*}
This fact is proved for instance in \textcite[Vol. I, Theorem 4.6.1 p. 273]{Bog-07}. When $X$ is compact, it is possible to identify $\mathcal{M}(X)$ with the dual of the space of continuous functions over $X$, $C(X)$. This is the Riesz representation theorem \parencite[Vol. II, Theorem 7.10.4 p.111]{Bog-07}. 
When $X$ is an arbitrary Polish space, while it is true that every measure $\mu\in\mathcal{M}(X)$ yields a continuous linear functional on $BC(X)$ (the space of bounded continuous functions), the converse is no longer true \parencite[Vol. II, Example 7.10.3 p.111]{Bog-07}.

We denote by $\mathcal M_+(X)$ the set of the finite nonnegative measures on $X$. Observe that one has $\mathcal M_+(X)\subset\mathcal M(X)$ and $\mathcal M_+(X)$ is a closed subset of $\mathcal M(X)$ for the norm topology of $\Vert \cdot \Vert_{TV}$. An alternate topology on $\mathcal M(X)$ can be defined by the Kantorovitch-Rubinstein norm \parencite[Vol. II, Chap. 8.3 p. 191]{Bog-07}, 
\begin{equation*}
	\Vert \mu\Vert_0:=\sup\left\{\int f\dd \mu\,: \, f\in \mathrm{Lip}_1(X), \sup_{x\in B}|f(x)|\leq 1\right\},
\end{equation*}
wherein we have set 
\begin{equation*}
	\mathrm{Lip}_1(X):=\left\{f\in BC(X)\, :\, |f(x)-f(y)|\leq {d(x,y)},\;\forall (x,y)\in X^2\right\}.
\end{equation*}
Let us recall \parencite[Theorem 8.3.2]{Bog-07} that the metric generated by $\Vert \cdot\Vert_0$ on $\mathcal M_+(X)$ is equivalent to the weak-$\ast$ topology generated by tests against bounded continuous test functions.  Note however that this equivalence is true only for $\mathcal M_+(X)$ and cannot be extended to $\mathcal M(X)$ since the latter space is not (in general) complete for the metric generated by $\Vert \cdot\Vert_0$. We denote by $d_0$ this metric on $\mathcal M_+(X)$, that is
\begin{equation}\label{eq:kantorub}
	d_0(\mu, \nu):=\Vert \mu-\nu\Vert_0\text{ for all }\mu, \nu\in \mathcal M_+(X).
\end{equation}

About the parameters arising in \eqref{eq:SI-no-mut} our main assumption reads as follows.
\begin{assumption}\label{as:params-nomut}
	The constants $\Lambda>0$ and $\theta>0$ are given.
	The functions $\alpha(x)$ and $\gamma(x)$ are bounded and continuous from ${X}$ into $\R$ and there exist positive constants $\alpha^\infty$ and $\gamma_0<\gamma^\infty$ such that  
	\begin{equation*} 
		\alpha(x)\leq \alpha^\infty, \qquad 0<\gamma_0\leq \gamma(x)\leq \gamma^\infty\qquad \text{ for all } x\in{X}.
	\end{equation*}
	{We let $S_0>0$ be given and  $I_0(\dd x)\not\equiv 0$ be a finite nonnegative Radon measure and $S_0>0$ be given. We define the two quantities $\alpha^*\geq 0$ and $\mathcal R_0(I_0)$ by
	\begin{equation}\label{def-alpha*}
		\alpha^*=:\sup_{x\in {\rm supp}\,I_0} \alpha(x)\text{ and }	\mathcal R_0(I_0):=\frac{\Lambda}{\theta}\alpha^*.
	\end{equation}}
	We finally assume that  the set 
	\begin{equation}\label{eq:Leps}
		L_\varepsilon(I_0):=\{x\in X\,:\,\alpha(x)\geq \alpha^*-\varepsilon\}\cap \supp I_0 
		=\{x\in\supp I_0\,:\,\alpha(x)\geq \alpha^*-\varepsilon\}
	\end{equation}
	is compact when $\varepsilon>0$ is sufficiently small.
\end{assumption}

Let us observe that if $S_0\geq 0$ then \eqref{eq:SI-no-mut} equipped with the initial data $S(0)=S_0$ and $I(0,\dd x)=I_0(\dd x)$ has a unique  solution $S(t)\geq 0$ and $I(t,\dd x)\in \mathcal M_+({X})$ for all $t\geq 0$. This is a direct application of the Cauchy-Lipschitz Theorem in the Banach space $\mathbb{R}\times \mathcal{M}(X)$. It is not difficult to show that $S(t)$ and $I(t, \dd x)$ are \textit{a priori} bounded (this will be proved in Lemma \ref{lem:bounds-nomut}), hence the solution is global. In addition $I$ is given by a quasi-explicit formula:
$$
I(t,\dd x)=\exp\left(\gamma(x)\left(\alpha(x)\int_0^tS(s)\dd s-t\right)\right)I_0(\dd x).
$$
The above formula ensures that ${\rm supp}\,I(t,\cdot)={\rm supp}\,I_0$ for all $t\geq 0$.

We now split our main results into several parts. We first derive very general results about the large time behavior of the solution $(S,I)$ of \eqref{eq:SI-no-mut} when $I_0$ is an arbitrary Radon measure. We show that $I(t, \dd x)$ concentrates on the points that maximize both $\alpha$ and $\gamma$.
We then apply this result to consider the case where $I_0(\dd x)$ is a finite or countable sum of Dirac masses. We continue our investigations with an absolutely continuous initial measure with respect to Lebesgue measure and a finite set $\{\alpha(x)=\alpha^*\}$. In that setting we are able to fully characterize the points where the measure $I(t,\dd x)$ concentrates as $t\to\infty$.    

\subsection{{Persistence and concentration}}
\label{sec:main-general}

As mentioned above this subsection is concerned with the large time behavior of the solution $(S,I)$ of \eqref{eq:SI-no-mut} where the initial measure $I_0(\dd x)$ is an arbitrary Radon measure. Using the above notations our first result reads as follows.

\begin{theorem}[Asymptotic behavior of measure-valued initial data]\label{thm:measures} 
	Let Assumption \ref{as:params-nomut} be satisfied and suppose that $\mathcal R_0(I_0)>1$. Let $(S(t), I(t,\dd x))$ be the solution of \eqref{eq:SI-no-mut} equipped with the initial data $S(0)=S_0$ and $I(0,\dd x)=I_0(\dd x)$.
	We distinguish two cases depending on the measure of {the set $\{\alpha(x)=\alpha^*\}$ with respect to $I_0$:}
	\begin{enumerate}[label={\rm\roman*)}]
		\item If $\int_{\alpha(x)=\alpha^*}I_0(\dd x)>0$, then one has
			\begin{equation*}
				S(t)\xrightarrow[t\to+\infty]{}\frac{1}{\alpha^*} \text{ and } I(t, \dd x) \xrightarrow[t\to+\infty]{} I^\infty(\dd x) := \mathbbm{1}_{\alpha(x)=\alpha^*}(x) e^{\tau \gamma(x)} I_0(\dd x) , 
			\end{equation*} 
			where $\tau\in\mathbb R$ denotes the unique solution of the equation 
			\begin{equation*}
				\int_{{X}}\gamma(x)\mathbbm{1}_{\alpha(x)=\alpha^*}(x) e^{\tau \gamma(x)} I_0(\dd x) = \frac{\theta}{\alpha^*}\big(\mathcal R_0(I_0)-1\big). 
			\end{equation*}
			The convergence of $I(t, \dd x)$ to $I^\infty(\dd x)$ holds in the total variation norm $\Vert \cdot \Vert_{TV}$.

		\item If $\int_{\alpha(x)=\alpha^*}I_0(\dd x)=0$, then one has $S(t)\to \frac{1}{\alpha^*}$ and $I(t, \dd x)$ is uniformly persistent, namely 
			\begin{equation*}
				\liminf_{t\to+\infty} \int_{{X}} I(t, \dd x)>0 .
			\end{equation*}
			Moreover $I(t, \dd x)$  is asymptotically concentrated as $t\to\infty$ on the set $\{\alpha(x)=\alpha^*\}$, in the sense that 
			\begin{equation*}
				d_0\big(I(t, \dd x), \mathcal M_+(\{\alpha(x)=\alpha^*\})\big)\xrightarrow[t\to+\infty]{} 0 ,
			\end{equation*}
			where $d_0$ is the Kantorovitch-Rubinstein distance. 						
	\end{enumerate}
\end{theorem}
In the statement of the Theorem \ref{thm:measures} and in the rest of the paper, we stress for clarity that $\mathcal M_+\big(\{\alpha(x)=\alpha^*\}\big)$ is the set of finite positive measures on the closed set {$\{x\in X\,:\, \alpha(x)=\alpha^*\}$}. This set is naturally embedded as a subset of the space {$\mathcal{M}_+\big(X\big)$}, which is closed for the topology induced by the total variation norm $\Vert \cdot\Vert_{TV}$ and also the one induced by the Kantorovitch-Rubinstein distance $d_0$.

Theorem \ref{thm:measures} can be interpreted as follows. In case i), when the set of pathogens with maximal fitness is ``already populated'', the behavior of the dynamical system is no different from the case of a finite system: it converges to a equilibrium which is concentrated on the set of maximal fitness. The case ii), when the maximal fitness is not attained for the population but can only be reached asymptotically, is more intricate and we can only prove that the population of pathogens is uniformly persistent and asymptotically concentrated on the set of maximal fitness. We cannot prove the convergence to an equilibrium distribution in general; in fact, it is false, see the example in section \ref{sec:no-convergence} below and also the counterexamples in \cite{Burie-Ducrot-Griette-2023}. In fact, as shown in the latter reference, it is not even true in general that the total mass of pathogen $\int_{X}I(t, \dd x)$ converges to a limit.

We continue our general result by showing that under additional properties for the initial measure $I_0$, the function $I(t,\dd x)$ concentrates in the large times on the set of the points in $\{\alpha(x)=\alpha^*\}$ that also maximize the function $\gamma$.

The additional assumption for the initial measure $I_0(\dd x)$ are expressed in term of some properties of its disintegration measure with respect to the function $\alpha$ {on $L_\varepsilon(I_0)$ with $\varepsilon $ sufficiently small}. We refer to the book of 
\textcite[VI, \S 3, Theorem  1 p. 418]{Bour-Int} for a proof of the disintegration Theorem which is recalled in {the Appendix, Theorem \ref{thm:disintegration}.}

Let $A(\dd y)$ be the image of $I_0(\dd x)$ under the continuous mapping $\alpha:{X}\to \mathbb R$, then there exists a family of nonnegative measures $I_0(y, \dd x)$ (the disintegration of $I_0$ with respect to $\alpha$) such that for almost every $y\in\alpha(\supp I_0)$ with respect to $A$ we have:
\begin{equation}\label{eq:disint-def}
	\supp I_0(y, \dd x)\subset \{x\in X\,:\, \alpha(x)=y\},\quad  \int_{\alpha(x)=y} I_0(y, \dd x)=1 \text{ and }
	I_0(\dd x) = \int I_0(y, \dd x) A(\dd y) 
\end{equation}
wherein the last equality means that 
\begin{equation*}
	\int_{{X}} f(x) I_0(\dd x) = \int_{y\in\mathbb R} \int_{\alpha(x)=y} f(x)I_0(y, \dd x) A(\dd y) \text{ for all } f\in {BC(X)}.
\end{equation*}
Note that, by definition, the measure $A$ is supported on the set $\alpha({\supp I_0})$. {The measure $A$ is called the \textit{pushforward} measure of $I_0$ under {the mapping $\alpha$.} Note that the disintegration is unique up to a redefinition on an $A$-negligible set of fibers, see {the disintegration theorem recalled in the  Appendix, Theorem \ref{thm:disintegration}.}}

We shall also make use, for all $y$ $A-$almost everywhere, of the disintegration measure of $I_0(y,\dd x)$ with respect to the function $\gamma$, as follows
$$
I_0(y,\dd x)=\int_{\mathbb{R}} I_0^{\alpha, \gamma}(y, z, \dd x) I_0^\alpha(y, \dd z),
$$
where $I_0^{\alpha, \gamma}(y, z, \dd x)$ is concentrated on the set $\{x\,:\, \alpha(x)=y\text{ and } \gamma(x) = z\}$.
This allows to the following reformulation of $I_0(\dd x)$:
\begin{equation*}
	I_0(\dd x) = \int_{y\in \mathbb R}\int_{z\in \mathbb{R}} I_0^{\alpha, \gamma}(y, z, \dd x) I_0^\alpha(y, \dd z)A(\dd y).
\end{equation*}

\begin{remark}[Explicit disintegration in Euclidean spaces]
	Suppose that {$X=\mathbb R^N$ and} $I_0\in L^1(\mathbb R^N)$. Since we restrict to measures which are absolutely continuous with respect to the Lebesgue measure here,  with a small abuse of notation we will omit the element $\dd x$ when the context is clear.
	Assume that $\alpha$ is Lipschitz continuous  on  $\mathbb R^N$ and that
	\begin{equation}\label{eq:reg-alpha-I_0}
		\dfrac{I_0(x)}{|\nabla \alpha(x)|}\in L^1(\mathbb R^N).
	\end{equation}
	The coarea formula implies that, for all  $g\in L^1(\mathbb R^N)$, we have 
	\begin{equation*}
		\int_{\mathbb R^N}g(x)|\nabla\alpha(x)|\dd x = \int_{\mathbb R} \int_{\alpha(x)=y}g(x)\mathcal H_{N-1}(\dd x)\dd y, 
	\end{equation*}
	where $\mathcal H_{N-1}(\dd x)$ is the $(N-1)$-dimensional Hausdorff measure (see Federer \cite[\S 3.2]{Federer-69}).

	Therefore if $g(x)=f(x)\frac{I_0(x)}{|\nabla\alpha(x)|}$ we get
	\begin{equation}\label{eq:coarea-I_0}
		\int_{\mathbb R^N}f(x)I_0(x)\dd x = \int_{\mathbb R} \int_{\alpha(x)=y}f(x)\frac{I_0(x)}{|\nabla\alpha(x)|}\mathcal H_{N-1}(\dd x)\dd y, 
	\end{equation}
	and if moreover $f(x)=\varphi(\alpha(x))$ we get
	\begin{equation*}
		\int_{\mathbb R}\varphi(y)A(\dd y)=\int_{\mathbb R^N}\varphi(\alpha(x))I_0(x)\dd x = \int_{\mathbb R} \varphi(y) \int_{\alpha(x)=y}\frac{I_0(x)}{|\nabla\alpha(x)|}\mathcal H_{N-1}(\dd x)\dd y, 
	\end{equation*}
	where we recall that $A(\dd y) $ is the image measure of $I_0(\dd x)$ through $\alpha$. Therefore we have an explicit expression for $A(\dd y)$: 
	\begin{equation}\label{eq:explicit-image}
		A(\dd y) = \int_{\alpha(x)=y}\frac{I_0(x)}{|\nabla\alpha(x)|}\mathcal H_{N-1}(\dd x)\dd y
	\end{equation}
	and (recalling \eqref{eq:coarea-I_0}) we deduce the following explicit disintegration of $I_0$: 
	\begin{equation}\label{eq:explicit-disint}
		I_0(y, \dd x) = \dfrac{\mathbbm{1}_{\{x|\,\alpha(x)=y\}}\frac{I_0(x)}{|\nabla\alpha(x)|}\mathcal H_{N-1}(\dd x)}{\int_{\alpha(x)=y}\frac{I_0(z)}{|\nabla\alpha(z)|}\mathcal H_{N-1}(\dd z)}.
	\end{equation}
	Equations \eqref{eq:explicit-image} and \eqref{eq:explicit-disint} give an explicit formula for the disintegration introduced in \eqref{eq:disint-def}.
\end{remark}

\begin{remark}
	In particular, if $\alpha$ is a $C^2$ function with no critical point in $\supp I_0$ except for a finite number of regular maxima (in the sense that the bilinear form $D^2\alpha(x)$ is non-degenerate at each maximum;  this is a typical situation), 
	then the assumption \eqref{eq:reg-alpha-I_0} above is automatically satisfied if $N\geq 3$ and $I_0\in L^\infty(\mathbb R^N)$.  If $N=2$ then a sufficient condition to satisfy \eqref{eq:reg-alpha-I_0} with $I_0\in L^\infty(\mathbb R^N)$ should involve $I_0$ vanishing sufficiently fast in the neighborhood of each maximum of $\alpha$.
\end{remark}

Now equipped with this disintegration of $I_0$ with respect to $\alpha$ we are now able to state our regularity assumption to derive more refine concentration information in the case where $\int_{\alpha(x)=\alpha^*}I_0(\dd x)=0$. 

\begin{assumption}[Regularity with respect to $\alpha, \gamma$]\label{as:reg-bound}
	Let us define $\gamma^*>0$ by
	\begin{equation}\label{def-gamma*}
		\gamma^*:=\sup_{\alpha(x)=\alpha^*} \gamma(x).
	\end{equation}
	We assume that, for each value $\bar\gamma<\gamma^*$ {in a neighborhoof of $\gamma^*$,} there exist constants $\delta>0$ and  $m>0$ such that 
	\begin{equation*}
		m\leq \int_{\gamma(x)\in([\bar \gamma,\gamma^*]) \text{ and } \alpha(x)=y} I_0(y, \dd x) \text{ for $A$-almost every } y\in (\alpha^*-\delta, \alpha^*]. 
	\end{equation*}
\end{assumption}
\begin{remark}
	The assumption \ref{as:reg-bound}  means that the initial measure $I_0(\dd x)$ is uniformly positive in a small neighborhood of $\{\alpha(x)=\alpha^*\}\cap \{\gamma(x)=\gamma^*\}$. For instance, if the assumptions  \eqref{eq:reg-alpha-I_0} is satisfied and if  there exists an open set $U$ containing $\{\alpha(x)=\alpha^*\}\cap \{\gamma(x)=\gamma^*\}$ on which $I_0(x)$ is almost everywhere uniformly positive, then Assumption \ref{as:reg-bound} is automatically satisfied. 
\end{remark}

The next proposition ensures {that, when the initial measure $I_0$ satisfies Assumption \ref{as:reg-bound}, then} the function $I(t,\dd x)$ concentrates on $\{x\in X\,:\,\alpha(x)=\alpha^* \text{ and } \gamma(x)=\gamma^*\}$.  
\begin{proposition}\label{prop:local-survival}
	Let Assumption \ref{as:params-nomut} hold, and suppose that $\mathcal{R}_0(I_0)>1$. Assume moreover that $\int_{\alpha(x)=\alpha^*}I_0(\dd x)=0$ and that Assumption \ref{as:reg-bound} holds, then: 
	\begin{enumerate}[label={\rm\roman*)}]
		\item The measure $I(t, \dd x)$ concentrates on the set $\{\alpha(x)=\alpha^*\}\cap\{\gamma(x)=\gamma^*\}$:
			\begin{equation*}
				d_0\left(I(t, \dd x), \mathcal M_+(\{\alpha(x)=\alpha^*\}\cap\{\gamma(x)=\gamma^*\})\right)\xrightarrow[t\to+\infty]{} 0.
			\end{equation*}
		\item The total mass of $I(t, \dd x)$ converges to a limit value:
			\begin{equation*}
				\int_{{X}}I(t, \dd x) \xrightarrow[t\to+\infty]{}\frac{\theta}{\alpha^*\gamma^*}\big(\mathcal R_0(I_0)-1\big). 
			\end{equation*}
		\item If there exists a Borel set $U\subset{X} $ such that $U\cap\{x\,:\,\alpha(x)=\alpha^*\text{ and }\gamma(x)=\gamma^*\}\neq \varnothing$ and
			\begin{equation*}
				\liminf_{\varepsilon\to 0}\overset{A(\dd y)}{\underset{\alpha^*-\varepsilon\leq y\leq \alpha^*}{\essinf}}\overset{I_0^\alpha(y, \dd z)}{\underset{\gamma^*-\varepsilon\leq z\leq \gamma^*}{\essinf}}
				\int_U I_0^{\alpha, \gamma}(y,z, \dd x)>0,  
			\end{equation*}
			then the following persistence occurs
			\begin{equation*}
				\liminf_{t\to\infty} \int_{{U}}I(t, \dd x) >0.
			\end{equation*}
	\end{enumerate}
\end{proposition}

{Here the notation $\overset{\mu(\dd y)}{\underset{K}{\essinf}}$ denotes the essential infimum taken on the set $K$ and relatively to the measure $\mu$ (i.e. up to redefinition on $\mu$-negligible sets). {The condition on the set $U$ in iii) means that the distribution of the population does not vanish in a neighborhood of $\{x\,:\,\alpha(x)=\alpha^*\text{ and }\gamma(x)=\gamma^*\}$ in $U$; the existence of a set $U$ cannot be guaranteed in general, in particular, no such $U$ exists in the case of the counterexample given in subsection \ref{sec:no-convergence}}.}

\subsection{{Refined concentration estimates: regular fitness maxima in Euclidean spaces}}
\label{sec:dep-initial-data}

{ 
We now {deepen} our analysis of \eqref{eq:SI-no-mut} set on {$X=\mathbb R^N$ when the fitness function has a finite number of regular maxima. We will not discuss here the case of a unique global maximum, in which the precise asymptotic behavior can be completely determined: see the Appendix for a statement of what we obtain. Rather, }
we describe the large time behavior of the solutions when the function $\alpha$ has a finite number of maxima on the support of $I_0(\dd x)$. {We} consider an initial data $(S_0,I_0)\in [0,\infty)\times \mathcal M_+(\R^N)$ with $I_0$ absolutely continuous with respect to the Lebesgue measure $\dd x$ in $\R^N$ {(in other words and with a small abuse of notation, $I_0\in L^1(\mathbb R^N)$)} in a neighborhood of the maxima of the fitness function. Recalling the definition of $\alpha^*$ in \eqref{def-alpha*}, throughout this section, we shall make use of the following set of {assumptions}.

{By a small abuse of notation, we will identify in this section the function $I_0\in L^1(\mathbb R^N)$ and the associated measure $I_0(x)\dd x\in \mathcal M_+(\mathbb R^N)$ when the context is clear.}
\begin{assumption}\label{ASS-calculs}
	We assume that:
	\begin{itemize}
		\item[(i)] the set $\{\alpha(x)=\alpha^*\}$  is a finite set, namely there exist $x_1,...,x_p$ in the interior of $\supp(I_0)$ such that $x_i\neq x_j$ for all $i\neq j$ and 
			$$
			\{\alpha(x)=\alpha^*\}=\{x_1,\cdots,x_p\}\text{ and }\mathcal R_0:=\frac{\Lambda\alpha^*}{\theta}>1.
			$$
		\item[(ii)] There exist $\ep_0>0$, $M>1$ and $\kappa_1\geq 0$,..,$\kappa_p\geq 0$ such that for all $i=1,..,p$ and for {almost} all $x\in B(x_i,\ep_0)\subset \supp (I_0)$ one has
			\begin{equation*}
				M^{-1}|x-x_i|^{\kappa_i}\leq I_0(x)\leq M|x-x_i|^{\kappa_i}.
			\end{equation*}
			Here and along this note we use $|\cdot|$ to denote the Euclidean norm of $\R^N$.
		\item[(iii)] The functions $\alpha$ and $\gamma$ are of class $C^2$ and there exists $\ell>0$ such that for each $i=1,..,p$ one has
			\begin{equation*}
				D^2\alpha(x_i)\xi^2\leq -\ell|\xi|^2,\;\forall \xi\in\R^N.
			\end{equation*}
	\end{itemize}
\end{assumption}

\begin{remark}
	Let us observe that since $x_i$ belongs to the interior of $\supp(I_0)$ then $D\alpha(x_i)=0$. 
\end{remark}

In order to state our next result, we introduce the following notation: we write $f(t)\asymp g(t)$ as $t\to\infty$ if there exists $C>1$ and $T>0$ such that
$$
C^{-1}|g(t)|\leq |f(t)|\leq C|g(t)|,\;\forall t\geq T.
$$

{According} to Theorem \ref{thm:measures} $(ii)$, one has $\alpha^*S(t)\to 1$ as $t\to\infty$, and as a special case we conclude that
$$
\bar S(t)=\frac{1}{t}\int_0^t S(l)\dd l\to \frac{1}{\alpha^*}\text{ as }t\to\infty.
$$
As a consequence the function $\eta(t):=\alpha^*\bar S(t)-1$ satisfies $\eta(t)=o(1)$ as $t\to\infty$.
To describe the asymptotic behavior of the solution $(S(t),I(t,\dd x))$ with initial data $S_0$ and $I_0$ as above, we shall derive a precise behavior of $\eta$ for $t\gg 1$. This refined analysis will allow us to characterize the points of concentration of $I(t,\dd x)$. Our result reads as follows.

\begin{theorem}\label{THEO-eta}
	Let Assumption \ref{ASS-calculs} be satisfied.
	Then the function $\eta=\eta(t)$ 
	satisfies the following asymptotic expansion
	\begin{equation}\label{expansion}
		\eta(t)= \varrho\frac{\ln t}{t}+O\left(\frac{1}{t}\right),\text{ as }t\to\infty.
	\end{equation}
	wherein we have set
	\begin{equation}\label{def-varrho}
		\varrho:=\min_{i=1, \ldots, p}\frac{N+\kappa_i}{2\gamma(x_i)}.
	\end{equation}
	Moreover there exists $\ep_1\in (0,\ep_0)$ such that for all $0<\ep<\ep_1$ and all $i=1,..,p$ one has
	\begin{equation}\label{mass-esti}
		\int_{|x-x_i|\leq \ep}I(t,\dd x)\asymp t^{\gamma(x_i)\varrho-\frac{N+\kappa_i}{2}}\text{ as }t\to\infty.
	\end{equation}
	As a special case, for all $\ep>0$ small enough and all $i=1,..,p$ one has
	\begin{equation*}
		\int_{|x-x_i|\leq \ep}I(t,\dd x)\;\begin{cases}\;\asymp 1\text{ if $i\in J$}\\ \to 0\text{ if $i\notin J$}\end{cases}\text{as $t\to\infty$},
	\end{equation*}
	{where $J$ is the set defined as
	\begin{equation}\label{def-J}
		J:=\left\{i=1,..,p:\;\frac{N+\kappa_i}{2\gamma(x_i)}=\varrho\right\}.
	\end{equation}
	}
\end{theorem}

The above theorem states that the function $I(t,\dd x)$ concentrates on the set of points $\{x_i,\;i\in J\}$ (see Corollary \ref{CORO-conv} below). Here Assumption \ref{as:reg-bound} on the uniform positiveness of the measure $I_0(\dd x)$ around the points $x_i$ is not satisfied in general, and {therefore} the measure $I$ concentrates on $\{\alpha(x)=\alpha^*\}$ as {predicted by} Theorem \ref{thm:measures}, but not necessarily on $\{\alpha(x)=\alpha^*\}\cap \{\gamma(x)=\gamma^*\}$ as {would have been given by} Proposition~\ref{prop:local-survival}. {In Figure \ref{Fig7} we provide a precise example of this non-standard behavior.}

In addition, the precise expansion of $\eta=\eta(t)$ provided in the above theorem allows us obtain the self-similar behavior of the solution $I(t,\dd x)$ around the maxima of the fitness function. This asymptotic directly follows from \eqref{eq:I-nomut}.
\begin{corollary}\label{CORO-self-similar}
	For each $i=1,...,p$ and $f\in C_c(\R^N)$, the set of the continuous and compactly supported functions,
	one has as $t\to\infty$:
	\begin{equation}
		\label{eqCORO-self-similar}
		t^{\frac{N}{2}}\int_{\R^N}f\left((x-x_i)\sqrt t\right)I(t,\dd x)\asymp t^{\gamma(x_i)\varrho -\frac{N+\kappa_i}{2}}\int_{\R^N}f(x)|x|^{\kappa_i}\exp\left(\frac{\gamma(x_i)}{2\alpha^*}D^2\alpha(x_i)x^2\right)\dd x.
	\end{equation}
\end{corollary}

Our next corollary relies on some properties of the $\omega-$limit set of 
the solution $I(t,\dd x)$. Using the estimates of the mass around $x_i$ given in \eqref{mass-esti}, it readily follows that any limit measures of $I(t,\dd x)$ belongs to a linear combination of $\delta_{x_i}$ with $i\in J$ and strictly positive coefficients of each of these Dirac masses. This reads as follows.   
\begin{corollary}\label{CORO-conv}
	Under the same assumptions as in Theorem \ref{THEO-eta}, the $\omega-$limit set $\overline{\mathcal O}(I_0)$ as defined in Lemma \ref{lem:unifpers-nomut} satisfies that there exist $0<A<B$ such that
	\begin{equation*}
		\overline{\mathcal O}(I_0)\subset \left\{\sum_{i\in J}c_i\delta_{x_i}:\; (c_i)_{i\in J}\in \left[A,B\right]^J\right\}.
	\end{equation*}
\end{corollary}

\subsection{{Oscillations}}
\label{sec:no-convergence}
Here we construct a counterexample which shows that, in general, it is hopeless to expect convergence of the genotypic distribution to a stationary measure on $X=\mathbb R^2$. Such a counterexample is new in the case of ``continuous'' spaces; we provided some counterexamples in the case of discrete spaces in \cite{Burie-Ducrot-Griette-2023}.

Fix $L>0$ and let $\Omega$ be the parametric curve described as
\begin{equation*}
	\Omega:=\left\{\omega(\tau):=\big(1-e^{-\tau}\big) \left(\cos\left(\frac{\pi}{L}\tau\right), \sin\left(\frac{\pi}{L}\tau\right)\right)\,:\, \tau\in\mathbb R^+\right\}, 
\end{equation*}
and define $I_0(\dd x \dd y) \in \mathcal M(\mathbb R^2)$ as the pushforward of the measure $e^{-\tau}\mathbbm{1}_{\tau\geq 0}\dd \tau$ by $\omega$. That is to say, 
\begin{equation*}
	\int_{\mathbb R^2} \varphi(x, y) I_0(\dd x \dd y) = \int_{\mathbb R^+}e^{-\tau} \varphi\left(\big(1-e^{-\tau}\big)\cos\left(\frac{\pi}{L}\tau\right), \big(1-e^{-\tau}\big)\sin\left(\frac{\pi}{L}\tau\right)\right) \dd \tau, \text{ for all } \varphi\in BC(\mathbb R^2).
\end{equation*}
Then we select
\begin{equation*}
	\alpha(x, y):=1-\left|1-\sqrt{x^2+y^2}\right|, \qquad \gamma(x, y)\equiv 1.
\end{equation*}
In this setting, $\alpha$ attains its global maximum on the unit circle in $\mathbb R^2$, while the support of $I_0$ (the curve $\Omega$) approaches the unit circle from the inside with an exponentially decreasing mass as the radius converges to 1. We are thus in the situation described in case ii) of Theorem \ref{thm:measures}; in particular, it is true that $\overline{S}(t)=\frac{1}{t}\int_0^t S(s)\dd s\to 1$ as $t\to +\infty$. Yet, explicit computations show that
\begin{equation*}
	I(t, \dd x \dd y)=e^{(1-e^{-\tau})\int_0^tS(s)\dd s-t}I_0(\tau)\omega(\dd \tau) = e^{(1-e^{-\tau})t\overline{S}(t)-t-{\tau}}\omega(\dd \tau)=:I(t, \tau) \omega(\dd \tau), 
\end{equation*}
where we denote $\omega(\dd\tau)$ the pushforward of the Lebesgue measure on $\mathbb R^+$ onto $\Omega$ and, with a small abuse of notation, $I_0(\tau) = e^{-\tau}\mathbbm{1}_\tau\geq 0$. More precisely, 
\begin{equation*}
	\int_{\mathbb R^2} \varphi(x, y) I(t, \dd x\dd y) = \int_{\mathbb R^+}I(t, \tau) \varphi\left(\big(1-e^{-\tau}\big)\cos\left(\frac{\pi}{L}\tau\right), \big(1-e^{-\tau}\big)\sin\left(\frac{\pi}{L}\tau\right)\right) \dd \tau, \text{ for all } \varphi\in BC(\mathbb R^2).
\end{equation*}
Now by using explicit computations, we establish the following claim. 
\begin{claim} \label{claim:asym}
	The function $I(t, \tau)$ is, up to a multiplicative error of order zero, a solitary wave whose position behaves like $\tau_0(t):=\ln(t)$: 
	\begin{equation}
		I(t, \tau) = e^{u(\tau-\tau_0(t))+o(1)}, 
	\end{equation}
	where $u(\tau) = \ln(\mathcal{I}^\infty)-e^{-\tau}-\tau$ and $\mathcal{I}^\infty:=\frac{\theta}{\alpha^*}(\mathcal{R}_0-1)$ is given by Proposition \ref{prop:local-survival}.
\end{claim}
We now prove the claim.
We note that Assumption \ref{as:reg-bound} is clearly satisfied since $\gamma\equiv 1$ therefore, by Proposition \ref{prop:local-survival}, we have  $\int_0^{+\infty} I(t, \tau)\dd \tau = \mathcal{I}^\infty+o(1)$ with $\mathcal{I}^\infty:=\theta\left(\frac{\Lambda}{\theta}\alpha^*-1\right)>0$. By computing the integral of $I(t,\tau)$, we will now obtain a precise description of the behavior of $\overline{S}(t)$ as $t\to+\infty$. Indeed,
\begin{align*}
	\int_0^{+\infty} I(t,\tau) \dd \tau &= \dfrac{e^{t\overline{S}(t)-t}}{t\overline{S}(t)}\int_0^{+\infty} t\overline{S}(t)e^{-\tau}e^{-t\overline{S}(t)e^{-\tau}}\dd \tau = \dfrac{e^{t\overline{S}(t)-t}}{t\overline{S}(t)}\left[e^{-t\overline{S}(t)e^{-\tau}}\right]_0^{+\infty} \\ 
	&=\dfrac{e^{t\overline{S}(t)-t}}{t\overline{S}(t)}\left(1-e^{-t\overline{S}(t)}\right) = \dfrac{e^{t\overline{S}(t)-t}}{t\overline{S}(t)}-\dfrac{e^{-t}}{t\overline{S}(t)} = \dfrac{e^{t\overline{S}(t)-t}}{t\overline{S}(t)}+o(1),
\end{align*}
therefore, recalling $\overline{S}(t)=1+o(1)$, we have 
$e^{t\overline{S}(t)-t} = t\mathcal{I}^\infty + o(t)$, 
and finally 
\begin{equation*}
	\overline{S}(t) = 1+\dfrac{\ln(t)}{t} + \dfrac{\ln(\mathcal{I}^\infty)}{t}+o\left(\frac{1}{t}\right).
\end{equation*}
We deduce that
\begin{align*}
	I\big(t, \tau + \tau_0(t)\big) &=\exp\left(\big(1-e^{-(\tau+\tau_0(t))}\big)\big(t+\ln(t)+\ln(\mathcal{I}^\infty)+o(1)\big) - t -(\tau+\tau_0(t))\right) \\
	&=\exp\left(\left(1-\dfrac{e^{-\tau}}{t}\right)\big(\ln(t)+\ln(\mathcal{I}^\infty)+o(1)\big) - te^{-(\tau+\tau_0(t))} -(\tau+\tau_0(t))\right) \\
	&=\exp\left(\ln(t)+\ln(\mathcal{I}^\infty)+o(1)+ e^{-\tau}-\tau-\ln\left(t\right)\right)\\
	&=\exp\left(\ln(\mathcal{I}^\infty)-e^{-\tau}-\tau + o(1)\right), 
\end{align*}
which proves the claim.\medskip

{Note that the above computations and in particular the result of Claim \ref{claim:asym} are completely independent of the parameter $L$. We see that  $I(t, \dd x\dd y)$ is asymptotically equivalent to a rotating mass which becomes concentrated on the unit circle and does not converge to a static distribution. We illustrate this fact in numerical simulations in section \ref{sec:numerical-oscillations}.  We can also prove that the distribution does not reach stationarity when the rotation speed of the spiral (which behaves like $L^{-1}$) is very slow. Indeed, the integral in the upper half-space never reaches a stationary value, as we show in the following Claim. 
\begin{claim} \label{claim:osc}
	There exists a function $\varepsilon(L)\geq 0$ such that $\varepsilon(L)\to 0$ as $L\to\infty$, and two sequences $t^1_n:=e^{(2n+1/2)L}\to +\infty$ and $t^2_n:=e^{(2n+3/2)L}\to +\infty$ such that 
	\begin{equation}\label{eq:osc-1}
		\liminf_{n\to+\infty}\int_{\mathbb R\times\mathbb R^+} I(t^1_n, \dd x \dd y) \geq \mathcal{I}^\infty-\varepsilon(L)
	\end{equation} 
	and 
	\begin{equation}\label{eq:osc-2}
		\limsup_{n\to+\infty}\int_{\mathbb R\times \mathbb R^+}I(t^2_n, \dd x \dd y) \leq \varepsilon(L).
	\end{equation}
\end{claim}
Indeed, we have 
\begin{align}
	\int_{\mathbb R\times \mathbb R^+}I(t, \dd x\dd y)&= \sum_{k=0}^{+\infty} \int_{2kL}^{(2k+1)L} I(t, \tau)\dd \tau = \sum_{k=0}^{+\infty}\int_{2kL}^{(2k+1)L} e^{u(\tau-\tau_0(t))+o(1)}\dd \tau \nonumber\\ 
	&= \big(1+o(1)\big) \sum_{k=0}^{+\infty}\int_{2kL-\tau_0(t)}^{(2k+1)L-\tau_0(t)} e^{u(\tau)}\dd \tau = \big(1+o(1)\big)\sum_{k=0}^{+\infty}\int_{2kL-\tau_0(t)}^{(2k+1)L-\tau_0(t)} \mathcal{I}^{\infty}e^{-e^{-\tau}-\tau}\dd \tau\nonumber\\
	&=\big(1+o(1)\big)\sum_{k=0}^{+\infty}\mathcal{I}^\infty\left(\exp\left(-e^{-(2k+1)L+\tau_0(t)}\right)-\exp\left(-e^{-2kL+\tau_0(t)}\right)\right),\label{eq:221120-1}
\end{align}
so when $t=t^1_n=e^{(2n+1/2)L}$ we have
\begin{align*}
	\int_{\mathbb R\times \mathbb R^+}I(t^1_n, \dd x\dd y)&=\big(1+o(1)\big)\sum_{k=0}^{+\infty}\mathcal{I}^\infty\left(\exp\left(-e^{-(2k+1)L+(2n+1/2)L}\right)-\exp\left(-e^{-2kL+(2n+1/2)L}\right)\right) \\ 
	&=\mathcal{I}^\infty\big(1+o(1)\big)\sum_{k=-n}^{+\infty}e^{-e^{-2kL-L/2}}-e^{-e^{-2kL+L/2}} 
\end{align*}
We note that  $ e^{-e^{-(2k\pm 1/2)L}}\leq e^{-(2k-1/2)L} $ as $L\to \infty$ when $k\leq -1$, and that for $k\geq 1$ we have by a Taylor expansion:
\begin{equation*}
	|e^{-e^{-(2k-1/2)L}}-e^{-e^{-(2k+1/2)L}}| \leq Ce^{-(2k-1/2)L},
\end{equation*}
for some constant $C>0$ independent of $L$ and $k$. Therefore by the dominated convergence theorem we have
\begin{equation*}
	\lim_{L\to+\infty} \left(\sum_{k=-\infty}^{-1} +\sum_{k=1}^{+\infty}\right)\left|e^{-e^{-2kL-L/2}}-e^{-e^{-2kL+L/2}}\right| = 0.
\end{equation*}
Now for $k=0$ we have $e^{-e^{-2kL-L/2}}-e^{-e^{-2kL+L/2}}=e^{-e^{-L/2}}-e^{-e^{L/2}}\to 1 $ as $L\to+\infty$. We have shown
\begin{equation*}
	\liminf_{L\to+\infty}\liminf_{n\to+\infty}\int_{\mathbb R\times \mathbb R^+}I(t^1_n, \dd x\dd y) \geq \mathcal{I}^\infty,
\end{equation*}
which proves \eqref{eq:osc-1}. Now \eqref{eq:221120-1} with $t=t^2_n=e^{(2k+3/2)L}$ leads us to 
\begin{align*}
	\int_{\mathbb R\times \mathbb R^+}I(t^2_n, \dd x\dd y)&=\big(1+o(1)\big)\sum_{k=0}^{+\infty}\mathcal{I}^\infty\left(\exp\left(-e^{-(2k+1)L+(2n+3/2)L}\right)-\exp\left(-e^{-2kL+(2n+3/2)L}\right)\right) \\ 
	&=\mathcal{I}^\infty\big(1+o(1)\big)\sum_{k=-n}^{+\infty}e^{-e^{-2kL+L/2}}-e^{-e^{-2kL+3L/2}}.
\end{align*}
Note that we have    
\begin{equation*}
	\lim_{L\to+\infty} \left(\sum_{k=-\infty}^{-1} +\sum_{k=1}^{+\infty}\right)\left|e^{-e^{-2kL+L/2}}-e^{-e^{-2kL+3L/2}}\right| = 0.
\end{equation*}
by the dominated convergence theorem and, for $k=0$, $e^{-e^{-2kL+L/2}}-e^{-e^{-2kL+3L/2}} = e^{-e^{L/2}}-e^{-e^{3L/2}}\to 0$ as $L\to\infty$. This shows 
\begin{equation*}
	\limsup_{L\to+\infty}\limsup_{n\to+\infty}\int_{\mathbb R\times \mathbb R^+}I(t^2_n, \dd x\dd y) \leq 0, 
\end{equation*}
which finishes the proof of Claim \ref{claim:osc}.
}

\section{Comments and numerical illustrations}
\label{sec:numerics}

\subsection{Numerical illustrations: the main Theorem}

In this section we provide numerical illustrations to some of our results. Note however that the figures are included at the end of the manuscript.

{
We start with an illustration of the long-time behavior of the solution to \eqref{eq:SI-no-mut} when the initial mass of the fitness maximum is positive ($I_0(\{\alpha(x)=\alpha^*\})>0$) in Figure \ref{Fig1}. To help with the visual representation, in our example, $I_0$ is chosen absolutely continuous with respect to the Lebesgue measure, \textit{i.e.} carried by a $L^1(\mathbb R^2)$ function. We provide a plot of the fitness function which is proportional to $\alpha$ (top left sub-figure) and the initial data (top right). The fitness function attains its maximum on the union of a rectangle $[0.2,0.5] \times[-0.5,0.5]$ with a line segment $\{-0.3\} \times[-0.5,0.5]$ and the support of the initial data intersects this set with non-negligible intersection. We plot the time evolution of the converging function $S(t)$ (bottom left) and a snapshot of the distribution $I(t,\dd x)$ at $t=50$ (bottom right). We observe that the mass that was initially located outside of the fitness maximum has vanished. What remains is a distribution of mass in the initial rectangle of maximal fitness (according to Theorem \ref{thm:measures}, the precise distribution  can be computed). The distribution located at $\{x_1=-0.3\}$ is still positive, but is negligible with respect to the Lebesgue measure and does not contribute to the mass. In Figure \ref{Fig3} we present four snapshots of the distribution $I(t, \dd x)$ to monitor the time evolution of this distribution with the same initial distribution.

 Next, we consider the case  when the initial mass of the fitness maximum is equal to zero ($I_0(\{\alpha(x)=\alpha^*\})=0$) in Figure \ref{Fig4}.   Again, we provide a plot of the fitness function (top left) and of the initial data (top right). The fitness function attains its maximum on the line segment $\{0.35\} \times[-0.5,0.5]$.  The time evolution of $S(t)$ (bottom left) and a snapshot of the distribution $I(t,\dd x)$ at $t=100$ (bottom right) are also displayed. The distribution  $I(t,\dd x)$ at $t=100$ is already concentrated on the maximum of the fitness.
Figure \ref{Fig5} consists of four snapshots of the example presented in Figure \ref{Fig4}.
}

\subsection{Transient dynamics on local maxima: a numerical example}
In many biologically relevant situations it may be more usual to observe situations involving a fitness function with one global maximum and several (possibly many) local maxima, whose {values are} not exactly equal to the global maximum but very close. In such a situation, while the long-term distribution will be concentrated on the global maximum, one may observe a transient behavior in which the orbits {stay} close to the equilibrium of {the several global maxima situation (corresponding to Theorem \ref{THEO-eta})}, before it concentrates on the eventual distribution. We leave the analytical treatment of such a situation open for future studies, however, we present a numerical experiment in Figure \ref{Fig8} which shows such a transient behavior.

In this simulation, we took a fitness function presenting one global maximum at $x_2=+0.5$ and a local maximum at {$x_1=-0.5$}, whose value is close to the global maximum.
The precise definition of $\alpha(x)$ is
\begin{equation}\label{eq:example-alpha-transient}
	\alpha(x) =  0.95\times\mathbbm P_{[x_1-\delta, x_1+\delta]}(x)+ \mathbbm P_{[x_2-\delta, x_2+\delta]}(x) \text{ with } \delta=0.2.
\end{equation}
The function $I_0(x)$ is chosen as

{

	\begin{equation}\label{eq:example-I_0-transient}
		I_0(x) = \min\left(1, 4 \left(x-x_1\right)^2\right) \min\left(1, 4 \left(x-x_2\right)^2\right) \mathbbm{1}_{[-1, 1]}(x),    
	\end{equation}

}so that $\kappa_1=2$ and $\kappa_2=2$. Finally,
\begin{equation}\label{eq:example-gamma-transient}
	\gamma(x)=\frac{1}{1 + \mathbbm{P}_{[x_1-\delta, x_1+\delta]}(x) +  3 \mathbbm{P}_{[x_2-\delta, x_2+\delta]}(x) }
\end{equation}
so that $\gamma(x_1)=\frac{1}{2}$ and $\gamma(x_2)=\frac{1}{4}$. Summarizing, we have
\begin{equation*}
	\dfrac{N+\kappa_1}{2\gamma(x_1)}= 3<6 =\dfrac{N+\kappa_2}{2\gamma(x_2)}.
\end{equation*} 
If $\alpha(x)$ had two global maximum at the same level,   Theorem \ref{THEO-eta} would predict that the mass $I(t, \dd x)$  vanishes near $x_2$ and concentrates on $x_1$. Since the value of $\alpha(x_2)$ is slightly higher than the value of $\alpha(x_1)$, however, it is clear that the eventual distribution will be concentrated on $x_2$. We observe numerically (see Figure \ref{Fig8}) that the distribution first concentrates on $x_1$ on a transient time scale, before the dynamics on $x_2$ takes precedence. We refer to \textcite{BDDD2020} for a related model with mutations where these transient behaviors are analytically characterized.

Figure \ref{Fig4}   illustrates the case when the maximal fitness is negligible for the initial measure (Point ii) of Theorem \ref{thm:measures}).
We provide a plot of the fitness function (top left-hand side) and the initial data (top right-hand side). The fitness function attains its maximum on a rectangle with positive Lebesgue measure and the support of the initial data intersects this rectangle with non-negligible intersection.
We also plot the time evolution of $S(t)$ (bottom left-hand side) and a snapshot of the distribution $I(t, x)$ at $t=100$. We observe that the mass that was initially located outside of the fitness maximum has vanished (bottom right-hand side). What remains is a distribution of mass around the initial line of maximal fitness, which is negligible for the initial data; however the distribution  takes very high values.

{In Figure \ref{Fig7} we provide a precise example of this non-standard behavior. The function $\alpha(x)$ is chosen to have two maxima $x_1=-0.5 $ and $x_2=0.5$; the precise definition of $\alpha(x)$ is 
\begin{equation}\label{eq:example-alpha}
	\alpha(x) =  \mathbbm P_{[x_1-\delta, x_1+\delta]}(x)+ \mathbbm P_{[x_2-\delta, x_2+\delta]}(x),
\end{equation}
where 
\begin{equation*}
	\mathbbm{P}_{[a,b]}(x):=\max\left(1-\frac{(a+b-2x)^2}{(a-b)^2}, 0\right) 
\end{equation*}
is the downward parabolic function of height one and support $[a,b]$ and $\delta=0.2$.
The function $\alpha(x)$ has the exact same local behavior in the neighborhood of $x_1$ and $x_2$. The function $I_0(x)$ is chosen as
\begin{equation}\label{eq:example-I_0}
	I_0(x) = \min\left(1, 1024 \left(x-x_1\right)^8\right) \min\left(1, 4 \left(x-x_2\right)^2\right) \mathbbm{1}_{[-1, 1]}(x),    
\end{equation}
so that $\kappa_1=8$ and $\kappa_2=2$. Finally we take 
\begin{equation}\label{eq:example-gamma}
	\gamma(x)=\frac{1}{1 + \mathbbm{P}_{[x_1-\delta, x_1+\delta]}(x) +  3 \mathbbm{P}_{[x_2-\delta, x_2+\delta]}(x) } 
\end{equation}
so that $\gamma(x_1)=\frac{1}{2}$ and $\gamma(x_2)=\frac{1}{4}$. Summarizing, we have 
\begin{equation*}
	\dfrac{N+\kappa_1}{2\gamma(x_1)}= 9>6 =\dfrac{N+\kappa_2}{2\gamma(x_2)}, 
\end{equation*}
so that Theorem \ref{THEO-eta} predicts that the mass $I(t, \dd x)$ will  vanish near $\{x_1\}=\{\alpha(x)=\alpha^*\}\cap\{ \gamma(x)=\gamma^*\}$ and concentrate on $x_2$.
}

\subsection{Numerical illustration: Oscillations}\label{sec:numerical-oscillations}

In Figure \ref{fig:osc} we present a numerical simulation that illustrates the example provided in section \ref{sec:no-convergence}. The parameter $L$ is set to 1. A supplementary movie is also available (\texttt{spiraling.avi}).

\subsection{Comments}

We have studied the asymptotic behavior of a simple epidemiological model whose originality is that the population is structured by a continuous genotypic variable. Thus the population is divided into a compartment of susceptible and infected by a certain strain of pathogen. The duration of the infectious period $(1/\gamma)$ and the basic reproduction number ($\alpha$, up to a multiplicative constant) of the disease depend on the trait considered. 

In this work, the pathogen population does not mutate and therefore if a trait is absent from the initial pathogen population, it cannot appear in the population afterwards. 

Assuming that the basic reproduction number has a maximum strictly greater than one for at least one phenotypic trait of the initial population, we have shown the convergence of the solution of this model towards an endemic equilibrium. Firstly, in the case where the maximum value of the basic reproduction number is reached on a continuum of phenotypic traits of the initial population, we can completely describe the asymptotic number of susceptible as well as the distribution of the infected population with respect to the different variants.

Secondly, in the case where the maximum of $\mathcal{R}_0$ is reached on a set of zero measure, we have shown the persistence of the infectious population and its asymptotic concentration on a subset of the set of the traits maximizing both the $\mathcal{R}_0$ and the additional mortality rate due to the pathogen (or in other words minimizing the infectious period). 

To go further in the analysis, we then considered the case of a finite number of traits maximizing the $ \mathcal{R}_0$. In this case, we were able to describe more precisely the set of traits around which the population of infected (and therefore of pathogens) is concentrated, as well as the profile of the asymptotic distribution of infected around these traits. In particular, we observed that even if there are no infected individuals with a trait maximizing the $ \mathcal{R}_0$ initially, the population can concentrate around this trait (but not on this trait since there is no mutation in this model). The selection of traits around which the population concentrates thus depends not only on the value of $ \mathcal{R}_0$ and the additional mortality rate, but also on the initial population distribution around the trait. A non-standard behavior may thus appear where the selected strain no longer maximizes $\gamma$, the virulence. The question arises whether such a configuration can be observed in vivo.

\section{Measure-valued solutions and proof of Theorem \ref{thm:measures}}
\label{sec:measure}

In this section we derive general properties of the solution of \eqref{eq:SI-no-mut} equipped with the given and fixed initial data $S(0)=S_0\in [0,\infty)$ and {$I_0(\dd x)\in \mathcal M_+(X)$}. 
Recall that $\alpha^*$ and $\mathcal R_0(I_0)$ are both defined in \eqref{def-alpha*}. Next for $\varepsilon>0$ {we recall that $L_\varepsilon(I_0)$ is} the following superlevel set {(defined by \eqref{eq:Leps} in Assumption \ref{as:params-nomut}):}
\begin{equation*}
	L_\varepsilon(I_0)=\left\{x\in\supp I_0\, :\, \alpha(x)\geq \alpha^*-\varepsilon\right\}=\bigcup_{\alpha^*-\varepsilon\leq y\leq\alpha^*}\{\alpha(x)=y\}.
\end{equation*}
{Recall also that the existence and uniqueness of a solution $\big(S(t), I(t, \dd x)\big)\in \mathbb R\times {\mathcal{M}(X)}$ corresponding to $(S_0, I_0)$ in the  Banach space $ \mathbb R\times {\mathcal{M}(X)}$ (where {$\mathcal{M}(X)$} is equipped with the norm $\Vert \cdot\Vert_{TV}$) follow directly from the  Cauchy-Lipschitz Theorem.}

The following lemma holds true.
\begin{lemma}\label{lem:bounds-nomut}
	Let Assumption \ref{as:params-nomut} {hold.} 
	Denote {$\big(S(t), I(t, \dd x)\big)\in \mathbb R\times {\mathcal{M}(X)}$ be the corresponding solution of the ordinary differential equation \eqref{eq:SI-no-mut}}. Then $\big(S(t), I(t, \dd x)\big)$ is defined for all $t\geq 0$ and
	\begin{gather*}
		0<\frac{\min(\theta, \gamma_*)}{\theta\Lambda\min(\theta, \gamma_*)+\alpha^*\gamma^*}\leq \liminf_{t\to+\infty}S(t)\leq \limsup_{t\to+\infty}S(t)\leq \frac{\Lambda}{\theta}<+\infty,\\
		\limsup_{t\to+\infty}\int_{{X}}I(t,\dd x) \leq \frac{\Lambda}{\min(\theta, \gamma_*)}<+\infty,
	\end{gather*}
	where $\gamma_*:=\inf_{x\in\supp I_0}\gamma(x)$, $\gamma^*:=\sup_{x\in\supp I_0}\gamma(x)$ and $\alpha^*:=\sup_{x\in\supp I_0}\alpha(x)$.
\end{lemma}
\begin{proof}
	We remark that 
	\begin{equation*}
		\frac{\dd}{\dd t}\left(S(t)+\int_{{X}}I(t,\dd x)\right)\leq \Lambda - \theta S(t) - \gamma_*\int_{{X}}I(t,\dd x),
	\end{equation*}
	therefore 
	\begin{equation*}
		S(t)+\int_{{X}}I(t,\dd x)\leq \frac{\Lambda}{\min(\theta, \gamma_*)} + \left(S_0+\int_{{X}}I_0(\dd x)-\frac{\Lambda}{\min(\theta, \gamma_*)}\right)e^{-\min(\theta, \gamma_0)t}. 
	\end{equation*}
	In particular $I(t, \dd x)$ is uniformly bounded in {$\mathcal M(X)$} and therefore we have the global existence of the solution as well as
	\begin{equation*}
		\limsup_{t\to+\infty}\int_{{X}} I(t, \dd x)\leq\frac{\Lambda}{\min(\theta, \gamma_*)} \text{ and } \limsup_{t\to+\infty} S(t) \leq\frac{\Lambda}{\min(\theta, \gamma_*)} .
	\end{equation*}
	Next we return to the $S$-component of equation \eqref{eq:SI-no-mut} and let $\varepsilon>0$ be given. We have, for $t_0$ sufficiently large and $t\geq t_0$,  
	\begin{equation*} 
		S_t = \Lambda -\left( \theta + \int_{{X}}\alpha(x)\gamma(x)I(t, \dd x)\right) S(t)\geq \Lambda - \left(\theta+\alpha^*\gamma^*\frac{\Lambda}{\min(\theta, \gamma_*)}+\varepsilon\right)S(t),
	\end{equation*}
	therefore 
	\begin{equation*}
		S(t)\geq e^{-\left(\theta+\frac{\Lambda\alpha^*\gamma^*}{\min(\theta, \gamma_*)}+\varepsilon\right)(t-t_0)}S(t_0)+\frac{\Lambda\min(\theta, \gamma_*)}{(\theta+\varepsilon)\min(\theta, \gamma_*)+\Lambda\alpha^*\gamma^*}\left(1-e^{-\left(\theta+\frac{\Lambda\alpha^*\gamma^*}{\min(\theta, \gamma_*)}+\varepsilon\right)(t-t_0)}\right),
	\end{equation*}
	so that finally by letting $t\to+\infty$ we get
	\begin{equation*}
		\liminf_{t\to+\infty} S(t) \geq \frac{\min(\theta, \gamma_*)\Lambda}{(\theta+\varepsilon)\min(\theta, \gamma_*)+\Lambda\alpha^*\gamma^*}.
	\end{equation*}
	Since $\varepsilon>0$ is arbitrary we have shown
	\begin{equation*}
		\liminf_{t\to+\infty} S(t) \geq \frac{\min(\theta, \gamma_*)\Lambda}{\theta\min(\theta, \gamma_*)+\Lambda\alpha^*\gamma^*}.
	\end{equation*}
	The Lemma is proved.
\end{proof}

\begin{lemma}\label{lem:limsup-no-mut}
	Let Assumption \ref{as:params-nomut} {hold}. 
	Let $\big(S(t), I(t, \dd x)\big)$ be the corresponding solution of \eqref{eq:SI-no-mut}.
	Then
	\begin{equation*}
		\limsup_{T\to+\infty}\frac{1}{T}\int_0^T S(t)\dd t \leq \dfrac{1}{\alpha^*},
	\end{equation*}
	where $\alpha^*$ is given in \eqref{def-alpha*}.
\end{lemma}
\begin{proof}
	Let us remark that the second component of \eqref{eq:SI-no-mut} can be written as
	\begin{align}
		I(t, \dd x) &= I_0(\dd x) e^{\gamma(x)\left(\alpha(x)\int_0^tS(s)\dd s-t\right)},\notag \\
		&=I_0(\dd x) \exp\left(\gamma(x)\int_0^t S(s)\dd s\left[\alpha(x)-\frac{t}{\int_0^tS(s)\dd s}\right]\right)\label{eq:I-nomut}.
	\end{align}
	Assume by contradiction that the conclusion of the Lemma does not hold, {\it i.e.} there exists $\varepsilon>0$ and a sequence $T_n\to +\infty$ such that 
	\begin{equation*}
		\frac{1}{T_n}\int_0^{T_n}S(t)\dd t\geq  \dfrac{1}{\alpha^*}+\varepsilon.
	\end{equation*}
	Then
	\begin{equation*}
		\frac{T_n}{\int_0^{T_n}S(t)\dd t}\leq  \frac{1}{\dfrac{1}{\alpha^*}+\varepsilon} \leq \alpha^*-\varepsilon',
	\end{equation*}
	where $\varepsilon'=\left(\alpha^*\right)^2\varepsilon+o(\varepsilon)$. Since the map $x\mapsto \alpha(x)$ is continuous, the set  $ L_\nu(I_0) = \{x\in\supp I_0\,:\, \alpha(x)\geq \alpha^*-\nu\} $
	has positive mass with respect to the measure $I_0(\dd x)$ for all  $\nu>0$, {\it i.e.} $\int_{L_\nu(I_0)} I_0(\dd x) >0$. This is true, in particular, for $\nu=\frac{\varepsilon'}{2}$, therefore 
	\begin{align*}
		\int_{L_{\varepsilon'/2}(I_0)}I(T_n,\dd x)&=\int_{L_{\varepsilon'/2}(I_0)}\exp\left(\gamma(x)\int_0^{T_n} S(s)\dd s\left[\alpha(x)-\frac{T_n}{\int_0^{T_n}S(s)\dd s}\right]\right)I_0(\dd x)\\
		&\geq\int_{L_{\varepsilon'/2}(I_0)}\exp\left(\gamma_*\int_0^{T_n} S(s)\dd s\cdot\frac{\varepsilon'}{2}\right)I_0(\dd x)\\ 
		&=\int_{L_{\varepsilon'/2}(I_0)}I_0(\dd x)\exp\left(\frac{\varepsilon'\gamma_*}{2}\int_0^{T_n} S(s)\dd s\right),
	\end{align*}
	where $\gamma_*=\inf_{x\in\supp I_0}\gamma(x)$.
	Since $\int_{L_{\varepsilon'/2}(I_0)}I_0(\dd x)>0$ and $\int_0^{T_n}S(t)\dd t \to +\infty$ when $n\to+\infty$, we have therefore 
	\begin{equation*}
		\limsup_{t\to+\infty}\int_{{X}}I(t,\dd x)\geq \limsup_{n\to+\infty}\int_{L_{\varepsilon'/2}(I_0)}I(T_n,\dd x)=+\infty,
	\end{equation*}
	which is a contradiction since $I(t,\dd x)$ is bounded in $\mathcal M({X})$ by Lemma \ref{lem:bounds-nomut}.
	This completes the proof of the Lemma.
\end{proof}

{An important tool in later proofs is that the mass of $I(t, \dd x)$ vanishes on any set sufficiently far away from $I_0$ when the Cesàro mean $\overline{S}(t)=\frac{1}{t}\int_{0}^t S(s)\dd s$ of $S$ is sufficiently close to $\alpha^*$, which we prove now.
\begin{lemma}\label{lem:Ivanishes}
	Let Assumption \ref{as:params-nomut} hold and $\big(S(t), I(t, \dd x)\big)$ be the corresponding solution of \eqref{eq:SI-no-mut}. Let $\{t_\tau\}$ be a net  $t_\tau\to\infty$ and $\varepsilon>0$ be such that  we have eventually
	\begin{equation}
		\overline{S}(t_\tau)=\frac{1}{t_\tau}\int_0^{t_\tau} S(s)\dd s \leq \frac{1}{\alpha^*-\varepsilon} \text{ for all }\tau.
	\end{equation}
	Then for any $p>1$ we have
	\begin{equation}
		\int_{\{\alpha(x)\leq \alpha^*-p\varepsilon\}}I(t_\tau, \dd x) \xrightarrow[t_\tau\to+\infty]{}0.
	\end{equation}
\end{lemma}
\begin{proof}
	Indeed, we can write
	\begin{align*}
		\int_{\{\alpha(x)\leq \alpha^*-p\varepsilon\}}I(t_\tau, \dd x) &= \int_{\{\alpha(x)\leq \alpha^*-p\varepsilon\}}\exp\left[\gamma(x)t_\tau\overline{S}(t_\tau) \left(\alpha(x)-\frac{1}{\overline{S}(t_\tau)}\right)\right] I_0(\dd x)\\
		&\leq \int_{\{\alpha(x)\leq \alpha^*-p\varepsilon\}}\exp\left[\gamma(x)t_\tau\overline{S}(t_\tau) \left(\alpha^*-p\varepsilon-\frac{1}{\overline{S}(t_\tau)}\right)\right] I_0(\dd x) \\
		&\leq \int_{\{\alpha(x)\leq \alpha^*-p\varepsilon\}}\exp\left[\gamma(x)t_\tau\overline{S}(t_\tau) (1-p)\varepsilon\right] I_0(\dd x).
	\end{align*}
	Since $p>1$ and $\liminf \overline{S}(t_\tau)>0$ by Lemma \ref{lem:bounds-nomut}, the argument of the exponential converges to $-\infty$ as $t_\tau\to+\infty$ therefore
	\begin{equation*}
		\int_{\{\alpha(x)\leq \alpha^*-p\varepsilon\}}I(t_\tau, \dd x)\xrightarrow[t_\tau\to+\infty]{}0.
	\end{equation*}
	The Lemma is proved.
\end{proof}
}
The following weak persistence property holds.
\begin{lemma}\label{lem:weak-persistent-nomut}
	Let Assumption \ref{as:params-nomut} hold and suppose that $\mathcal{R}_0(I_0)>1$. 
	Let $\big(S(t), I(t, \dd x)\big)$ be the corresponding solution of \eqref{eq:SI-no-mut}.
	Then 
	\begin{equation*}
		\limsup_{t\to+\infty} \int_{{X}} I(t,\dd x) \ge \frac{\theta}{\alpha^*\gamma^*} \big( \mathcal{R}_0(I_0)-1 \big) >0, 
	\end{equation*}
	where $\gamma^*:=\sup_{x\in\supp I_0} \gamma(x)$.
\end{lemma}
\begin{proof} 
	Assume by contradiction that  for  $t_0$ sufficiently large we have
	\begin{equation*}
		\int_{{X}} I(t,\dd x)  \le  \eta' < \eta=:\frac{\theta}{\alpha^*\gamma^*} \big({\mathcal R}_0(I_0)-1 \big) \text{ for all }t\geq t_0,
	\end{equation*}
	with $\eta'>0$.

	As a consequence of Lemma \ref{lem:limsup-no-mut} we have 
	\begin{equation}\label{eq:infS1}
		\liminf_{t\to+\infty} S(t)\leq \limsup_{T\to+\infty}\frac{1}{T}\int_0^T S(t)\dd t \leq \frac{1}{\alpha^*}.
	\end{equation}

	Let $\underline{S} := \liminf_{t\to+\infty} S(t)$.
	Let  $(t_n)_{n\geq 0}$ be a sequence that tends to $\infty$ as $n\to\infty$ and such that $\lim_{n \to +\infty} S'(t_n)=0$ and $\lim_{ n \to +\infty} S(t_n) = \underline{S}$. As $\int_{{X}} I(t_n,\dd x)  \leq \eta'$ for $n$ large enough we deduce from the equality
	\begin{equation*}
		S'(t_n)=\Lambda - \theta S(t_n) - S(t_n) \int_{{X}} \alpha(x)\gamma(x) I(t_n,\dd x),
	\end{equation*}
	that
	\begin{equation*} 
		0 \ge \Lambda - \theta \underline{S} - \underline{S} \alpha^*\gamma^* \eta'
	\end{equation*}
	so that
	\begin{equation*}
		\underline{S} \ge  \frac{\Lambda}{\theta+ \alpha^*\gamma^*\eta'} > \frac{\Lambda}{\theta+ \alpha^*\gamma^*\eta} 
	\end{equation*}
	and by definition of   $\eta$
	\begin{equation*}
		\underline{S} > \frac{\Lambda}{\theta {\mathcal R}_0} = \frac{1}{\alpha^*},
	\end{equation*}
	which contradicts \eqref{eq:infS1}.
\end{proof} 

Let us remind that $\mathcal M_+({X})$, equipped with the Kantorovitch-Rubinstein metric $d_0$ defined in  \eqref{eq:kantorub}, is a complete metric space. 
\begin{lemma}[Compactness of the orbit and concentration]\label{lem:compactness-nomut}
	Let Assumption \ref{as:params-nomut} hold and 
	Let $\big(S(t), I(t, \dd x)\big)$ be the corresponding solution of \eqref{eq:SI-no-mut}.
	Then, the closure of the orbit of $(S_0, I_0)$,
	\begin{equation*}
		\overline{\mathcal O}(S_0,I_0):=\left\{(s, \mu)\in\mathbb{R}\times \mathcal M_+({X})\,: \, \text{there exists } (t_n) \text{ such that } d_0(I(t_n, \dd x),\mu)\xrightarrow[n\to +\infty]{} 0 \text{ and } |S(t_n)-s|\xrightarrow[n\to+\infty]{}0\right\},
	\end{equation*}
	is compact.

	Moreover $\mathcal R_0(I_0)>1$ and $t_n\to+\infty$ is an arbitrary sequence along which 
	\begin{equation}\label{eq:170523-1}
		\liminf_{n\to+\infty}\int_{X}I(t, \dd x)>0, 
	\end{equation}
	then one can extract from $(t_n)$ a subsequence $(t_{n_k})$ such that the shifted orbits 
	\begin{equation*}
		t\mapsto \big(S(t+t_{n_k}), I(t+t_{n_k}, \dd x)\big)
	\end{equation*}
	converge weak-$\ast$ pointwise to a complete orbit $\big(S^\infty, I^\infty(t, \dd x)\big)$ satisfying the following properties:
	\begin{equation}\label{eq:170523-2}
		\int_{X}I^\infty(t, \dd x)>0 \text{ for all } t\in\mathbb{R} ,
	\end{equation}
	and 
	\begin{equation}\label{eq:170523-3}
		\int_{\alpha(x)<\alpha^*}I^\infty(t, \dd x) =0.
	\end{equation}
	Finally the convergence $S(t+t_{n_k})\to S^\infty(t)$ holds locally uniformly in $C^1(\mathbb{R})$.
\end{lemma}
\begin{proof}
	First of all let us remark that 
	\begin{equation*}
		I(t,\dd x) = e^{\left(\int_0^t S(s)\dd s\alpha(x)-t\right)\gamma(x)}I_0(\dd x), 
	\end{equation*}
	and therefore the orbit $t\mapsto I(t, \dd x)$ is continuous for the metric $d_0$. 

	By Lemma \ref{lem:limsup-no-mut} we have 
	\begin{equation*}
		\limsup_{T\to+\infty}\frac{1}{T}\int_0^TS(s)\dd s \leq \frac{1}{\alpha^*},
	\end{equation*}
	where $\alpha^*$ defined in \eqref{def-alpha*}. We prove that the family $\{I(t, \dd x), t\geq 0\}$ is uniformly tight. 
	Let $\varepsilon>0$ be sufficiently small, so that the set $L_\varepsilon(I_0)$ is {compact}. {By Lemma \ref{lem:Ivanishes} and Lemma \ref{lem:limsup-no-mut} we have 
	\begin{equation*}
		\limsup_{t\to+\infty} \int_{X\backslash L_\varepsilon(I_0)}I(t, \dd x) = 0.
	\end{equation*}
	Thus given any threshold $\nu>0$, there is $T_\nu$ such that $I(t, X\backslash L_\varepsilon(I_0))\leq \nu$ for all $t\geq T_\nu$, and since $I_0$ is Radon there exists a compact set $K_\nu\subset X$ such that $I_0(X\backslash K)\leq \nu e^{-\gamma^\infty T_\nu\alpha^\infty \sup_{0\leq s\leq T_\nu} \overline{S}(s)}$ so that
	\begin{equation*}
		\int_{X\backslash K} I(t, \dd x) \leq \int_{X\backslash K} I_0(\dd x) e^{\gamma^\infty T_\nu(\alpha^\infty \sup_{0\leq s\leq T_\nu} \overline{S}(s) - 1)} \leq \nu. 
	\end{equation*}
	Thus for all $t\geq 0$ we have $I\big(t, X\backslash\big(K_\nu\cup L_{\varepsilon}(I_0)\big)\big)$. The set  $\{I(t, \dd x)\,: t\geq 0\}$ is uniformly tight. Moreover it is bounded in the total variation norm (see Lemma \ref{lem:bounds-nomut}) in the complete separable metric space $X$, therefore precompact for the weak topology by Prokhorov's Theorem \parencite[Theorem 8.6.2, Vol. II p. 202]{Bog-07}. 
	}

	Next let $(t_n)$ be an arbitrary sequence along which \eqref{eq:170523-1} holds. Thanks to the compactness of the orbit, we extract from $(S(t_n), I(t_n, \dd x))$ a subsequence still denoted $t_n$, such that $S(t_n)\to S^\infty$ and $I(t_n, \dd x) \to I^\infty(\dd x)$ weakly. Clearly $S'(t+t_n) $ is bounded independently on $n$; differentiating the first line in \eqref{eq:SI-no-mut-a} we see that $S''(t+t_n)$ is also bounded independently on $n$ when $t$ is in an arbitrary compact set. Thus, up to a  diagonal extraction, we may assume that both $S(t+t_n)$ and $S'(t+t_n)$ converge locally uniformly in $t$ to a limit $S^\infty(t)$ and $(S^\infty)'(t)$. This proves the final statement of the Lemma. We can now take the weak-$\ast$ limit in the formula 
	\begin{equation*}
		I(t+t_n, \dd x) = I(t_n, \dd x) \exp\left[\gamma(x)t\left(\alpha(x)\frac{1}{t} \int_{0}^{t} S(t_n+s)\dd s-1\right)\right]
	\end{equation*}
	which shows that, for fixed but arbitrary $t\in\mathbb{R}$, we have that $I(t_n+t, \dd x)$ converges weakly to 
	\begin{equation*}
		I^\infty(t, \dd x):=I^\infty(\dd x) \exp\left[\gamma(x)t\left(\alpha(x)\frac{1}{t} \int_{0}^{t} S^\infty(s)\dd s-1\right)\right].
	\end{equation*}
	In particular $(S^\infty(t), I^\infty(t, \dd x))$ is a complete orbit of the equation \eqref{eq:SI-no-mut}. Since $1\in BC(X)$ we have that 
	\begin{equation*}
		\int_{X}I^\infty(\dd x) = \int_{X}1 I^\infty(\dd x) =\lim_{n\to+\infty}\int_{X}1 I(t+t_n, \dd x) >0, 
	\end{equation*}
	which shows \eqref{eq:170523-2}. 

	Next it follows from Lemma \ref{lem:limsup-no-mut} that $\limsup_{n\to+\infty}\frac{1}{t_n}\int_0^{t_n} S(s)\dd s \leq \frac{1}{\alpha^*}$ and thus by Lemma \ref{lem:Ivanishes} that 
	\begin{equation*}
		\int_{\alpha(x)\leq \alpha^*-\frac{1}{k}}I^\infty(\dd x) = \lim_{n\to+\infty} \int_{\alpha(x)\leq \alpha^*-\frac{1}{k}}I(t_n, \dd x)  = 0, \text{ for all }k\in\mathbb{N}.
	\end{equation*}
	Thus by Fatou's Lemma
	\begin{equation*}
		\int_{\alpha(x)<\alpha^*} I^\infty(\dd x) \leq \lim_{k\to+\infty}\int_{\alpha(x)\leq \alpha^*-\frac{1}{k}}I^\infty(\dd x)=0.
	\end{equation*}
	This proves \eqref{eq:170523-3} and completes the proof of Lemma \ref{lem:compactness-nomut}.
\end{proof}

Next we show the weak uniform persistence property if $\mathcal R_0(I_0)>1$. 
\begin{lemma}[Uniform persistence]\label{lem:unifpers-nomut}
	Let Assumption \ref{as:params-nomut} hold and suppose that $\mathcal{R}_0>1$.
	Let $\big(S(t), I(t, \dd x)\big)$ be the corresponding solution of \eqref{eq:SI-no-mut}. Then 
	\begin{equation*}
		\liminf_{t\to+\infty} \int_{X}I(t, \dd x) >0.
	\end{equation*}
\end{lemma}
\begin{proof}
	We adapt here the argument of \textcite[Proposition 3.2]{Mag-Zha-05} to our context. Suppose by contradiction that there exists a sequence $t_n\to+\infty$ such that 
	\begin{equation*}
		\lim_{n\to+\infty} \int_{X}I(t_n, \dd x)=0. 
	\end{equation*}
	By Lemma \ref{lem:bounds-nomut} we know that $S(t_n)\geq \underline{S}>0$ for all $n\in\mathbb{N}$. By Lemma \ref{lem:weak-persistent-nomut}, for each $n\in\mathbb{N}$ sufficiently large, there exists $s_<t_n$ such that 
	\begin{equation}\label{eq:230518-2}
		\int_{X} I(s_n, \dd x)=\frac{\theta}{2\alpha^*\gamma^*}\big(\mathcal{R}_0(I_0)-1\big) \text{ and } \int_{X} I(t, \dd x)\leq \frac{\theta}{2\alpha^*\gamma^*}\big(\mathcal{R}_0(I_0)-1\big) \text{ for all } s_n\leq t\leq t_n.
	\end{equation}
	Up to replacing $t_n$ by a subsequence, we will assume without loss of generality that $t_{n-1}< s_n <t_{n}$ for all $n\in\mathbb{N}$. Thanks to Lemma \ref{lem:compactness-nomut} and up to a further extraction, the shifted orbits $\big(S(t+s_n), I(t+s_n, \dd x)\big)$ converge to a complete orbit $\big(S^\infty(t), I^\infty(t, \dd x)\big)$ which satisfies $0<\underline{S} \leq S^\infty(t)\leq \frac{\Lambda}{\theta}$ and $\int_{X}I^\infty(t, \dd x)>0$ for all $t\in\mathbb{R}$. Moreover $I^\infty(t, \dd x)$ is concentrated on the set $\{x\,:\,\alpha(x)=\alpha^*\}$, and in particular $\mathcal{R}_0(I^\infty(0, \dd x))=\mathcal{R}_0(I_0)>1$. By Lemma \ref{lem:weak-persistent-nomut} we have therefore
	\begin{equation}\label{eq:230518-1}
		\limsup_{t\to+\infty} \int_{X} I^{\infty}(t, \dd x) \geq \frac{\theta}{\alpha^*\gamma^*}\big(\mathcal{R}_0(I_0)-1\big).
	\end{equation}
	Next we investigate the time $T_n:= t_n-s_n$. Up to extracting a subsequence, there are two options. 
	\begin{itemize}
		\item $T_n\to T$ as $n\to+\infty$. In that case  we have $\int_{X}I^\infty(T, \dd x)=0 $ so $I^\infty(T, \dd x)\equiv 0$, and by the uniqueness of the solution to \eqref{eq:SI-no-mut} we have $I^\infty(t, \dd x)\equiv 0$ for all $t\geq T$. This contradicts \eqref{eq:230518-1}.
		\item $T_n\to+\infty $ as $n\to+\infty$. In that case, recalling \eqref{eq:230518-2}, we have $\int_{X}I^\infty(t, \dd x)\leq \frac{\theta}{2\alpha^*\gamma^*}\big(\mathcal{R}_0(I_0)-1\big)$ for all $t\geq 0$, and this also contradicts \eqref{eq:230518-1}.
	\end{itemize}
	This completes the proof of Lemma \ref{lem:unifpers-nomut}.
\end{proof}

\begin{lemma}\label{lem:liminf-no-mut}
	Let Assumption \ref{as:params-nomut} {hold.} 
	Let $\big(S(t), I(t, \dd x)\big)$ be the corresponding solution of \eqref{eq:SI-no-mut}.
	Assume that {$\mathcal R_0(I_0)>1 $}. 
	Then
	\begin{equation*}
		\liminf_{T\to+\infty}\frac{1}{T}\int_0^T S(t)\dd t \geq \dfrac{1}{\alpha^*}, 
	\end{equation*}
	with
	$ \alpha^*$ given in \eqref{def-alpha*}.
\end{lemma}
\begin{proof}
	Assume by contradiction that the conclusion of the Lemma does not hold, {\it i.e.} there exists $\varepsilon>0$ and a sequence $T_n\to +\infty$ such that 
	\begin{equation*}
		\frac{1}{T_n}\int_0^{T_n}S(t)\dd t\leq  \dfrac{1}{\alpha^*}-\varepsilon.
	\end{equation*}
	Then
	\begin{equation*}
		\frac{T_n}{\int_0^{T_n}S(t)\dd t}\geq  \frac{1}{\dfrac{1}{\displaystyle\alpha^*}-\varepsilon} \geq \alpha^*+\varepsilon',
	\end{equation*}
	where $\varepsilon'=\left(\alpha^*\right)^2\varepsilon+o(\varepsilon)$, {provided $\varepsilon$ is sufficiently small.}  Therefore 
	{by Lemma \ref{lem:Ivanishes} we have} 
	\begin{equation*}
		\liminf_{t\to+\infty}\int_{{X}}I(t,\dd x)\leq \liminf_{n\to+\infty}\int_{{X}}I(T_n,\dd x)=0,
	\end{equation*}
	which is in contradiction with Lemma \ref{lem:unifpers-nomut}.
	This proves the Lemma.
\end{proof}
\begin{remark}
	By combining Lemma \ref{lem:limsup-no-mut} and Lemma \ref{lem:liminf-no-mut} we obtain that $\frac{1}{T}\int_0^TS(t)\dd t$ admits a limit when $T\to+\infty$ and
	\begin{equation*}
		\lim_{T\to+\infty}\frac{1}{T}\int_0^T S(t)\dd t = \frac{1}{\alpha^*}.
	\end{equation*}
\end{remark}
\begin{lemma}\label{lem:weakconc-nomut}
	Let Assumption \ref{as:params-nomut} {hold.} 
	Let $\big(S(t), I(t, \dd x)\big)$ be the corresponding solution of \eqref{eq:SI-no-mut}.
	Then one has
	\begin{equation*}
		d_0\left(I(t, \dd x), \mathcal M_+(\{\alpha(x)=\alpha^*\})\right)\xrightarrow[t\to+\infty]{} 0.
	\end{equation*}
\end{lemma}
\begin{proof}
	Let $\varepsilon>0$ be as in the statement of Lemma \ref{lem:weakconc-nomut}. By Lemma \ref{lem:limsup-no-mut}, there exists $T\geq 0$ such that for all $t\geq T$ we have
	\begin{equation*}
		\dfrac{t}{\int_0^tS(s)\dd s}\geq \alpha^* - \frac{\varepsilon}{2}.
	\end{equation*}
	Hence {Lemma \ref{lem:Ivanishes} implies that} 
	\begin{align*}
		\int_{{X}\backslash L_\varepsilon(I_0)}I(t, \dd x)  \xrightarrow[t\to+\infty]{}0.
	\end{align*}
	In particular, if $I(t, \dd x)|_{L_\varepsilon(I_0)}$ denotes the restriction of $I(t, \dd x)$ to $L_\varepsilon(I_0)$, we have $\Vert  I(t, \dd x) - {I(t, \dd x)}|_{L_\varepsilon(I_0)}\Vert_{TV}\xrightarrow[t\to+\infty]{}0$
	and hence
	\begin{equation*}
		d_0\big(I(t, \dd x), I(t, \dd x)|_{L_\varepsilon(I_0)}\big)\leq d_{TV}\big(I(t, \dd x), I(t, \dd x)|_{L_\varepsilon(I_0)}\big)
		\xrightarrow[t\to+\infty]{} 0.
	\end{equation*} 
	Here $\varepsilon>0$ can be chosen arbitrarily small. By Lemma \ref{lem:bounds-nomut} we know moreover that 
	\begin{equation*}
		\limsup_{t\to+\infty}\int_{X}I(t, \dd x) \leq \frac{\Lambda}{\min(\theta, \gamma_*)}, 
	\end{equation*}
	so that for $t$ sufficiently large, we have 
	\begin{equation*}
		\int_{X}I(t, \dd x) \leq 2\frac{\Lambda}{\min(\theta, \gamma_*)}.
	\end{equation*}
	Finally by using Proposition \ref{prop:est-d0-supp} {(proved in the Appendix)}, we have 
	\begin{align*}
		d_0\big(I(t, \dd x), \mathcal M_+\left(\{\alpha(x)=\alpha^*\}\right)\big)&\leq d_0\big(I(t, \dd x), I(t, \dd x)|_{L_\varepsilon(I_0)}\big)+d_0\big(I(t, \dd x)|_{L_\varepsilon(I_0)}, \mathcal M_+\left(\{\alpha(x)=\alpha^*\}\right)\big)\\
		&\leq d_0\big(I(t, \dd x), I(t, \dd x)|_{L_\varepsilon(I_0)}\big) + 2\frac{\Lambda}{\min(\theta, \gamma_*)} \sup_{x\in L_\varepsilon(I_0)}d\big(x, \{\alpha(x)=\alpha^*\}\big). 
	\end{align*}
	Since 
	\begin{equation*}
		\sup_{x\in L_\varepsilon(I_0)}d\big(x, \{\alpha(x)=\alpha^*\}\big) \xrightarrow[\varepsilon\to 0]{} 0, 
	\end{equation*}
	the Kantorovitch-Rubinstein distance between $I(t, \dd x)$ and $\mathcal M\big(\{\alpha(x)=\alpha^*\}\big)$ can indeed be made arbitrarily small as $t\to +\infty$. This proves the Lemma.
\end{proof}

\begin{lemma}\label{lem:conv-alpha-const}
	Let Assumption \ref{as:params-nomut} hold and suppose moreover that $\alpha(x)\equiv \alpha^*$ for all $x\in\supp I_0$ and that $\mathcal{R}_0(I_0)>1$. 
	Let $\big(S(t), I(t, \dd x)\big)$ be the corresponding solution of \eqref{eq:SI-no-mut} and let $S^*:=\frac{1}{\alpha^*}$ and  $i^*(x):=e^{\tau \gamma(x)}$ where $\tau\in\mathbb{R}$ is the unique solution of the equation 
	\begin{equation}\label{eq:id-Iinf}
		\int_{{X}} \gamma(x)e^{\tau \gamma(x)}I_0(\dd x) = \frac{\theta}{\alpha^*}\left(\mathcal R_0-1\right).
	\end{equation}
	Define $g(x):=x-\ln(x)$ and let 
	\begin{equation*}
		D(V):=\left\{\big(S, i(x)\big) \in \mathbb{R}\times L^\infty_{i^*}(I_0)\,:\, S>0 \text{ and }\overset{I_0(\dd x)}{\essinf}\dfrac{i(x)}{i^*(x)}>0\right\},
	\end{equation*}
	where $L^\infty_{i^*}(I_0)$ is the weighted $L^\infty$ space equipped with the norm $\Vert i\Vert_{L^{\infty}_{i^*}(I_0)}:=\Vert \frac{i}{i^*}\Vert_{L^\infty(I_0)}$. Then $D(V)$ is open in $\mathbb{R}\times L^\infty_{i^*}(I_0)$ and the functional 
	\begin{equation*}
		V(S, i):= S^*g\left(\frac{S}{S^*}\right)+\int_{X}i^*(x)g\left(\dfrac{i(x)}{i^*(x)}\right) I_0(\dd x)
	\end{equation*}
	is well-defined and continuous on $D(V)$. Moreover if $i(t, x)$ is the Radon-Nikodym derivative of $I(t, \dd x)$ with respect to $I_0$ (in other words, $I(t, \dd x)=i(t, x)I_0(\dd x)$), then $t\mapsto V(S(t), i(t, x))$ is of class $C^1$ and we have
	\begin{equation}\label{eq:Lyapunov-decreasing}
		\frac{\dd}{\dd t} V\big(S(t), i(t, x)\big)\leq -\theta\dfrac{(S(t)-S^*(t))^2}{S(t)} \text{ for all }t\in\mathbb{R}.
	\end{equation}
\end{lemma}
\begin{proof}
	The well-definition of $\tau$ is clear since the left-hand side of \eqref{eq:id-Iinf} is strictly increasing and connects $0$ when $\tau\to -\infty$, to $+\infty$ when $\tau\to+\infty$. The openness of $D(V)$, the well-definition and continuity of $V(S, i)$ are also clear. 

	Let us check \eqref{eq:Lyapunov-decreasing}. We first remark that $i(t, x)=e^{t\gamma(x)(\alpha(x)\overline{S}(t)-1)}$, so it is clear that $i(t, x)\in D(V)$ and that $t\mapsto V(S(t), i(t, x))\in C^1(\mathbb{R})$. 

	We show \eqref{eq:Lyapunov-decreasing}. Let us write $V_1(S) = S^* g\left(\frac{S(t)}{S^*}\right) $ and $V_2(t) = \int_{{X}} i^*(x)g\left(\frac{i(t, x)}{i^*(x)}\right)I_0(\dd x)$, then we have
	\begin{align*}
		V_1'(t) &= S^*\frac{S'(t)}{S^*}g'\left(\frac{S(t)}{S^*}\right) = \left(\Lambda - \theta S(t) - S(t)\int_{{X}} \alpha^*\gamma(x)i(t, x)I_0(\dd x)\right)\left(1-\frac{S^*}{S(t)}\right)  \\  
		&=\left(\Lambda - \theta S(t) - S(t)\int_{{X}} \alpha^*\gamma(x)i(t, x)I_0(\dd x) -\Lambda + \theta S^* + S^*\int_{{X}} \alpha^*\gamma(x)i^*(x)I_0(\dd x)\right)\left(1-\frac{S^*}{S(t)}\right) \\ 
		& = -\theta\frac{(S(t)-S^*)^2}{S(t)} + \left(S^*\int_{{X}} \alpha^*\gamma(x)i^*(x)I_0(\dd x) - S(t)\int_{{X}} \alpha^*\gamma(x)i(t, x)I_0(\dd x)\right) \left(1-\frac{S^*}{S(t)}\right), \\
		& = -\theta\frac{(S(t)-S^*)^2}{S(t)} + S^*\int_{{X}} \alpha^*\gamma(x)i^*(x)I_0(\dd x) - \frac{(S^*)^2}{S(t)}\int_{{X}} \alpha^*\gamma(x)i^*(x)I_0(\dd x) \\ 
		&\quad - S(t)\int_{{X}} \alpha^*\gamma(x)i(t, x)I_0(\dd x) +S^*\int_{{X}} \alpha^*\gamma(x)i(t, x)I_0(\dd x)  , 
	\end{align*}
	and 
	\begin{align*}
		V_2'(t) &= \int_{{X}} i^*(x) \frac{ i_t(t, x)}{i^*(x)} g'\left(\frac{i(t, x)}{i^*(x)}\right) I_0(\dd x) = \int_{{X}} \gamma(x)\left(\alpha^* S(t) - 1\right)i(t, x)\left(1-\frac{i^*(x)}{i(t, x)}\right) I_0(\dd x) \\ 
		& = \int_{{X}} \gamma(x)\left(\alpha^* S(t) - 1\right)\left(i(t,x)-i^*(x)\right) I_0(\dd x) \\ 
		& =  \int_{{X}}\gamma(x)\alpha^* S(t)i(t,x) I_0(\dd x) - \int_{{X}}\gamma(x)i(t, x)I_0(\dd x) - \int_{{X}} \gamma(x)\alpha^* S(t) i^*(x) I_0(\dd x) \\ 
		&\quad + \int_{{X}} \gamma(x) i^*(x)I_0(\dd x).
	\end{align*}
	Recalling $S^*=\frac{1}{\alpha^*}$,  we have therefore
	\begin{align*} 
		\frac{\dd}{\dd t} V(S(t), i(t, \cdot)) &= \frac{\dd }{\dd t} V_1(t) + \frac{\dd}{\dd t} V_2(t) \\
		& = -\theta\frac{(S(t)-S^*)^2}{S(t)}  + 2\int_{{X}} \gamma(x)i^*(x) \dd x- \frac{(S^*)^2}{S(t)}\int_{{X}} \alpha^*\gamma(x)i^*(x)I_0(\dd x)\\
		&\quad - \int_{{X}} \alpha^*\gamma(x)S(t) i^*(t) I_0(\dd x) .
	\end{align*}
	Since 
	\begin{equation*}
		\int_{{X}}\alpha^* \gamma(x) i^*(x) \left(S(t)+\frac{(S^*)^2}{S(t)} \right) I_0(\dd x) \geq \int_{{X}}\alpha^* \gamma(x) i^*(x) \times 2S^* I_0(\dd x), 
	\end{equation*}
	which stems from the inequality $a+b\geq 2\sqrt{ab}$, we have indeed proved that \eqref{eq:Lyapunov-decreasing} holds. 
\end{proof}

Next we can determine the long-time behavior when the initial measure $I_0$ puts a positive mass on the set of maximal fitness. 
\begin{lemma}\label{lem:exponent bounded}
	Let Assumption \ref{as:params-nomut} hold. Assume that 
	$\mathcal R_0(I_0)>1$ and suppose that $I_0(\{\alpha(x)=\alpha^*\})>0$, or in other words, 
	\begin{equation*}
		\int_{\alpha(x)=\alpha^*} I_0(\dd x)>0.
	\end{equation*}
	Let $\eta(t):=\alpha^* \overline{S}(t)-1$, then there exists a constant $\overline{\eta}$ such that 
	\begin{equation}\label{eq:230518-3}
		-\infty<-\overline{\eta}\leq t\eta(t)\leq \overline{\eta}<+\infty,
	\end{equation}
	for all $t\geq 0$.
	If moreover $\liminf_{t\to-\infty} \int_{\alpha(x)=\alpha^*}I(t, \dd x)>0$, then up to changing the constant $\overline{\eta}$,  \eqref{eq:230518-3} holds for all $t\in\mathbb{R}$.
\end{lemma}
\begin{proof}
	We first remark that  $I(t, \dd x)$ can be written as 
	\begin{equation*}
		I(t, \dd x) = \exp\left(\gamma(x)t\left[\eta(t) + (\alpha(x)-\alpha^*)\frac{1}{t}\int_0^tS(s)\dd s\right]\right)I_0(\dd x).
	\end{equation*}
	By Jensen's inequality we have
	\begin{equation*}
		\exp\left(\int_{\alpha(x)=\alpha^*}\gamma(x) t\eta(t) \frac{I_0(\dd x)}{\int_{\alpha(x)=\alpha^*} I_0}\right)\leq \int_{\alpha(x)=\alpha^*} e^{\gamma(x)t\eta(t)}\frac{I_0(\dd x)}{\int_{\alpha(x)=\alpha^*} I_0(\dd z)}, 
	\end{equation*}
	so that 
	\begin{equation*}
		t\eta(t)\leq \frac{\int_{\alpha(x)=\alpha^*} I_0}{\int_{\alpha(x)=\alpha^*}\gamma(x) I_0(\dd x)}\ln\left(\int_{{X}} e^{\gamma(x)t\eta(t)}\frac{I_0(\dd x)}{\int_{{X}} I_0}\right) =\frac{\int_{\alpha(x)=\alpha^*} I_0}{\int_{\alpha(x)=\alpha^*}\gamma(x) I_0(\dd x)}\ln\left(\frac{1}{\int_{\alpha(x)=\alpha^*} I_0}\int_{\alpha(x)=\alpha^*} I(t, \dd x)\right) .
	\end{equation*}
	Applying Lemma \ref{lem:bounds-nomut}, $ I(t, \dd x) $ is bounded  and we have indeed an upper bound for $t\eta(t)$.   Next, writing 
	\begin{equation*}
		I(t, \dd x) = \exp\left(\gamma(x)t\eta(t) + (\alpha(x)-\alpha^*)\int_0^tS(s)\dd s\right)I_0(\dd x)
	\end{equation*}
	and recalling that $\int_0^tS(s)\dd s \to+\infty$ as $t\to+\infty$, the function  $\exp\left(\gamma(x)t\eta(t) + (\alpha(x)-\alpha^*)\int_0^tS(s)\dd s\right)$ converges almost everywhere (with respect to $I_0$)  to $0$ on ${X}\backslash \{\alpha(x)=\alpha^*\}$, so that by Lebesgue's dominated convergence theorem, we have
	\begin{equation*}
		\lim_{t\to+\infty} \int_{{X}\backslash \{\alpha(x)=\alpha^*\}}I(t, \dd x)=\int_{{X}\backslash \{\alpha(x)=\alpha^*\}}\lim_{t\to+\infty} \exp\left(\gamma(x)t\eta(t) + (\alpha(x)-\alpha^*)\int_0^tS(s)\dd s\right)I_0(\dd x) = 0.
	\end{equation*}
	Next it follows from Lemma \ref{lem:liminf-no-mut} that $\liminf_{t\to+\infty} I(t, \dd x)>0$, so that 
	\begin{equation*}
		\liminf_{t\to+\infty} \int_{\alpha(x)=\alpha^*}I(t, \dd x)=\liminf_{t\to+\infty}\int_{{X}}I(t, \dd x)>0.
	\end{equation*}
	Assume by contradiction that there is a sequence $(t_n)$ such that  $t_n\eta(t_n)\to -\infty$, then 
	\begin{equation*}
		\int_{\alpha(x)=\alpha^*}I(t, \dd x) = \int_{\alpha(x)=\alpha^*}e^{\gamma(x) t_n\eta(t_n)} I_0(\dd x)\leq  \int_{\alpha(x)=\alpha^*}e^{\gamma_* t_n\eta(t_n)} I_0(\dd x)=e^{\gamma_* t_n\eta(t_n)} \int_{\alpha(x)=\alpha^*}I_0(\dd x)\xrightarrow[t\to+\infty]{}0, 
	\end{equation*}
	where $\gamma_*:=\inf_{x\in\supp I_0}\gamma(x)>0$. This is a contradiction. Therefore there is a constant $\underline{\eta}>0$ such that 
	\begin{equation*}
		t\eta(t)\geq -\underline{\eta}>-\infty. 
	\end{equation*}
	In particular, the function $t\mapsto t\eta(t)$ is bounded by two constants, 
	\begin{equation*}
		-\infty<-\underline{\eta}\leq t\eta(t)\leq \overline{\eta}<+\infty.
	\end{equation*}
	This completes the proof of Lemma \ref{lem:exponent bounded}.
\end{proof}
Finally we prove that any complete orbit that is already concentrated on $\{\alpha(x)=\alpha^*\}$ is constant, provided the mass can be bounded when $t\to -\infty$. This is a kind of LaSalle principle, since we have a partial Lyapunov functional by Lemma \ref{lem:conv-alpha-const}.
\begin{lemma}\label{lem:Sconstant}
	Let Assumption \ref{as:params-nomut} hold and assume that $\alpha(x)\equiv \alpha^*$ for all $x\in\supp I_0$. Suppose that $\big(S(t), I(t,\dd x)\big)$ is a complete orbit of \eqref{eq:SI-no-mut} with initial data $\big(S_0, I_0(\dd x)\big)$,  and assume that 
	\begin{equation*}
		\liminf_{t\to-\infty} \int_{X}I(t, \dd x) >0. 
	\end{equation*}
	Then 
	\begin{equation*}
		S(t) \equiv \frac{1}{\alpha^*}, \text{ for all  } t\in\mathbb{R}.
	\end{equation*}
\end{lemma}
\begin{proof}
	Thanks to our assumption and the results of Lemma \ref{lem:exponent bounded}, we know that $t\eta(t)$ is bounded for $t\in\mathbb{R}$; moreover by Lemma \ref{lem:conv-alpha-const}, the functional $V(S(t), i(t, x))$ is well-defined and decreasing along the orbit $\big(S(t), i(t, x)\big), t\in\mathbb{R}$. Since $V(S(t), i(t, x))$ is bounded and decreasing there exists $V^\infty\in\mathbb{R}$ such that 
	\begin{equation*}
		V(S(t), i(t, x))\xrightarrow[t\to-\infty]{}V^\infty.
	\end{equation*}
	Let $t_n\to -\infty$ be a sequence with $S(t_n)\to S^{-\infty}$ and $t_n\eta(t_n)\to \eta^{-\infty}$ when $t_n\to -\infty$, so that $I(t_n, \dd x)=e^{t_n\eta(t_n)\gamma(x)}I_0(\dd x)$ converges when $t\to-\infty$ to $I^{-\infty}_0(\dd x):=e^{\eta^{-\infty}}I_0(\dd x)$. Then the shifted orbits $\big(S(t+t_n), I(t+t_n, \dd x)\big)$ converge, as $t_n\to-\infty$, to a complete orbit $\big(S^{-\infty}(t), I^{-\infty}(t, \dd x)\big)$ with $I^{-\infty}(0, \dd x)=I^{-\infty}_0(\dd x) = e^{\eta^{-\infty}} I_0(\dd x)$. By the continuity of $V$, along the new orbit $\big(S^{-\infty}(t), I^{-\infty}(t, \dd x)\big)$, we have that
	\begin{equation*}
		V\big(S^{-\infty}(t), I^{-\infty}(t, \dd x)\big) \equiv V^{-\infty}
	\end{equation*}
	is a constant. Thus $\frac{\dd }{\dd t} V\big(S^{-\infty}(t), I^{-\infty}(t, \dd x)\big)=0$ and, by \eqref{eq:Lyapunov-decreasing}, 
	\begin{equation*}
		S^{-\infty}(t)\equiv S^*=\frac{1}{\alpha^*} \text{ and } (S^{-\infty})'(t) \equiv 0 \text{ for all }t\in\mathbb{R}.
	\end{equation*}
	Then it follows from the first line in \eqref{eq:SI-no-mut-a} that 
	\begin{equation*}
		0=\Lambda-\theta S^*-S^*\int_{X}\alpha^* \gamma(x)I^{-\infty}(t, \dd x) \text{ for all } t\in\mathbb{R},
	\end{equation*}
	therefore in particular 
	\begin{equation*}
		\int_{X}\gamma(x)e^{\eta^{\infty}\gamma(x)}I_0(\dd x) = \frac{\theta}{\alpha^*}\big(\mathcal{R}_0-1\big).
	\end{equation*}
	Thus $\eta^{-\infty}$ is a solution of \eqref{eq:id-Iinf} and, by the uniqueness of the solution, we have $\eta^{-\infty}=\tau$ and $I^{-\infty}_0(\dd x)=i^*(x)I_0(\dd x)$. Thus
	\begin{equation*}
		V^{-\infty} = V\big(S^{-\infty}, e^{\eta^{-\infty}\gamma(x)}\big) = V\big(S^*, i^*(x)\big) = 0.
	\end{equation*}
	Thus $V^{-\infty}$ is the smallest possible value of $V(S, i(x))$. Since $t\mapsto V(S(t), i(t, x))$ is nonincreasing, we have therefore 
	\begin{equation*}
		V\big(S(t), i(t, x)\big) \equiv 0, \frac{\dd}{\dd t}V\big(S(t), i(t, x)\big) \equiv 0, \text{ for all } t\in\mathbb{R}.
	\end{equation*}
	By \eqref{eq:Lyapunov-decreasing}, we have therefore
	\begin{equation*}
		S(t)\equiv S^*=\frac{1}{\alpha^*} \text{ for all }t\in\mathbb{R}, 
	\end{equation*}
	which completes the proof of Lemma \ref{lem:Sconstant}.
\end{proof}

\begin{lemma}\label{lem:exclusion-positive}
	Let Assumption \ref{as:params-nomut} hold. Assume that 
	$\mathcal R_0(I_0)>1$ and suppose that $I_0(\{\alpha(x)=\alpha^*\})>0$, or in other words, 
	\begin{equation*}
		\int_{\alpha(x)=\alpha^*} I_0(\dd x)>0.
	\end{equation*}
	Then 
	\begin{equation*}
		S(t)\xrightarrow[t\to+\infty]{} \frac{1}{\alpha^*} \text{ and } \Vert I(t, \dd x)-I^\infty(\dd x)\Vert_{TV} \xrightarrow[t\to+\infty]{}0,
	\end{equation*}
	where we have the following formula for $I^\infty(\dd x)$, with $\tau$ being the unique solution of \eqref{eq:id-Iinf}, 
	\begin{equation*}
		I^\infty(\dd x) = e^{\tau\gamma(x)}\mathbbm{1}_{\alpha(x)=\alpha^*} I_0(\dd x).
	\end{equation*}
\end{lemma}
\begin{proof}
	Suppose that there exists a sequence $t_n\to+\infty$ and $\eta^*\in [-\overline{\eta}, \overline{\eta}]$ such that 
	\begin{equation*}
		\lim_{n\to+\infty} t_n\eta(t_n) = \eta^*.
	\end{equation*}
	By Lemma \ref{lem:compactness-nomut}, the shifted orbits $\big(S(t+t_n), I(t+t_n, \dd x)\big)$ converge  to a complete orbit $\big(S^{\infty}(t), I^{\infty}(t, \dd x)\big)$ with $I(t_n, \dd x)\xrightarrow[t\to+\infty]{} e^{\gamma(x)\eta^*}\mathbbm{1}_{\alpha(x)=\alpha^*}I_0(\dd x)=I^{\infty}(t, \dd x)$. We know that $\supp I^{\infty}(t, \dd x)\subset \{\alpha(x)=\alpha^*\}$ and  by Lemma \ref{lem:weak-persistent-nomut} we have 
	\begin{equation*}
		\int_{X}I^\infty(t, \dd x) \geq \liminf_{s\to+\infty}\int_{X}I(s, \dd x)>0 \text{ for all }t\in\mathbb{R}. 
	\end{equation*}
	Thus we can apply Lemma \ref{lem:Sconstant} which shows that 
	\begin{equation*}
		S^\infty(t)\equiv \frac{1}{\alpha^*} \text{ for all }t\in\mathbb{R}.
	\end{equation*}
	Thus $(S^\infty)'(t)\equiv 0$ and by using the first line in \eqref{eq:SI-no-mut-a} we find that 
	\begin{equation*}
		\int_{X}\gamma(x)e^{\eta^*\gamma(x)}\mathbbm{1}_{\alpha(x)=\alpha^*}I_0(\dd x) = \frac{\theta}{\alpha^*}\big(\mathcal{R}_0(I_0)-1\big).
	\end{equation*}
	This is \eqref{eq:id-Iinf} which has a unique solution $\eta^*=\tau$. Since we can extract from any sequence $t_n\to+\infty$ a subsequence with $t_n\eta(t_n)\to \tau$, we conclude that $\lim_{t\to+\infty} t\eta(t)=\tau$ therefore 
	\begin{equation*}
		I(t, \dd x) \xrightarrow[t\to+\infty]{\Vert\cdot\Vert_{TV}} e^{\tau\gamma(x)}\mathbbm{1}_{\alpha(x)=\alpha^*}I_0(\dd x).
	\end{equation*}
	We show similarly that $S(t)\to \frac{1}{\alpha^*}$ as $t\to+\infty$.
\end{proof}

When the set of maximal fitness $\{\alpha(x)=\alpha^*\}$ is negligible for $I_0$, it is more difficult to obtain a general result for the long-time behavior of $I(t, \dd x)$. We start with a short but useful estimate on the rate $\eta(t)$
\begin{lemma}\label{lem:teta(t)}
	Let Assumption \ref{as:params-nomut} {hold.}	
	Assume that 
	$\mathcal R_0(I_0)>1$. Suppose that $I_0(\{\alpha(x)=\alpha^*\})=0$ and set 
	$$
	\eta(t):=\alpha^*\overline{S}(t)-1,\text{ with }\overline{S}(t)=\frac{1}{t}\int_0^t S(s)\dd s,
	$$
	where $\alpha^*:=\sup_{x\in\supp I_0}\alpha(x)$. Then it holds
	\begin{equation*}
		t\eta(t)\xrightarrow[t\to\infty]{}+\infty.
	\end{equation*}
\end{lemma} 
\begin{proof}
	Assume by contradiction that there exists a sequence $t_n\to +\infty$ such that $t_n \eta(t_n)$ has a uniform upper bound as $t_n\to+\infty$, then observe that the quantity 
	$$
	e^{\gamma(x)(\alpha(x)\overline{S}(t_n)-1)t_n} = e^{\gamma(x)(\alpha(x)-\alpha^*)\overline{S}(t_n)t_n + \gamma(x)\eta(t_n)t_n}
	$$
	is uniformly bounded in $t_n$ and vanishes as $t_n\to+\infty$ almost everywhere with respect to $I_0(\dd x)$. By a direct application of Lebesgue's dominated convergence Theorem, we have therefore
	\begin{equation*}
		\int_{{X}}I(t_n,\dd x) = \int_{{X}} e^{\gamma(x)(\alpha(x)\overline{S}(t_n)-1)t_n}I_0(\dd x)\xrightarrow[t_n\to+\infty]{}0,
	\end{equation*}
	which is in contradiction with Lemma \ref{lem:unifpers-nomut}. We conclude that $t\eta(t) \to +\infty$ as $t\to +\infty$.
\end{proof}

We are now in the position to prove Theorem \ref{thm:measures}.
\begin{proof}[Proof of Theorem \ref{thm:measures}]
	The convergence of $S(t)$ and  $I(t, \dd x)$ in case i) was proved in Lemma \ref{lem:exclusion-positive}.

	Let us focus on case ii), that is to say, we assume 
	\begin{equation*}
		\int_{\alpha(x)=\alpha^*}I_0(\dd x) = 0.
	\end{equation*}
	The uniform persistence of $I(t, \dd x)$ is a consequence of \ref{lem:unifpers-nomut}. The concentration on the maximal fitness was proved in Lemma \ref{lem:weakconc-nomut}. Let us show that $S(t)\to \frac{1}{\alpha^*}$.  Suppose by contradiction that it is not the case, then there exists $\varepsilon>0$ and a sequence $t_n\to+\infty$ with $\left|S(t_n)-\frac{1}{\alpha^*}\right|\geq \varepsilon$. By Lemma \ref{lem:compactness-nomut} we can extract $t_n$ a subsequence such that the shifted orbits $\big(S(t+t_n), I(t+t_n, \dd x)\big)$ converge to $\big(S^\infty(t), I^\infty(t, \dd x)\big)$. We have $\int_{\alpha(x)<\alpha^*}I^\infty(t, \dd x)=0$, $\int_{X}I^\infty(t, \dd x)>0$, and 
	\begin{equation*}
		\liminf_{t\to-\infty}\int_{X}I^\infty(t, \dd x) \geq \liminf_{t\to+\infty}\int_{X} I(t, \dd x)>0, 
	\end{equation*}
	so by Lemma \ref{lem:Sconstant} we have 
	\begin{equation*}
		\lim_{n\to+\infty} S(t_n) = S^\infty(0)= \frac{1}{\alpha^*}.
	\end{equation*}
	This is obviously a contradiction. Theorem \ref{thm:measures} is proved. 
\end{proof}

We now turn to the proof of Proposition \ref{prop:local-survival} and we first prove that $I(t,\dd x)$ concentrates on the set of points maximizing both $\alpha$ and $\gamma$. This property is summarized in the next lemma.

\begin{lemma}\label{lem:local-survival}
	Let Assumptions \ref{as:reg-bound} {hold.}	
	Assume that 
	$\mathcal R_0(I_0)>1$ and that $I_0(\{\alpha(x)=\alpha^*\})=0$. 
	Recalling the definition of $\alpha^*$ in \eqref{def-alpha*} and $\gamma^*$ in Assumption \ref{as:reg-bound}, set $\Gamma_0(I_0)$ be  the set of maximal points of $\gamma$ on $\{\alpha(x)=\alpha^*\}$, defined by
	\begin{equation*}
		\Gamma_0(I_0):=\{x\,:\,\gamma(x)=\gamma^*\text{ and } \alpha(x)=\alpha^*\}.
	\end{equation*}
	Then one has   
	\begin{equation*}
		d_0\left(I(t, \dd x), \mathcal M_+(\Gamma_0(I_0))\right)\xrightarrow[t\to+\infty]{}0.
	\end{equation*}
\end{lemma}
\begin{proof}
	We decompose the proof in several steps. \medskip

	\noindent\textbf{Step 1: We show that $I(t, \dd x)$ and $\mathbbm{1}_{\alpha(x)\overline{S}(t)\geq 1} I(t, \dd x) $ are asymptotically close in $\Vert\cdot\Vert_{TV}$.}  That is to say, 
	\begin{equation*}
		\Vert I(t, \dd x)-\mathbbm{1}_{\alpha(\cdot)\overline{S}(t)\geq 1} I(t, \dd x)\Vert_{TV} \xrightarrow[t\to+\infty]{}0.
	\end{equation*}
	Indeed we have
	\begin{equation*}
		I(t, \dd x)-\mathbbm{1}_{\alpha(x)\overline{S}(t)\geq 1} I(t, \dd x) = \mathbbm{1}_{\alpha(x)\overline{S}(t)<1}I(t, \dd x)=\mathbbm{1}_{\alpha(x)\overline{S}(t)<1}e^{\gamma(x)(\alpha(x)\overline{S}(t)-1)t}I_0(\dd x).
	\end{equation*}
	First note that the function $\mathbbm{1}_{\alpha(x)\overline{S}(t)<1}I(t, \dd x)=\mathbbm{1}_{\alpha(x)\overline{S}(t)<1}e^{\gamma(x)(\alpha(x)\overline{S}(t)-1)t}{I_0(\dd x)}$ is uniformly bounded. On the other hand, since $I_0(\{\alpha(x)=\alpha^*\})=0$ recall that $\overline{S}(t)\to \frac{1}{\alpha^*}$ for $t\to\infty$, so that $\mathbbm{1}_{\alpha(x)\overline{S}(t)<1}\to 0$ as $t\to\infty$ almost everywhere with respect to $I_0$. It follows from Lebesgue's dominated convergence Theorem that 
	\begin{equation*}
		\int_{{X}}\mathbbm{1}_{\alpha(x)\overline{S}(t)<1}e^{\gamma(x)(\alpha(x)\overline{S}(t)-1)t} I_0(\dd x)\xrightarrow[t\to+\infty]{}0.
	\end{equation*}
	\medskip

	\noindent\textbf{Step 2: We show that the measure $\mathbbm{1}_{\overline{S}(t)y\geq 1}e^{\bar\gamma(y\overline{S}(t)-1)t} A(\dd y)$ is bounded when $t\to \infty$ for all $\bar\gamma<\gamma^*$.} {Recall that $A$ is the pushforward measure of $I_0$ by the continuous map $\alpha$.} Note that $I_0(\{\alpha(x)=\alpha^*\})=0$ implies that $A(\{\alpha^*\})=0$ and remark that one has 
	\begin{equation*}
		\int_{{X}} \mathbbm{1}_{\alpha(x)\overline{S}(t)\geq 1} I(t, \dd x) = \int_{\min\left(\alpha^*,1/\overline{S}(t)\right)}^{\alpha^*} \int_{\alpha(x)=y}  e^{\gamma(x)(y\overline{S}(t)-1)t}I_0(y, \dd x)A(\dd y), 
	\end{equation*}
	so, according to Step 1, for $t$ sufficiently large one has
	\begin{align*}
		\int_{\gamma(x)\in[\bar\gamma,\gamma^*]}I(t, \dd x) &= \int_{\min(\alpha^*,1/\overline{S}(t))}^{\alpha^*} \int_{\alpha(x)=y \text{ and } \gamma(x)\in [\bar\gamma,\gamma^*]}   e^{\gamma(x)(y\overline{S}(t)-1)t}I_0(y, \dd x)A(\dd y)+o(1)\\ 
		&\geq \int_{\min(\alpha^*,1/\overline{S}(t))}^{\alpha^*} \int_{\alpha(x)=y \text{ and } \gamma(x)\in [\bar\gamma,\gamma^*]} e^{\bar\gamma(y\overline{S}(t)-1)t}I_0(y, \dd x)A(\dd y)+o(1)\\
		&=\int_{\min(\alpha^*,1/\overline{S}(t))}^{\alpha^*} \int_{\alpha(x)=y \text{ and } \gamma(x)\in [\bar\gamma,\gamma^*]}I_0(y, \dd x) e^{\bar\gamma(y\overline{S}(t)-1)t}A(\dd y)+o(1) \\ 
		&\geq m \int_{\min(\alpha^*,1/\overline{S}(t))}^{\alpha^*}  e^{\bar\gamma(y\overline{S}(t)-1)t}A(\dd y)+o(1), 
	\end{align*}
	wherein $m>0$ is the constant associated with $\bar\gamma$ in Assumption \ref{as:reg-bound} and {we used the Landau notation $o(1)$ to collect terms that converges to 0 as $t\to+\infty$}. Recalling the upper bound for $I(t, \dd x)$ from Lemma \ref{lem:bounds-nomut}, we have
	\begin{equation*}
		\limsup_{t\to+\infty}\int_{\min(\alpha^*,1/\overline{S}(t))}^{\alpha^*}  e^{\bar\gamma(y\overline{S}(t)-1)t}A(\dd y) \leq \limsup_{t\to+\infty}\dfrac{1}{m}\int_{{X}} I(t, \dd x) \leq \frac{\Lambda}{m\min(\theta, \gamma_0)} <+\infty.
	\end{equation*}
	This implies that 
	\begin{equation*}
		\limsup_{t\to+\infty}\int_{\alpha\left({\rm supp}(I_0)\right)}  e^{\bar\gamma(y\overline{S}(t)-1)t}A(\dd y) <\infty.
	\end{equation*}
	Note that, if the constant $m$ is independent of $\bar \gamma$, then the above estimate does not depend on $\bar\gamma$ either.

	\medskip

	\noindent\textbf{Step 3: We show that $\int \mathbbm{1}_{\gamma(x)< \bar\gamma}\mathbbm{1}_{\overline{S}(t)\alpha(x)\geq 1}I(t, \dd x)$ vanishes whenever $\bar\gamma<\gamma^*$.} \\
	Fix $\bar\gamma<\gamma^*$ and let $0<\varepsilon{\leq}\frac{\gamma^*-\bar\gamma}{2}$. Then we have
	\begin{align*}
		\int_{\gamma(x)\leq\bar\gamma\text{ and } \alpha(x)\geq 1/\overline{S}(t)} I(t, \dd x)&= \int_{\min(\alpha^*,1/\overline{S}(t))}^{\alpha^*}\int_{\gamma(x)\leq\bar\gamma\text{ and } \alpha(x)\geq 1/\overline{S}(t)} e^{\gamma(x)(y\overline{S}(t)-1)t}I_0(y, \dd x) A(\dd y) \\
		&\leq \int_{\min(\alpha^*,1/\overline{S}(t))}^{\alpha^*}\int_{\gamma(x)\leq\bar\gamma\text{ and } \alpha(x)\geq 1/\overline{S}(t)}I_0(y, \dd x)e^{(\gamma^*-2\varepsilon)(y\overline{S}(t)-1)t} A(\dd y) \\
		&\leq \int_{\min(\alpha^*,1/\overline{S}(t))}^{\alpha^*}\int_{\gamma(x)\leq\bar\gamma\text{ and } \alpha(x)\geq 1/\overline{S}(t)}I_0(y, \dd x)e^{-\varepsilon (y\overline{S}(t)-1)t}e^{\bar\gamma(y\overline{S}(t)-1)t} A(\dd y) .
	\end{align*}
	Reducing $\varepsilon$ if necessary we may assume that $\frac{\bar\gamma}{\varepsilon}>1$. Therefore it follows from H\"older's inequality that 
	\begin{align}
		\nonumber\int_{\gamma(x)\leq\bar\gamma\text{ and } \alpha(x)\geq 1/\overline{S}(t)}I(t, \dd x)&\leq \left(\int_{\min(\alpha^*,1/\overline{S}(t))}^{\alpha^*}\left(e^{-\varepsilon (y\overline{S}(t)-1)t}\right)^{\frac{\bar\gamma}{\varepsilon}}e^{\bar\gamma(y\overline{S}(t)-1)t} A(\dd y)\right)^{\frac{\varepsilon}{\bar\gamma}}\\
		\nonumber&\quad\times\left(\int_{\min(\alpha^*,1/\overline{S}(t))}^{\alpha^*}\left(\int_{\alpha(x)=y}I_0(y, \dd x)\right)^{\frac{\bar\gamma}{\bar\gamma-\varepsilon}}e^{\bar\gamma(y\overline{S}(t)-1)t} A(\dd y)\right)^{1-\frac{\varepsilon}{\bar\gamma}} \\ 
		\nonumber&\leq \left(\int_{\min(\alpha^*,1/\overline{S}(t))}^{\alpha^*} A(\dd y)\right)^{\frac{\varepsilon}{\bar\gamma}} \left(\int_{\min(\alpha^*,1/\overline{S}(t))}^{\alpha^*}e^{\bar\gamma(y\overline{S}(t)-1)t} A(\dd y)\right)^{1-\frac{\varepsilon}{\bar\gamma}}\\
		&=I_0(L_{\alpha^*-\min(\alpha^*,1/\overline{S}(t))}(I_0))^{\frac{\varepsilon}{\bar\gamma}} \left(\int_{\min(\alpha^*,1/\overline{S}(t))}^{\alpha^*}e^{\bar\gamma(y\overline{S}(t)-1)t} A(\dd y)\right)^{1-\frac{\varepsilon}{\bar\gamma}}. \label{eq:uppermass-gamma} 
	\end{align}
	Since $\overline{S}(t)\to 1/\alpha^*$ as $t\to\infty$,  $I_0(L_\epsilon(I_0))\xrightarrow[\epsilon\to 0]{}0$ and by the boundedness of $\int_{\min(\alpha^*,1/\overline{S}(t))}^{\alpha^*}e^{\bar\gamma(y\overline{S}(t)-1)t} A(\dd y)$ shown in Step 2, we have indeed 
	\begin{equation*}
		\int_{\gamma(x)\leq\bar\gamma\text{ and } \alpha(x)\geq 1/\overline{S}(t)} I(t, \dd x)\xrightarrow[t\to+\infty]{}0,
	\end{equation*}
	and this completes {proof of} Lemma \ref{lem:local-survival}.
\end{proof}

\begin{proof}[Proof of Proposition \ref{prop:local-survival}]
	{
		The concentration of the distribution to $\mathcal M_+(\{\alpha(x)=\alpha^*\}\cap\{\gamma(x)=\gamma^*\})$ was shown in Lemma \ref{lem:local-survival}. 

		Next we prove the asymptotic mass. Pick a sentence $t_n\to +\infty$. By the compactness of the orbit (proved in Lemma \ref{lem:unifpers-nomut}) we can extract from $t_n$ a subsequence $t'_n$ such that there exists a Radon measure $I^\infty(\dd x)$ with 
		\begin{equation*}
			d_0(I(t, \dd x), I^\infty(\dd x))\xrightarrow[t\to+\infty]{}0,
		\end{equation*}
		and since $S(t)\to \frac{1}{\alpha^*}$ and upon further extraction, $S'(t'_n)\to 0$. Therefore,
		\begin{equation*}
			\int_{{X}} \alpha(x)\gamma(x)I(t'_n, \dd x)=\dfrac{\Lambda-S'(t'_n)}{S(t'_n)}-\theta \xrightarrow[n\to+\infty]{} \alpha^*\Lambda - \theta = \theta\left(\mathcal R_0(I_0)-1\right).
		\end{equation*}
		By the concentration result in Lemma \ref{lem:local-survival}, $I^\infty$ is concentrated on $\{\alpha(x)=\alpha^*\}\cap \{\gamma(x)=\gamma^*\}$. Therefore 
		\begin{equation*}
			\alpha^*\gamma^*\int I^\infty(\dd x) = \int \alpha(x)\gamma(x) I^\infty(\dd x) = \lim_{n\to +\infty}{\int_{{X}} \alpha(x)\gamma(x)I(t'_n, \dd x)} = \theta\left(\mathcal R_0(I_0)-1\right), 
		\end{equation*}
		so that 
		\begin{equation*}
			\lim_{n\to +\infty} \int I(t'_n, \dd x) = \int I^\infty(\dd x) = \dfrac{\theta}{\alpha^*\gamma^*}\left(\mathcal R_0(I_0)-1\right).
		\end{equation*}
		Since the limit is independent of the sequence $t_n$, we have indeed shown that 
		\begin{equation*}
			\lim_{t\to +\infty} \int_{{X}} I(t, \dd x)  = \dfrac{\theta}{\alpha^*\gamma^*}\left(\mathcal R_0(I_0)-1\right).
		\end{equation*}
	}

	To prove the {last} statement, set	
	\begin{equation*}
		f(t):=\int_{\min(\alpha^*,1/\overline{S}(t))}^{\alpha^*} \int_{\bar\gamma}^{\gamma^*} e^{z(y\overline{S}(t)-1)t}I_0^\alpha(y, \dd z)A(\dd y), 
	\end{equation*}
	where  $\bar\gamma<\gamma^*$.
	It follows from \eqref{eq:uppermass-gamma} that 
	\begin{equation*}
		\int_{\gamma(x)\leq \bar\gamma\text{ and }\alpha(x)\geq 1/\overline{S}(t)}I(t, \dd x)\xrightarrow[t\to+\infty]{}0, 
	\end{equation*}
	therefore 
	\begin{equation*} 
		f(t) = \int_{\min(\alpha^*,1/\overline{S}(t))}^{\alpha^*} \int_{\bar\gamma}^{\gamma^*}\int_{\{\gamma(x)=z\}} e^{z(y\overline{S}(t)-1)t} I_0^{\alpha, \gamma}(y, z, \dd x)I_0^\alpha(z, \dd y)A(\dd y) = \int_{\bar\gamma\leq\gamma(x)\leq \gamma^*\text{ and }\alpha(x)\geq 1/\overline{S}(t)} I(t, \dd x) 
	\end{equation*}
	satisfies	
	\begin{equation*}
		0<\liminf_{t\to+\infty} \int I(t, \dd x) =\liminf_{t\to+\infty}{\left[\int_{\gamma(x)\leq\bar\gamma\text{ and } \alpha(x)\geq 1/\overline{S}(t)}  I(t, \dd x)+f(t)\right]} \leq \liminf_{t\to+\infty} f(t) 
	\end{equation*}
	Remark that 
	\begin{align*}
		\int \mathbbm{1}_{{U}}(x)\mathbbm{1}_{\gamma\geq \bar\gamma}\mathbbm{1}_{{S}({t})y\geq 1} I({t}, \dd x)& = \int_{\min(\alpha^*,1/{\overline S}({t}))}^{\alpha^*} \int_{\bar\gamma}^{\gamma^*}\int_{\{\gamma(x)=z\}} \mathbbm{1}_{{U}}(x)e^{z(y{S}({t})-1){t}} I_0^{\alpha, \gamma}(y, z, \dd x)I_0^\alpha(y, \dd z)A(\dd y)\\ 
		&= \int_{\min(\alpha^*,1/{\overline S}({t}))}^{\alpha^*} \int_{\bar\gamma}^{\gamma^*}\int_{\{\gamma(x)=z\}} \mathbbm{1}_{{U}}(x) I_0^{\alpha, \gamma}(y, z, \dd x)e^{z(y{S}({t})-1){t}}I_0^\alpha(y, \dd z)A(\dd y)\\ 
		&\geq  \int_{\min(\alpha^*,1/{\overline S}({t}))}^{\alpha^*} \int_{\bar\gamma}^{\gamma^*}\frac{m}{2}e^{z(y{S}({t})-1){t}} I_0^\alpha(y, \dd z)A(\dd y)\\  
		&\geq f({t})\frac{m}{2},
	\end{align*}
	provided $t$ is sufficiently large and $\bar\gamma$ is sufficiently close to $\gamma^*$, where 
	\begin{equation*}
		m:=
		\liminf_{\varepsilon\to 0}\overset{A(\dd y)}{\underset{\alpha^*-\varepsilon\leq y\leq \alpha^*}{\essinf}}\overset{I_0^\alpha(y, \dd z)}{\underset{\gamma^*-\varepsilon\leq z\leq \gamma^*}{\essinf}}
		\int \mathbbm{1}_{x\in U}I_0^{\alpha, \gamma}(y,z, \dd x)>0. 
	\end{equation*}
	Therefore
	\begin{equation*}
		\liminf_{t\to+\infty}\int_{U} I(t, \dd x)\geq \frac{m}{2}\liminf_{t\to+\infty}f(t)>0.
	\end{equation*}
	This completes proof of Proposition \ref{prop:local-survival}.
\end{proof}

\section{The case of a finite number of regular maxima}
\label{sec:finite-maxima}

In this section we prove Theorem \ref{THEO-eta}. To that aim, we shall make use of the following formula
\begin{equation}\label{formule}
	I(t,\dd x)=\exp\left(\gamma(x)\left(\alpha(x)\int_0^tS(s)\dd s-t\right)\right)I_0(\dd x).
\end{equation}
{
	Recall also the definition of $\eta(t)$:
	\begin{equation*}
		\eta(t) = \alpha^* \, \frac{1}{t}\int_0^t S(s)\dd s -1 = \alpha^*\overline{S}(t)-1.
	\end{equation*}
}

\begin{proof}[Proof of Theorem \ref{THEO-eta}]
	We split the proof of this result into three parts. We first derive a suitable upper bound. We then derive a lower bound in a second step and we  conclude the proof of the theorem by estimating the large time asymptotic of the mass of $I$ around each point of $\{\alpha(x)=\alpha^*\}$.

	\noindent{\bf Upper bound:}\\
	Let $i=1,..,p$ be given. Recall that $\nabla \alpha(x_i)=0$. 
	Now due to $(iii)$ in Assumption \ref{ASS-calculs}
	there exist $m>0$ and $T>\ep_0^{-2}$ large enough such that
	for all $t\geq T$ and for all $y\in B\left(0,t^{-\frac{1}{2}}\right)$ we have
	\begin{equation*}
		\alpha(x_i+y)-\alpha^*\leq -\alpha^* m\|y\|^2.
	\end{equation*}
	As a consequence, setting
	$$
	\Gamma(x)=\gamma(x)\frac{\alpha(x)}{\alpha^*},
	$$
	we infer from \eqref{formule} and the lower estimate of $I_0$ around $x_i$ given in Assumption \ref{ASS-calculs} $(ii)$, that for all $t>T$
	\begin{equation*}
		\begin{split}
			&\int_{\|x_i-x\|\leq t^{-\frac{1}{2}}}I(t,\dd x)\geq M^{-1}\int_{|y|\leq t^{-\frac{1}{2}}} |y|^{\kappa_i}\exp\left[t\eta(t)\Gamma(x_i+y)-t\gamma(x_i+y)m|y|^2\right]\dd y.
		\end{split}
	\end{equation*}
	Next since the function $I=I(t,\dd x)$ has a bounded mass, there exists some constant $C>0$ such that
	$$
	\int_{\R^N} I(t,\dd x)\leq C,\;\forall t\geq 0.
	$$
	Coupling the two above estimates yields for all $t>T$
	\begin{equation*}
		\int_{|y|\leq t^{-\frac{1}{2}}} |y|^{\kappa_i}\exp\left[t\eta(t)\Gamma(x_i+y)-t\gamma(x_i+y)m|y|^2\right]\dd y\leq MC.
	\end{equation*}
	Hence setting $z=y\sqrt t $ into the above integral rewrites as
	\begin{equation*}
		\int_{|z|\leq 1} t^{-\frac{\kappa_i}{2}}|z|^{\kappa_i}\exp\left[t\eta(t)\Gamma(x_i+t^{-\frac{1}{2}}z)-\gamma(x_i+t^{-\frac{1}{2}}z)m|z|^2\right]\frac{\dd z}{t^{N/2}}\leq MC,\;\forall t>T.
	\end{equation*}
	Now, since $\gamma$ and $\alpha$ are both smooth functions, we have uniformly for $|z|\leq 1$ and $t\gg 1$:
	\begin{align*}
		&\Gamma(x_i+t^{-\frac{1}{2}}z)=\gamma(x_i)+O\left(t^{-\frac{1}{2}}\right),\\
		&\gamma(x_i+t^{-\frac{1}{2}}z)=\gamma(x_i)+O\left(t^{-\frac{1}{2}}\right).
	\end{align*}
	This yields {for all $t\gg 1$}
	\begin{gather*}
		\int_{|z|\leq 1} t^{-\frac{\kappa_i}{2}}|z|^{\kappa_i}\exp\left[t\eta(t)\left(\gamma(x_i)+O\left(t^{-\frac{1}{2}}\right)\right)-\gamma(x_i)m|z|^2\right]\frac{\dd z}{t^{N/2}}\leq CM,\\ 
		t^{-\frac{\kappa_i}{2}-\frac{N}{2}}e^{t\eta(t)\left(\gamma(x_i)+O\left(t^{-\frac{1}{2}}\right)\right)}\int_{|z|\leq 1} |z|^{\kappa_i}e^{-\gamma(x_i)m|z|^2}\dd z\leq CM,
	\end{gather*}
	that also ensures the existence of some constant $c_1\in\R$ such that
	\begin{equation*}
		t\eta(t)\left(\gamma(x_i)+O\left(t^{-\frac{1}{2}}\right)\right)-\frac{N+\kappa_i}{2}\ln t\leq c_1,\;\forall t\gg 1,
	\end{equation*}
	or equivalently
	\begin{equation*}
		\eta(t)\leq \frac{N+\kappa_i}{2{\gamma(x_i)}}\frac{\ln t}{t}+O\left(\frac{1}{t}\right)\text{ as }t\to\infty.
	\end{equation*}
	Since the above upper-bound holds for all $i=1,..,p$, we obtain the following upper-bound
	\begin{equation}\label{upper}
		\eta(t)\leq \varrho\frac{\ln t}{t}+O\left(\frac{1}{t}\right)\text{ as }t\to\infty,
	\end{equation}
	where $\varrho$ is defined in \eqref{def-varrho}.

	\noindent{\bf Lower bound:}\\
	Let $\ep_1\in (0,\ep_0)$ small enough be given such that for all $i=1,..,p$ and $|y|\leq\ep_1$ one has
	\begin{equation*}
		\alpha(x_i+y)\leq \alpha^*-\frac{\ell}{2}|y|^2.
	\end{equation*}
	Herein $\ell>0$ is defined in Assumption \ref{ASS-calculs} $(iii)$.
	Next define $m>0$ by
	\begin{equation*}
		m=\frac{\ell}{2} \min_{i=1,..,p}\min_{|y|\leq\ep_1}\gamma(x_i+y)>0.
	\end{equation*}
	{Recall that $\Gamma(x)=\frac{\alpha(x)\gamma(x)}{\alpha^*}$ and $\nabla\Gamma(x)=\frac{1}{\alpha^*}\left(\alpha(x)\nabla\gamma(x)+\gamma(x)\nabla\alpha(x)\right)$. }
	Consider $M>0$ such that for all $k=1,..,p$ and all $|x-x_k|\leq \ep_1$ 
	one has
	\begin{equation}\label{Gamma}
		|\Gamma(x)-\gamma(x_k)-\nabla\gamma(x_k)\cdot(x-x_k)|\leq M|x-x_k|^2.
	\end{equation}

	Next fix $i=1,..,p$ and $\ep\in (0,\ep_1)$. Then one has for all $t>0$
	\begin{equation*}
		\begin{split}
			\int_{|x-x_i|\leq\ep}I(t,\dd x)&\leq \int_{|x-x_i|\leq\ep}\exp\left[t\eta(t)\Gamma(x)-t m |x-x_i|^2\right]I_0(\dd x)\\
			&\leq e^{t\eta(t)\gamma(x_i)}\int_{|x-x_i|\leq\ep}\exp\left[t\left(\eta(t)\nabla\gamma({x_i})\cdot (x-x_i)- (m+O(\eta(t)) |x-x_i|^2\right)\right]I_0(\dd x).
		\end{split}
	\end{equation*}
	Now observe that for all $t\gg 1$ one has 
	$$
	\eta(t)\nabla\gamma (x_k)\cdot(x-x_i)- (m+O(\eta(t))) |x-x_i|^2=-(m+O(\eta(t)))\left|x-x_i-\frac{\eta(t)\nabla\gamma(x_i)}{2(m+O(\eta(t)))}\right|^2+\frac{\eta(t)^2\Vert \nabla\gamma(x_i)\Vert^2}{4(m+O(\eta(t)))},
	$$
	so that we get, using Assumption \ref{ASS-calculs} $(ii)$, that
	\begin{equation*}
		\begin{split}
			\int_{|x-x_i|\leq\ep}I(t,\dd x)&\leq e^{t\eta(t)\gamma(x_i)+\frac{t\eta(t)^2\Vert \nabla\gamma(x_i)\Vert^2}{4(m+O(\eta(t))}}\int_{|x-x_i|\leq\ep}\exp\left[-(m+O(\eta(t))t\left|x-x_i-\frac{\eta(t)\nabla\gamma (x_i)}{2(m+O(\eta(t))}\right|^2\right]I_0(\dd x)\\
			&\leq Me^{t\eta(t)\gamma(x_i)+\frac{t\eta(t)^2\Vert \nabla\gamma(x_i)\Vert^2}{4(m+O(\eta(t))}} \\ 
			&\quad \times\int_{|x-x_i|\leq\ep}|x-x_i|^{\kappa_i}\exp\left[-(m+O(\eta(t))t\left|x-x_i-\frac{\eta(t)\nabla\gamma (x_i)}{2(m+O(\eta(t))}\right|^2\right]\dd x.
		\end{split}
	\end{equation*}
	We now make use of the following change of variables in the above integral
	$$
	z=\sqrt t\left(x-x_i-\frac{\eta(t)\nabla\gamma (x_i)}{2(m+O(\eta(t))}\right),
	$$
	so that we end up with
	\begin{equation*}
		\int_{|x-x_i|\leq\ep}I(t,\dd x)\leq t^{-\frac{N+\kappa_i}{2}}e^{t\eta(t)\gamma(x_i)+\frac{t\eta(t)^2\Vert \nabla\gamma(x_i)\Vert^2}{4(m+O(\eta(t))}}C(t),
	\end{equation*}
	with $C(t)$ given by
	\begin{equation*}
		C(t):=M\int_{|z|\leq\sqrt t\left(\ep+O(\eta(t))\right)}|z+\sqrt{t}O(\eta(t))|^{\kappa_i}e^{-\frac{m{+O(\eta(t))}}{2}|z|^2}\dd z.
	\end{equation*}
	Now let us recall that  Lemma \ref{lem:teta(t)} ensures that
	\begin{equation*}
		\lim_{t\to\infty} t\eta(t)=\infty.
	\end{equation*}
	Hence one already knows that $\eta(t)\geq 0$ for all $t\gg 1$. Moreover \eqref{upper} ensures that
	\begin{equation*}
		\lim_{t\to\infty} \sqrt t\eta(t)=0,
	\end{equation*}
	so that Lebesgue convergence theorem ensures that
	$$
	C(t)\to C_\infty:=M\int_{\R^N}|z|^{\kappa_i}e^{-\frac{m}{2}|z|^2}\dd z\in (0,\infty)\text{ as }t\to\infty.
	$$
	As a conclusion of the above analysis, we have obtained that there exists 
	some constant $C'$ such that for all $\ep\in (0,\ep_1)$ and all $i=1,..,p$ one has
	\begin{equation}\label{esti1}
		\int_{|x-x_i|\leq\ep}I(t,\dd x)\leq C't^{-\frac{N+\kappa_i}{2}}e^{t\eta(t)\gamma(x_i)},\;\forall t\gg 1.
	\end{equation}
	Since $I(t,\dd x)$ concentrates on $\{\alpha(x)=\alpha^*\}$, then for all $\ep\in (0,\ep_1)$ one has
	\begin{equation*}
		\int_{\R^N}I(t,\dd x)=\sum_{i=1}^p \int_{|x-x_i|\leq\ep}I(t,\dd x)dx+o(1)\text{ as }t\to\infty.
	\end{equation*}
	Using the persistence of $I$ stated in Theorem \ref{thm:measures} (see Lemma \ref{lem:bounds-nomut}), we end-up with
	\begin{equation*}
		0<\liminf_{t\to\infty}\int_{\R^N}I(t,\dd x)\leq \liminf_{t\to\infty}\sum_{i=1}^p \int_{|x-x_i|\leq\ep}I(t,\dd x),
	\end{equation*}
	so that \eqref{esti1} ensures that there exists $c>0$ and $T>0$ such that
	\begin{equation}\label{esti2}
		0<c\leq \sum_{i=1}^p e^{\gamma(x_i)\left(t\eta(t)-\frac{N+\kappa_i}{2\gamma(x_i)}\ln t\right)},\;\forall t\geq T.
	\end{equation}
	Now recalling the definition of $\varrho$ and $J$ in \eqref{def-varrho} and \eqref{def-J}, the upper bound for $\eta(t)$ provided in \eqref{upper} 
	implies 
	$$
	\sum_{i\notin J} e^{\gamma(x_i)\left(t\eta(t)-\frac{N+\kappa_i}{2\gamma(x_i)}\ln t\right)}\to 0\text{ as }t\to\infty,
	$$
	and \eqref{esti2} rewrites as 
	\begin{equation*}
		0<\frac{c}{2}\leq \sum_{i\in J} e^{\gamma(x_i)\left(t\eta(t)-\varrho\ln t\right)},\;\forall t\gg 1.
	\end{equation*}
	This yields
	$$
	\liminf_{t\to\infty}\left(t\eta(t)-\varrho\ln t\right)>-\infty,
	$$
	that is
	\begin{equation}\label{lower}
		\eta(t)\geq \varrho\frac{\ln t}{t}+O\left(\frac{1}{t}\right)\text{ as }t\to\infty.
	\end{equation}
	Then \eqref{expansion} follows coupling \eqref{upper} and \eqref{lower}.

	\noindent{\bf Estimate of the masses:}
	In this last step we turn to the proof of \eqref{mass-esti}. Observe first that
	the upper estimate directly follows from the asymptotic expansion of $\eta(t)$ in \eqref{expansion} together with \eqref{esti1}. Next, the proof for the lower estimate follows from similar inequalities as the one derived in the second step above. 
\end{proof}
\bigskip

\printbibliography

\newpage
\appendix

\begin{center}
	\LARGE \textbf{Appendix}
\end{center}

\section{The case of a unique fitness maximum}
If the function $\alpha(x)$ has a unique global maximum in the support of the initial data, then our analysis leads to a complete description of the asymptotic state of the population.  This may be the unique case when the {behavior} of the orbit is completely known, independently on the positivity of the initial mass of the fitness maximizing set $\{\alpha(x)=\alpha^*\}$.
\begin{theorem}[The case of a unique global maximum]\label{thm:single-max}
	Let Assumption \ref{as:params-nomut} be satisfied.
	{Suppose} that the function $\alpha=\alpha(x)$ has a unique maximum $\alpha^*$ on the support of $I_0$ attained at $x^*\in \supp I_0$, and that 
	\begin{equation*}
		\mathcal R_0(I_0):=\frac{\Lambda}{\theta} \alpha^*>1.
	\end{equation*}
	Then it holds that
	\begin{equation*}
		S(t) \xrightarrow[t\to+\infty]{} \frac{1}{\alpha^*}, \qquad d_0\left(I(t, \dd x),I^\infty \delta_{x^*}(\dd x)\right) \xrightarrow[t\to+\infty]{} 0, 
	\end{equation*}
	where $\delta_{x^*}(\dd x)$ denotes the Dirac measure at $x^*$ and 
	\begin{equation*}
		I^\infty := \frac{\theta}{\alpha^*\gamma(x^*)}(\mathcal R_0(I_0)-1).
	\end{equation*}
\end{theorem}

\section{Existence of a regular metric projection}

In this Section we let $(M, d) $ be a complete metric space. Let $\mathcal K(M)$ be the set of compact subsets in $M$ and let $K\in\mathcal K(M)$. We first recall that we can define a kind of frame of reference, internal to $K$, which allows to identify each point in $K$.

Let us denote $\mathcal K(M)$ the set formed by all compact subsets of $M$. Recall that $(\mathcal K(M), d_H)$ is a complete metric space, where $d_H$ is the Hausdorff distance
\begin{equation*}
	d_H(K_1, K_2)= \max\left(\sup_{x\in K_1} d(x,K_2), \sup_{x\in K_2}d(x, K_1)\right).
\end{equation*}
\begin{proposition}[Metric coordinates]\label{prop:metric-coordinates}
	There exists a finite number of points  $x_1, \ldots, x_n\in K$ with the property that each $ y\in K$ can be identified uniquely by the distance between $y$ and $x_1, \ldots, x_n$. In other words the map 
	\begin{equation*}
		y\overset{c_K}{\longmapsto} \begin{pmatrix} 
			d(y, x_1) \\ \vdots \\ d(y, x_n)
		\end{pmatrix} \in \mathbb R^n_+, 
	\end{equation*}
	is one-to-one. Moreover $c_K$ is continuous and its reciprocal function $c_K^{-1}:c_K(K)\to K$ is also continuous.
\end{proposition}
\begin{proof}
	Let us choose $x_1\in K$ and $x_2\in K$ such that $x_1\neq x_2$. We recursively construct a sequence $x_n$ and a compact set $K_n$ such that 
	\begin{align*}
		K_n &= \{y\in K\,:\,  d(y, x_i) = d(y, x_1) \text{ for all } 1\leq i\leq n\},\\
		x_{n+1} &\in K_n,
	\end{align*}
	the choice of $x_{n+1}$ being arbitrary. Clearly $K_n$ is a compact set and $K_{n+1}\varsubsetneq K_n$. Suppose by contradiction that $K_n\neq\varnothing$ for all $n\in\mathbb N$, then (because $K$ is compact) one can construct a sequence $x_{\varphi(n)}$, extracted from $x_n$, and which converges to a point 
	\begin{equation*}
		x=\lim_{n\to+\infty} x_{\varphi(n)}\in \bigcap_{n\in\mathbb N}K_n=:K_\infty.
	\end{equation*}
	In particular $K_\infty$ is not empty. However we see that, by definition of $K_\infty$, we have  $d(x, x_n)=d(x, x_1)>0$ for all $n\in\mathbb N$, which contradicts the fact that 
	\begin{equation*}
		\lim_{n\to +\infty} d(x, x_{\varphi(n)})=0.
	\end{equation*}
	Hence we have shown by contradiction that there  exists $n_0\in\mathbb N$ such that $K_{n_0}=\varnothing$ and $K_{n_0-1}\neq \varnothing$. This is precisely the injectivity of the map $c_K:K\to \mathbb R^{n_0}$. 

	To show the continuity, we remark that $c_K$ is continuous, and therefore for each closed set $F\subset K$, $F$ is compact so that $c_K(F)$ is compact and therefore closed.  Therefore $(c_K^{-1})^{-1}(F) = c_K(F)$ is closed in $c_K(K)$.
	The proposition is proved. 
\end{proof}
Recall that the Borel $\sigma$-algebra $\mathcal B(M)$ is the closure of the set of all open sets in $\mathfrak P(M)=2^M$ under the operations of complement and countable union. 
A function $\varphi:M\to N$ is Borel measurable if the reciprocal image of any Borel set is Borel, {\it i.e.} $\varphi^{-1}(B)\in\mathcal B(M)$ for all $B\in \mathcal B(N)$.

\begin{proposition}[Borel function of choice]\label{prop:Borel-choice}
	There exists a Borel measurable map $c:\big(\mathcal K(K), d_H\big) \to (K, d)$ such that 
	\begin{equation*}
		c(K')\in K' \text{ for all } K'\in\mathcal K(K).
	\end{equation*}
\end{proposition}
\begin{proof}
	Let $c_K:K\to \mathbb R^{n_0}$ be the map constructed in Proposition \ref{prop:metric-coordinates}. For a compact $K'\subset K$ we define 
	\begin{equation*}
		c(K'):=c_K^{-1}\left(\min_{y\in c_K(K')} y\right), 
	\end{equation*}
	where the minimum is taken with respect to the lexicographical order in $\mathbb R^{n_0}$ (which is a total order and therefore identifies a unique minimum for each $K'\in\mathcal K(K)$). Since the map $\widetilde{K}\subset \mathbb R^{n_0} \to \min_{y\in K} y$ is Borel for the topology on $\mathcal K(\mathbb R^{n_0})$ induced by the Hausdorff metric, so is $c$. The proposition is proved.
\end{proof}

\begin{proposition}[Borel measurability of the metric projection]\label{prop:Hausdorff-projection}
	Let $K\subset M$ be compact. The  map $P_K:M\to \mathcal K(K)$ defined by 
	\begin{equation*}
		P_K(x)=\{y\in K\, :\, d(x, y) = d(x, K)\}, 
	\end{equation*}
	is Borel measurable.
\end{proposition}
\begin{proof}
	First we remark that the map 
	\begin{equation*}
		P_K(x):=\{y\in K\, :\, d(x, y) = d(x, K)\}\in \mathcal K(M), 
	\end{equation*}
	is well-defined for each $x\in M$, and therefore forms a mapping from $M$ into $\mathcal K(K)\subset \mathcal K(M)$. Indeed $P_K(x)$ is clearly closed in the compact space $K$, therefore is compact.

	To show the Borel measurability of $P_K$, we first remark that, given a compact space $K'\subset K$, the set
	\begin{equation*}
		\widetilde{P_K^{-1}}(K'):=\{x\in M\,:\, P_K(x)\cap K'\neq \varnothing\}
	\end{equation*}
	is closed. Indeed let $x_n\to x$ be a sequence in $\widetilde{P_K^{-1}}(K')$, then by definition there exists $y_n\in K'$ such that $d(x_n, y_n)=d(x_n, K)$. By the compactness of $K'$, there exists $y\in K'$ and a subsequence $y_{\varphi(n)}$ extracted from $y_n$ such that $y_{\varphi(n)}\to y$. Because of the continuity of $z\mapsto d(z, K)$, we have
	\begin{equation*}
		d(x, y) =\lim_{n\to +\infty} d(x_{\varphi(n)}, y_{\varphi(n)}) = \lim_{n\to +\infty}d(x_{\varphi(n)}, K) =  d(x, K), 
	\end{equation*}
	therefore $y\in P_K(x)\cap K'$, which shows that $x\in\widetilde{P_K^{-1}}(K')$. Hence $\widetilde{P_K^{-1}}(K')$ is closed.

	We are now in a position to show the Borel regularity of $P_K$. Let $C\in \mathcal K(K)$ and $R>0$ be given. We define $B_H(C, R)$ the ball of center $C$ and radius $R$ in the Hausdorff metric: 
	\begin{equation*}
		B_H(C, R) = \{ C'\in \mathcal K(K)\,:\, d_H(C, C')\leq R\}.
	\end{equation*}
	Then
	\begin{equation*}
		P_K^{-1}(B_H(C, R)) = \{ x\in M\,:\, d_H(P_K(x), C) \leq R \}=B_1\cap B_2,
	\end{equation*}
	where 
	\begin{align*}
		B_1&:= \{ x\in M\,:\, d(y, C)\leq R \text{ for all } y\in P_K(x)\},  \text{ and } \\ 
		B_2&:= \{ x\in M\,:\, d(z, P_K(x))\leq R \text{ for all } z\in C\}.
	\end{align*}
	It can be readily seen that $B_1$ is a Borel set by writing
	\begin{equation*}
		B_1=\widetilde{P_K^{-1}}(V_R(C)) \bigcap_{n\geq 1} \left(M\backslash \left(\widetilde{P_K^{-1}}(K\backslash V_{R+\frac{1}{n}}(C))\right)\right), 
	\end{equation*}
	where $V_R(C):=\{y\in K\,:\, d(y, C)\leq R\}$. To see that $B_2$ is a Borel set, we choose a 
	{sequence $z_n$ which is dense in $C$ and write
	\begin{equation*}
		B_2 = \bigcap_{k\geq 1}\bigcap_{n\geq 1} \widetilde{P_K^{-1}}\big(B(z_n, R+1/k)\big).
	\end{equation*}
	Indeed if $x\in B_2$ then $P_K(x)$ intersects every ball of radius $R$ and center $z\in C$; in particular $P_K(x)$ intersects every ball of radius $R+1/k$ and center $z_n$. Conversely suppose that $P_K(x)$ intersects every ball $B(z_n, R+1/k)$ for $n\geq 1$ and $k\geq 1$. If $z\in C$ then there is a sequence $z_{\varphi(k)}$ such that $z=\lim z_{\varphi(k)}$, and (by assumption) we have $P_K(x)\cap B(z_{\varphi(k)}, R+1/k)\neq \varnothing$. Therefore 
	\begin{equation*}
		d\big(z, P_K(x)\big) = \lim_{k\to+\infty}d\big(z_{\varphi(k)}, P_K(x)\big)\leq \lim_{k\to+\infty}R+\frac{1}{k} = R. 
	\end{equation*}
	Thus $x\in B_2$. The equality is proved.
	}

	We conclude that $ P_K^{-1}(B_H(C, R))$ is a Borel set for all $C\in \mathcal K(K)$ and $R>0$, and since those sets form a basis of the Borel $\sigma$-algebra, $P_K$ is indeed Borel measurable. The Lemma is proved. 
\end{proof}

\begin{theorem}[Existence of a regular metric projection]\label{thm:regular-metric-projection}
	Let $K\subset M$ be compact. There exists a Borel measurable map $P_K:M\to K$ such that 
	\begin{equation*}
		d\big(x, P_K(x)\big) = d(x, K) .
	\end{equation*}
\end{theorem}
\begin{proof}
	The proof is immediate by combining Proposition \ref{prop:Hausdorff-projection} Proposition \ref{prop:Borel-choice}.
\end{proof}
\begin{proposition}[Metric projection on measure spaces]\label{prop:est-d0-supp}
	Let $K\in \mathcal K(M)$ be a given compact set. Let $\mu\in \mathcal M_+(M)$ be a given nonnegative Borel measure on $M$. Then  the Kantorovitch-Rubinstein distance between $\mu $ and $\mathcal M_+(K)$ can be bounded by the distance between $K$ and the furthest point in $\supp \mu$:
	\begin{equation*}
		d_0(\mu, \mathcal M_+(K))\leq \Vert \mu\Vert_{TV} \sup_{x\in\supp \mu}d(x, K).
	\end{equation*}
\end{proposition}
\begin{proof}
	Indeed, let us choose a Borel measurable metric projection $P_K$ on $K$ as in Theorem \ref{thm:regular-metric-projection}. Let $\mu^K$ be the image measure defined on $\mathcal B(K)$ by 
	\begin{equation*}
		\mu^K(B) := \mu\big(P_K^{-1}(B)\big), \text{ for all } B\in\mathcal B(K).
	\end{equation*}
	Then in particular for all $f\in BC(M) $ we have 
	\begin{equation*}
		\int_{\mathbb R^N} f(P_K(x))\dd \mu(x) = \int_{K}f(x) \dd \mu^K(x). 
	\end{equation*}
	Let $f\in\mathrm{Lip}_1(M)$, then we have
	\begin{align*}
		\int_{\mathbb R^N} f(x) \dd(\mu-\mu^K)(x) &=\int_{\mathbb R^N} f(x) \dd\mu(x) - \int_{\mathbb R^N}f(x)\dd \mu^K(x) \\  
		&= \int_{\mathbb R^N} f(x) \dd\mu(x) - \int_{\mathbb R^N}f(P_K(x))\dd \mu(x)\\
		&=\int_{\mathbb R^N} \big(f(x)-f(P_K(x))\big)\dd \mu(x)\\
		&\leq \int_{\supp \mu} |x-P_K(x)|\dd \mu(x) \\ 
		&\leq\sup_{y\in\supp \mu} d(y, K) \int_{\supp \mu}1\dd \mu =\Vert \mu\Vert_{TV}\sup_{x\in\supp K}d(x, K).
	\end{align*}
	Therefore $d_0(\mu, \mu^K)\leq \Vert \mu\Vert_{TV}\sup_{x\in\supp K}d(x, K) $ and, since $\mu^K\in\mathcal M_+(K)$, 
	\begin{equation*}
		d_0\big(\mu, \mathcal M_+(K)\big)\leq d_0(\mu, \mu^K) \leq \Vert \mu\Vert_{TV}\sup_{x\in \supp \mu} d(x, K).
	\end{equation*}
	The Proposition is proved.
\end{proof}

\section{Disintegration of measures}
\label{sec:disintegration}

\subsection{Bourbaki's disintegration theorem}

We recall the disintegration theorem as stated in \cite[VI, \S 3, Theorem 1 p. 418]{Bour-Int}. We use Bourbaki's version, which is proved by functional analytic arguments,  for convenience, although other approaches exist which are based on measure-theoretic arguments and  may be deemed more intuitive. We refer to Ionescu Tulcea and Ionescu Tulcea for a disintegration theorem resulting from the theory of (strong) liftings \cite{Ion-Ion-64,Ion-Ion-69}.   

Let us first we recall some background on adequate families. This is adapted from \cite[V.16 \S3]{Bour-Int} to the context of finite measures of $\mathbb R^N$. We let $T$ and  $X$ be locally compact topological spaces and  $\mu\in\mathcal M_+(T)$  be a fixed Borel measure.
\begin{definition}[Scalarly essentially integrable family]
	Let $\Lambda:t\mapsto \lambda_t$ be a mapping from $T$ into $\mathcal M_+(X)$. $\Lambda$ is \textit{scalarly essentially integrable} for the measure $\mu$ if for every compactly supported continuous function $f\in C_c(X)$, the function $t\mapsto \int_X f(x) \lambda_t(\dd x)$ is in $L^1(\mu)$. Setting $\nu(f)=\int_{T} \int_{X}f(x)\lambda_t(\dd x)\mu(\dd t)$ defines a linear form on $C_c(X)$, hence a measure $\nu$, which is the \textit{integral} of the family $\Lambda$, and we denote
	\begin{equation*}
		\int_{T} \lambda_t\,\mu(\dd t) := \nu.
	\end{equation*}
\end{definition}
Recall that every positive Borel measure $\mu$ on a locally compact space $X$ defines a positive bounded linear functional on $C_c(X)$ equipped with the inductive limit of the topologies on $C_c(K)$ when $K$ runs over the compact subsets of $X$. Conversely if $\mu$ is a positive bounded linear functional on $C_c(X)$, there are two canonical ways to define a measure on the Borel $\sigma$-algebra.
\begin{enumerate}
	\item Outer-regular construction. Let $U\subset X$ be a open, then one can define
		\begin{equation*}
			\mu^*(U):=\sup \left\{\mu(f)\,: f\in C_c(X), 0\leq f(x) \leq \mathbbm{1}_U(x)\right\}, 
		\end{equation*}
		then for an arbitrary Borel set $B$, 
		\begin{equation*}
			\mu^*(B):=\inf\left\{\mu^*(U)\, : U \text{ open, }B\subset U\right\}.
		\end{equation*}
		This notion corresponds to that of the \textit{upper integral} discussed in \cite[IV.1 \S1]{Bour-Int}.

	\item Inner-regular construction. If $U\subset X$ is open, we define $\mu^\bullet(U):=\mu^*(U)$ and similarly if $K\subset X$ is compact, then $\mu^\bullet(K):=\mu^*(K)$. Then for an arbitrary Borel set $B$ which is contained in an open set of finite measure: $B\subset U$ with $\mu^\bullet(U)<+\infty$, we define 
		\begin{equation*}
			\mu^\bullet(B):=\sup\{\mu^\bullet(K)\,:\,K\text{ compact, } K\subset B\}.
		\end{equation*}
		Else $\mu^\bullet(B)=+\infty$.
		This corresponds to the \textit{essential upper integral} discussed in \cite[V.1, \S 1]{Bour-Int}.
\end{enumerate}
It is always true that $\mu^\bullet\leq \mu^*$, however it may happen that $\mu^*\neq \mu^\bullet$ when $\mu^*$ is not finite, see e.g. \cite[II\S7.11 p.113]{Bog-07} or \cite[V.1, \S 1]{Bour-Int}. If $\mu $ is a Borel measure, then we define the corresponding notions of $\mu^\bullet $ and $\mu^*$ associated with the linear functional $f\mapsto \int_{X}f(x) \mu(\dd x)$. Note that if $\mu$ is Radon, then $\mu^*=\mu=\mu^\bullet$.
\begin{definition}[Pre-adequate and adequate families]
	We follow \cite[Definition 1, V.17\S3]{Bour-Int}.
	Let $\Lambda:t\mapsto\lambda_t$ be a scalarly essentially $\mu$-integrable mapping from $T$ into $\mathcal M_+(X)$, $\nu$ the integral of $\Lambda$.

	We say that $ \Lambda$ is $\mu$-pre-adequate if, for every lower semi-continuous function $f\geq 0$ defined on $X$, the function $t\mapsto\int f(x)\lambda_t^\bullet(\dd x)$ is $\mu$-measurable on $T$ and 
	\begin{equation*}
		\int_{X} f(x) \nu^\bullet(\dd x)= \int_{T} \int_{X}f(x)\lambda_t^\bullet(\dd x) \mu^\bullet(\dd t).
	\end{equation*}
	We say that $\Lambda $ is $\mu$-adequate if $\Lambda$ is $\mu'$-pre-adequate for every positive Borel measure $\mu'\leq \mu$.
\end{definition}
The last notion we need to define is the one of \textit{$\mu$-proper function}.
\begin{definition}[$\mu$-proper function]
	We say that a function $p:T\to X$ is \textit{$\mu$-proper} if it is $\mu$-measurable and, for every compact set $K\subset X$, the set $p^{-1}(K) $ is $\mu^\bullet$-measurable and $\mu^\bullet(p^{-1}(K))<+\infty$.
\end{definition}
If $\mu$ is Radon, in particular, then every $\mu$-measurable mapping $p:T\to X$ ($X$ being equipped with the Borel $\sigma$-algebra) is $\mu$-proper.
The following Theorem is taken from \cite[Theorem 1, VI.41 No.1, \S3]{Bour-Int}.
\begin{theorem}[Disintegration of measures]\label{thm:disintegration}
	Let $T$ and $X$ be two locally compact spaces having countable bases, $\mu$ be a positive measure on $T$, $p$ be a $\mu$-proper mapping of $T$ into $X$, and $\nu=p(\mu)$ the image of $\mu$ under $p$. There exists a $\nu$-adequate family $x\mapsto \lambda_x$ ($x\in X$) of positive measures on $T$, having the following properties:
	\begin{enumerate}[label={\rm\alph*)}]
		\item $\Vert \lambda_x\Vert = 1$ for all $x\in p(T)$;
		\item $\lambda_x$ is concentrated on the set $p^{-1}(\{x\})$ for all $x\in p(T)$, and $\lambda_x=0$ for $x\not\in p(T)$;
		\item $\mu=\int \lambda_x\, \nu(\dd x)$.
	\end{enumerate}
	Moreover, if $x\mapsto \lambda_x'$ ($x\in X$) is a second $\nu$-adequate family of positive measures on $T$ having the properties {\rm b)} and {\rm c)}, then $\lambda_x'=\lambda_x$ almost everywhere in $B$ with respect to the measure $\nu$.
\end{theorem}

\newpage

\begin{figure}
	\begin{center}
		\includegraphics[width=3in]{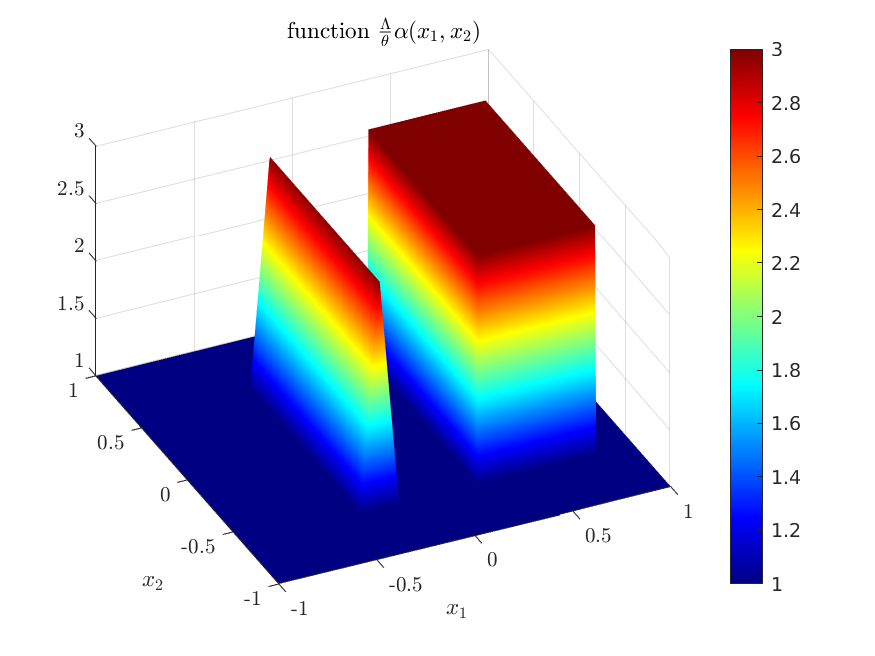}
		\includegraphics[width=3in]{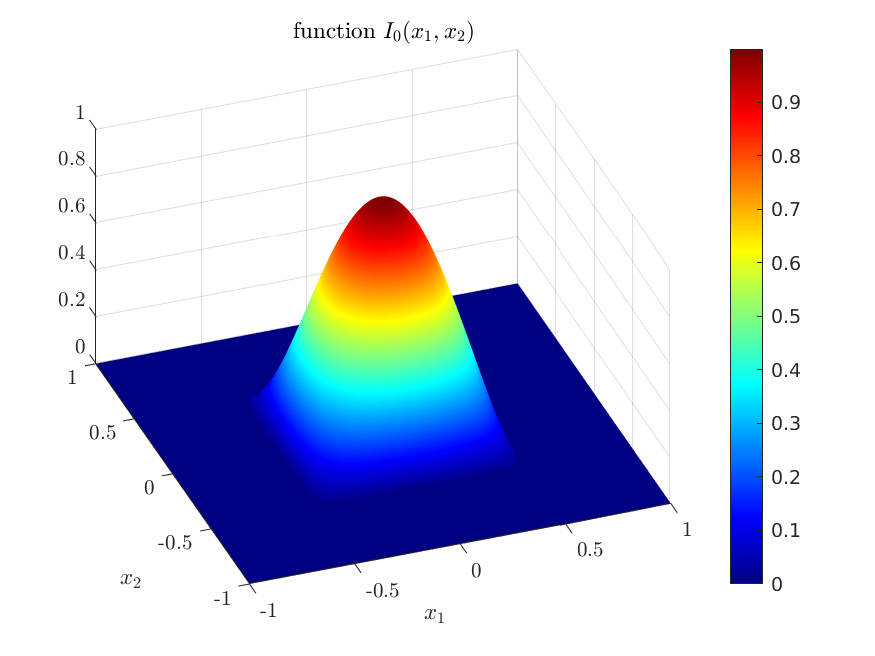}
		\includegraphics[width=3in]{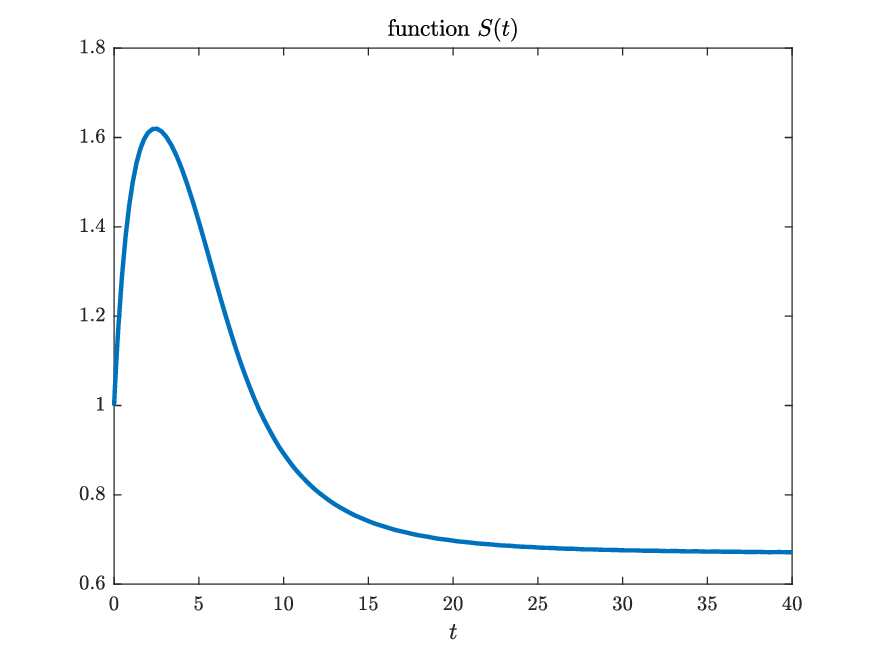}
		\includegraphics[width=3in]{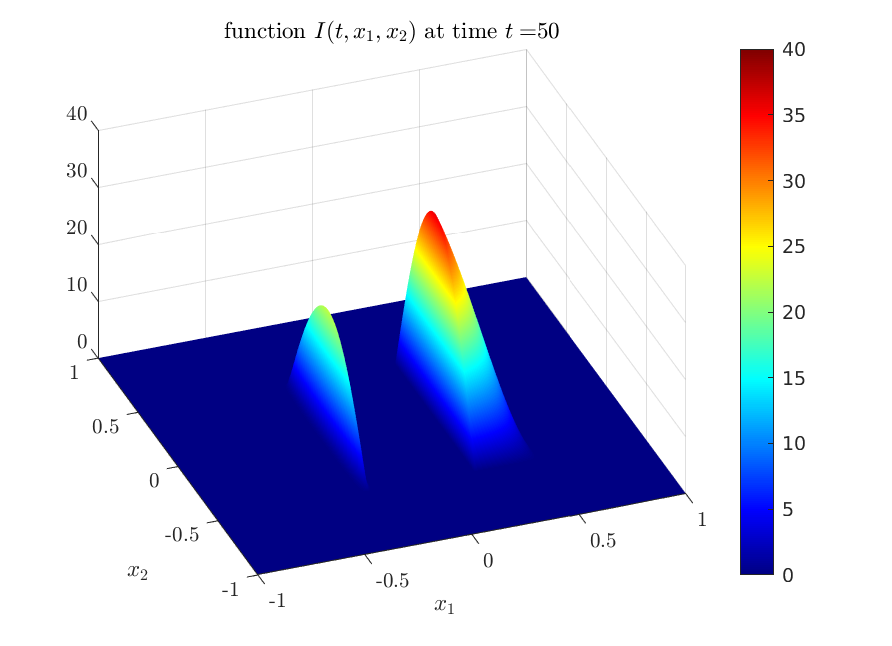}
	\end{center}
	\caption{Illustration of Theorem \ref{thm:measures} in the case i), {\it i.e.}, when $I_0\big(\{\alpha(x)=\alpha^*\}\big)>0$. 
	Parameters of this simulation are: $\Lambda=2$, $\theta=1$, 
	$\alpha(x)=0.5+\left( \mathbbm{T}_{[-0.4,-0.2]}(x_1) +  \mathbbm{1}_{[0.2,0.8]}(x_1)\right) \mathbbm{1}_{[-0.6,0.6]}(x_2)$ where $x=(x_1,x_2) \in \R^2$ and $\mathbbm{T}_{[-0.4,-0.2]}$ is the triangular function of height one and support $[-0.4,-0.2]$, and $\gamma= \frac{1}{2\alpha}$. Initial condition is given by $I_0(\dd x)=I_0(x_1,x_2)  \, \dd x$ where $I_0(x_1,x_2)=\mathbbm{1}_{[-0.5,0.5]}(x_1)\cos(\pi x_1)\mathbbm{1}_{[-0.5,0.5]}(x_2)\cos(\pi x_2)$.
	In particular,   we have  $\alpha^*=3/2$ and $\{\alpha(x)=\alpha^*\}=\left(\{-0.3\} \cup [0.2,0.5] \right) \times[-0.5,0.5]$. 
	Function $t \to S(t)$ converges towards $1/\alpha^*=2/3$  
	and function $x\to I(t,x)$ at time $t=50$ is asymptotically concentrated on  $\{\alpha(x)=\alpha^*\}=\left(\{-0.3\} \cup [0.2,0.5] \right) \times[-0.5,0.5]$. 
	\label{Fig1}}
\end{figure}

\begin{figure}
	\begin{center}
		\includegraphics[width=3in]{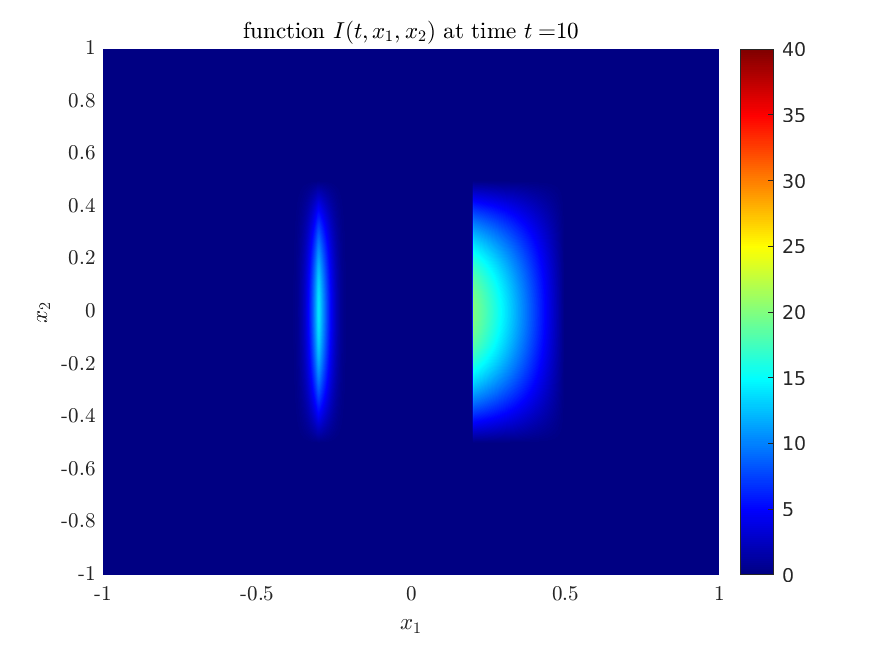} 
		\includegraphics[width=3in]{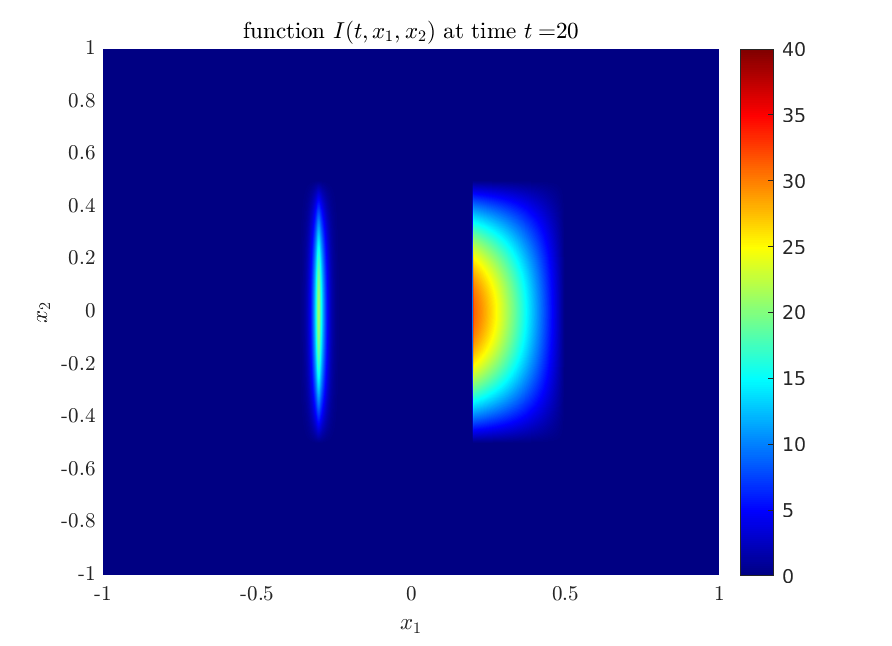} 
		\includegraphics[width=3in]{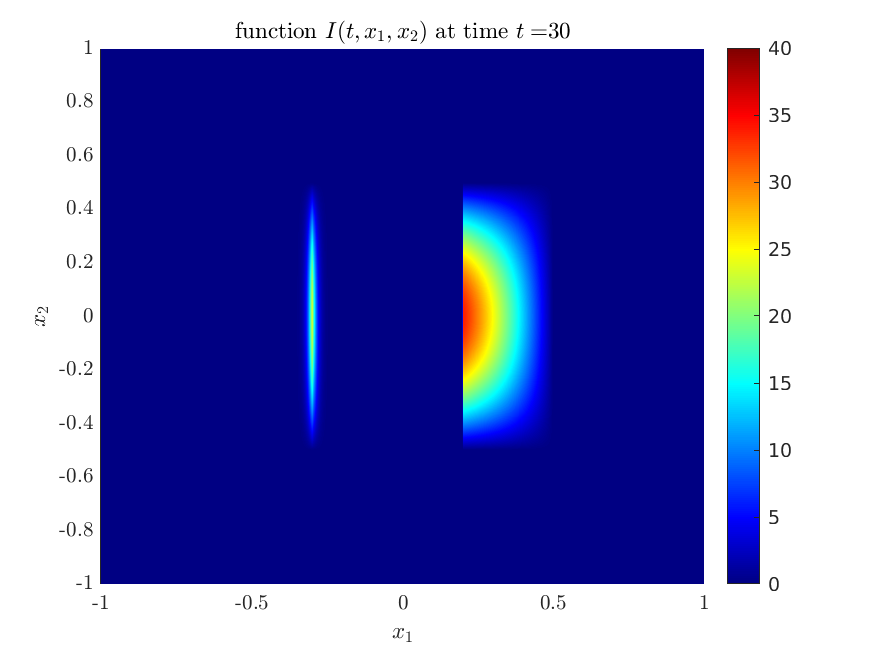} 
		\includegraphics[width=3in]{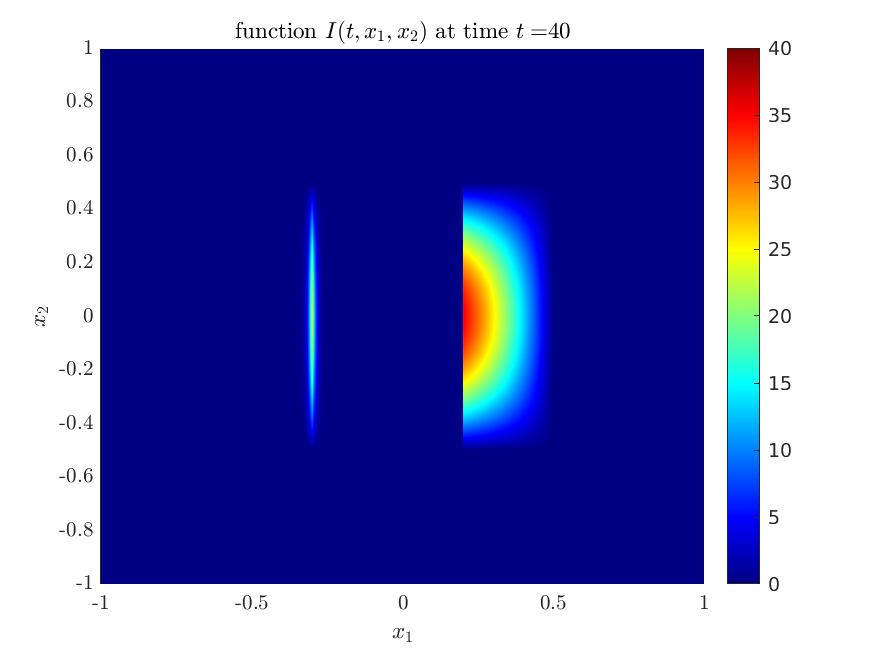} 
	\end{center}
	\caption{Illustration of Theorem \ref{thm:measures} in the case i), {\it i.e.}, when $I_0\big(\{\alpha(x)=\alpha^*\}\big)>0$.  
	Function $x\to I(t,x)$ at time $t=10,20,30$ and $40$. The function $I$ remains bounded in this case.
	\label{Fig3}}
\end{figure}
\begin{figure}
	\begin{center}
		\includegraphics[width=3in]{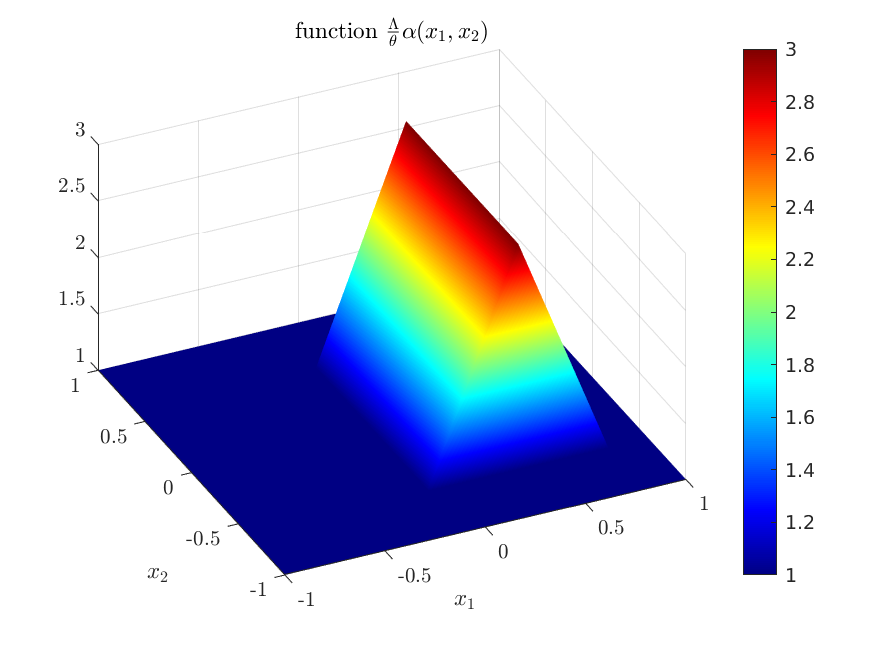}
		\includegraphics[width=3in]{Figures/I0_3D.eps} \\ 
		\includegraphics[width=3in]{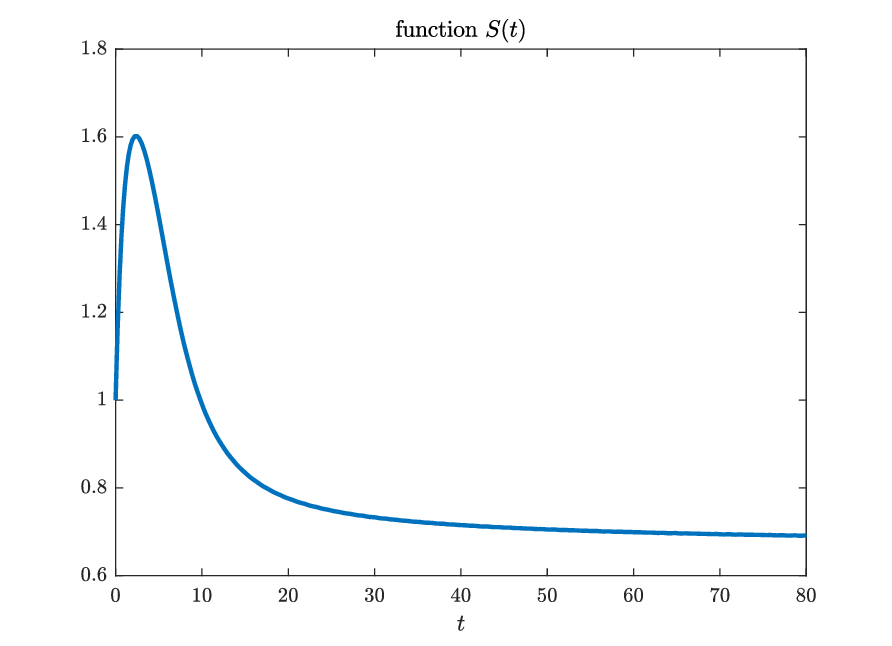}
		\includegraphics[height=2.2in]{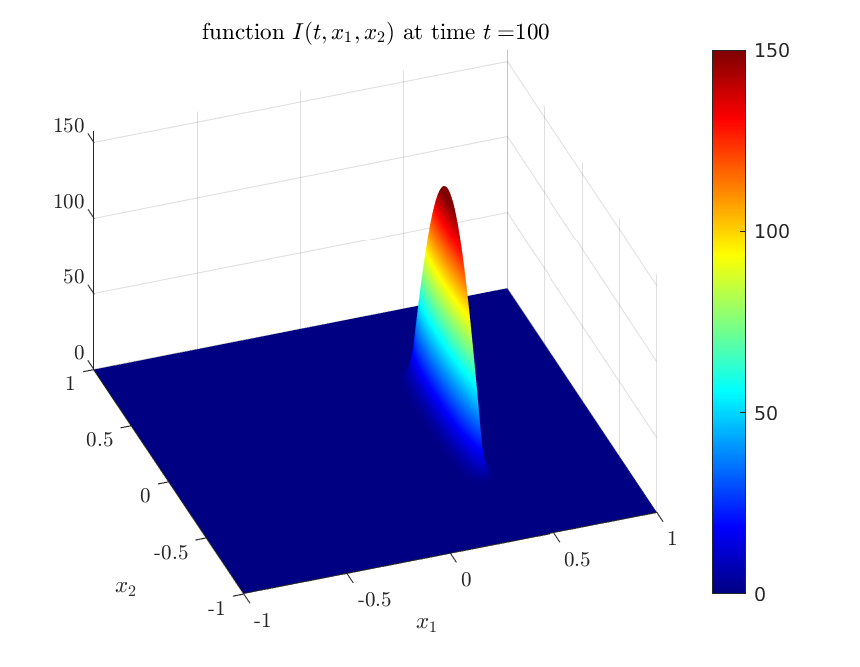}
	\end{center}
	\caption{Illustration of Theorem \ref{thm:measures} in the case ii), {\it i.e.}, when $\int_{\alpha(x)=\alpha^*}I_0(\dd x)=0$.
	Parameters of this simulation are: $\Lambda=2$, $\theta=1$, 
	$\alpha(x)=0.5+\mathbbm{T}_{[-0.1,0.8]}(x_1) \mathbbm{1}_{[-0.5,0.5]}(x_2)$  where $x=(x_1,x_2) \in \R^2$ and $\mathbbm{T}_{[-0.1,0.5]}$ is the triangular function of height one and support $[-0.1,0.5]$, $\gamma= \frac{1}{2\alpha}$.   Initial condition is given by $I_0(\dd x)=I_0(x_1,x_2)  \, \dd x$ where $I_0(x)=\mathbbm{1}_{[-0.5,0.5]}(x_1)\cos(\pi x_1)\mathbbm{1}_{[-0.5,0.5]}(x_2)\cos(\pi x_2)$.
	In particular, $\alpha^*=3/2$ and $\{\alpha(x)=\alpha^*\}=\{0.35\} \times[-0.5,0.5]$.
	The function $t \to S(t)$ converges towards $1/\alpha^*=2/3$. 
	The function $x\to I(t,x)$ at time $t=100$ is asymptotically concentrated on  $\{\alpha(x)=\alpha^*\}=\{0.35\} \times[-0.5,0.5]$.
	\label{Fig4}}
\end{figure}
\begin{figure}
	\begin{center}
		\includegraphics[height=2.2in]{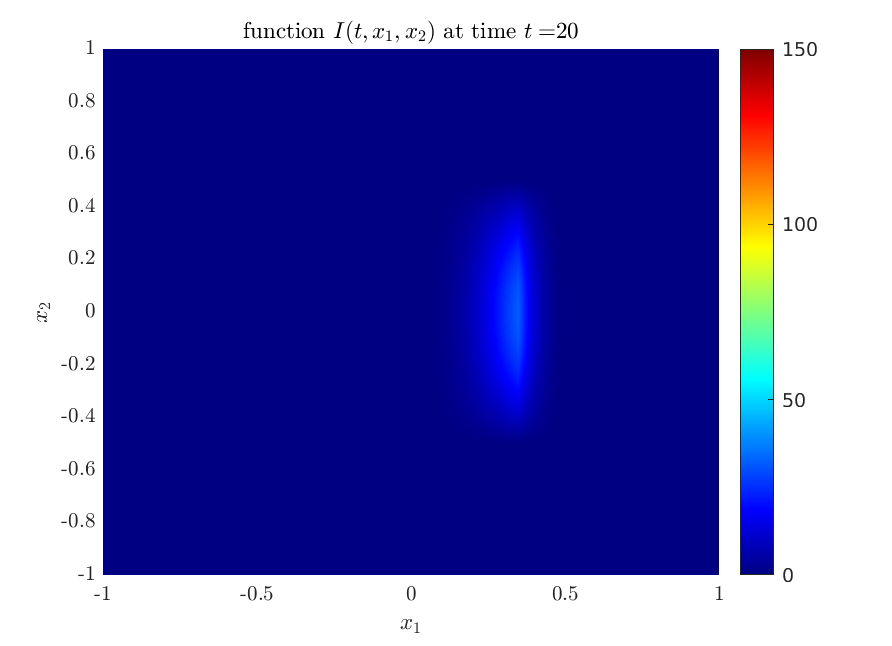} 
		\includegraphics[height=2.2in]{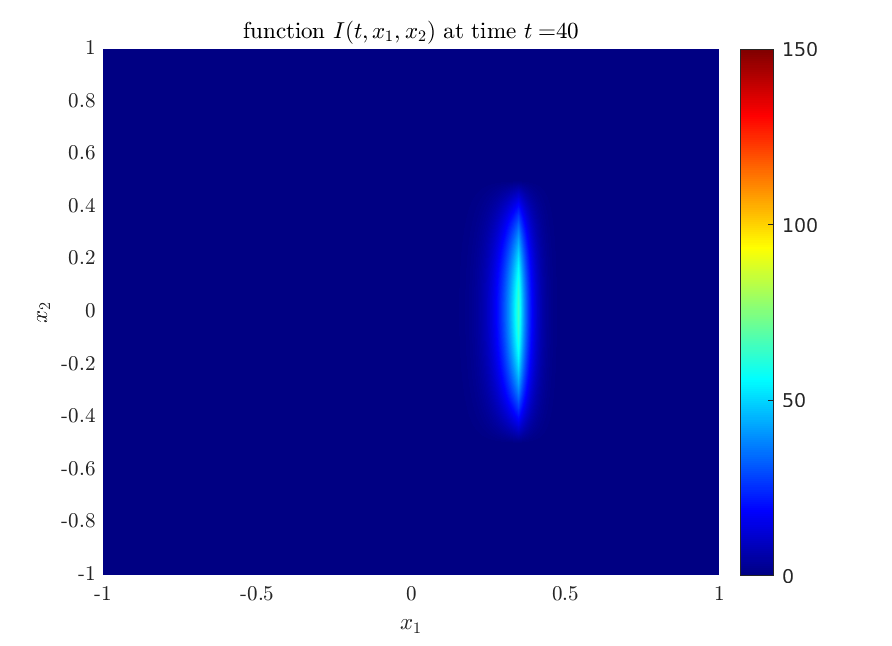} 
		\includegraphics[height=2.2in]{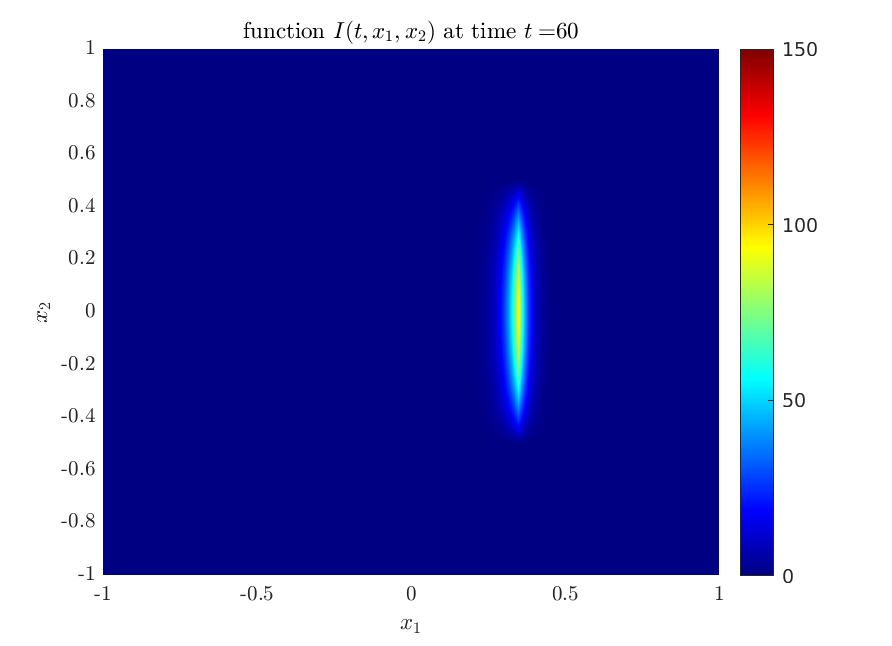} 
		\includegraphics[height=2.2in]{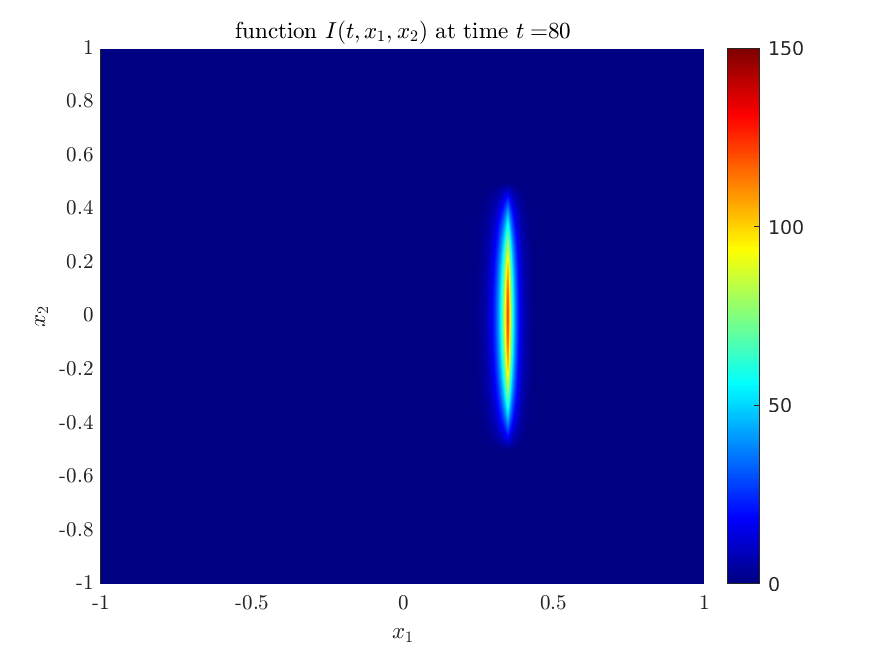} 
	\end{center}
	\caption{  Illustration of Theorem \ref{thm:measures} in the case ii), {\it i.e.}, when $I_0\big(\{\alpha(x)=\alpha^*\}\big)=0$. 
	Function $x\to I(t,x)$ at time $t=20,40,60$ and $80$. The function $I$ asymptotically converges towards a singular measure. 
	\label{Fig5}}
\end{figure}

\begin{figure}
	\centering
	\begin{minipage}{0.48\textwidth}
		\includegraphics[width=\textwidth]{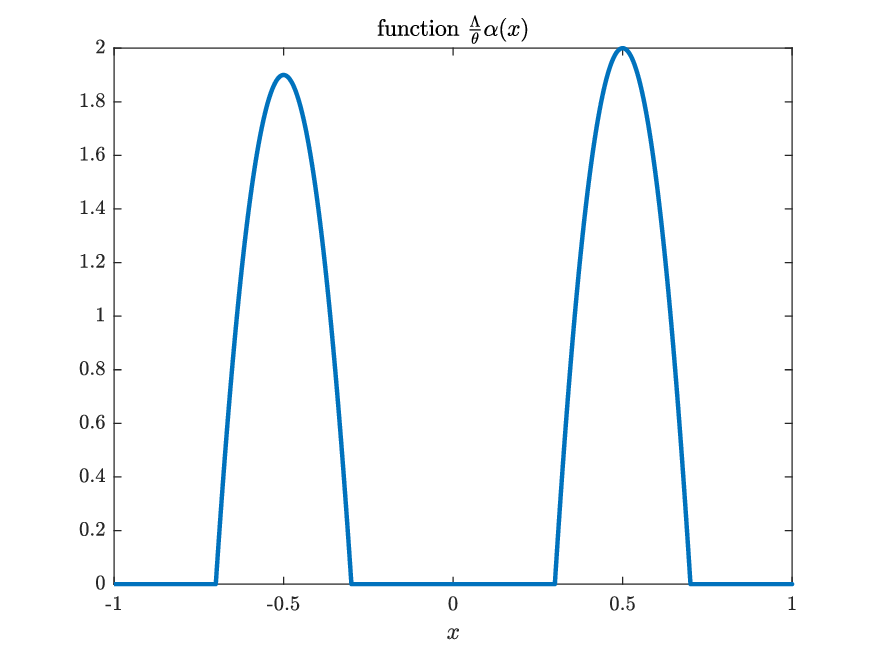}
	\end{minipage}
	\begin{minipage}{0.48\textwidth}
		\includegraphics[width=\textwidth]{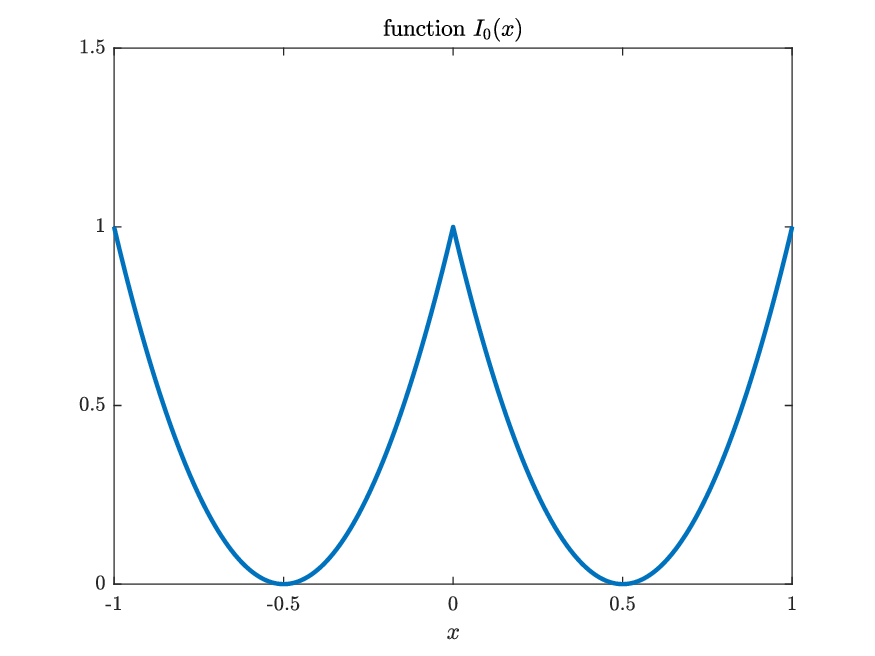}
	\end{minipage}

	\begin{minipage}{0.48\textwidth}
		\includegraphics[width=\textwidth]{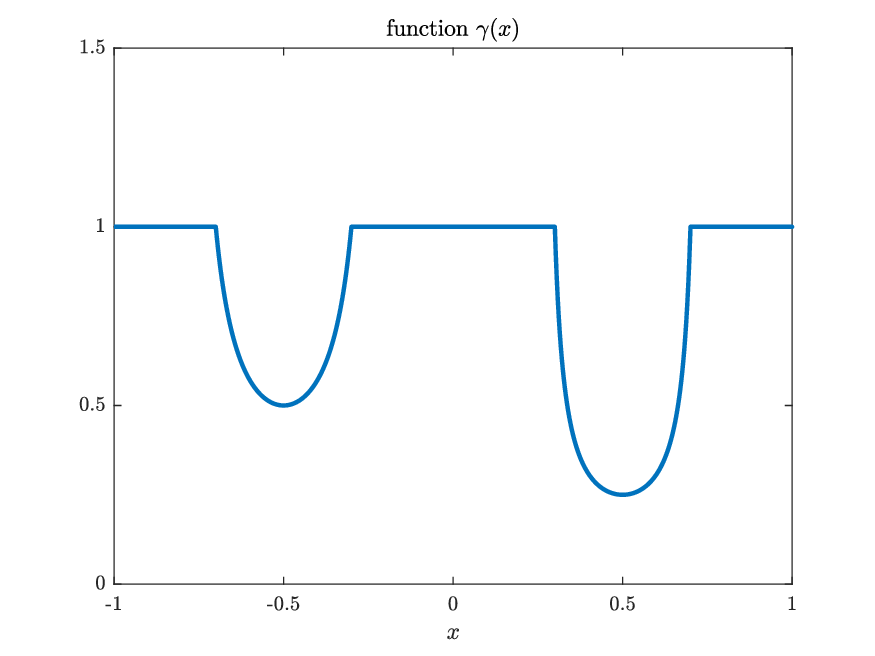}
	\end{minipage}
	\begin{minipage}{0.48\textwidth}
		\includegraphics[width=\textwidth]{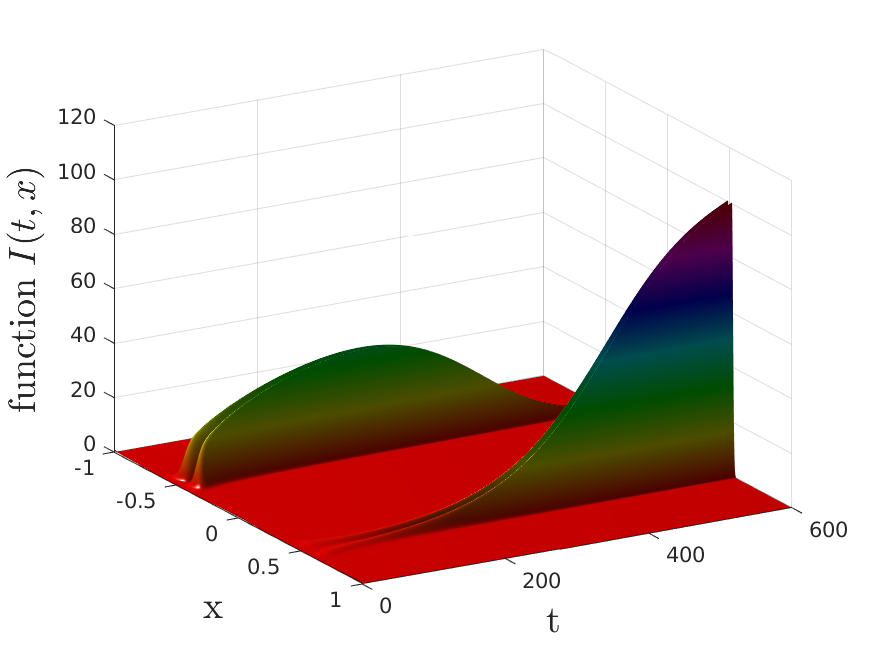}
	\end{minipage}
	\caption{Illustration of a transient behavior for \eqref{eq:SI-no-mut}.
	Parameters of this simulation are: $\Lambda=2$, $\theta=1$, $\alpha(x) $ is given by \eqref{eq:example-alpha-transient}, $I_0(x)$  by \eqref{eq:example-I_0-transient} and
	$\gamma(x) $  by \eqref{eq:example-gamma-transient}. In particular, $\alpha^*=1$, $\{\alpha(x)=\alpha^*\}=\{x_2\}$ with $x_1=-0.5$, $x_2=0.5$.
	{
		Other parameters are $\kappa_1=2$, $\kappa_2=2$, $\gamma(x_1)=1/2$,  $\gamma(x_2)=1/4$. The value of the local maximum at $x_1$, $\alpha(x_1)=0.95$,  being very close to $\alpha^*$, observe that the distribution $I(t,x)$  first concentrates around $x_1$ before the global maximum $x_2$ becomes dominant (bottom right plot).}
	}    \label{Fig8}
\end{figure}
\begin{figure}
	\centering
	\begin{minipage}{0.48\textwidth}
		\includegraphics[width=\textwidth]{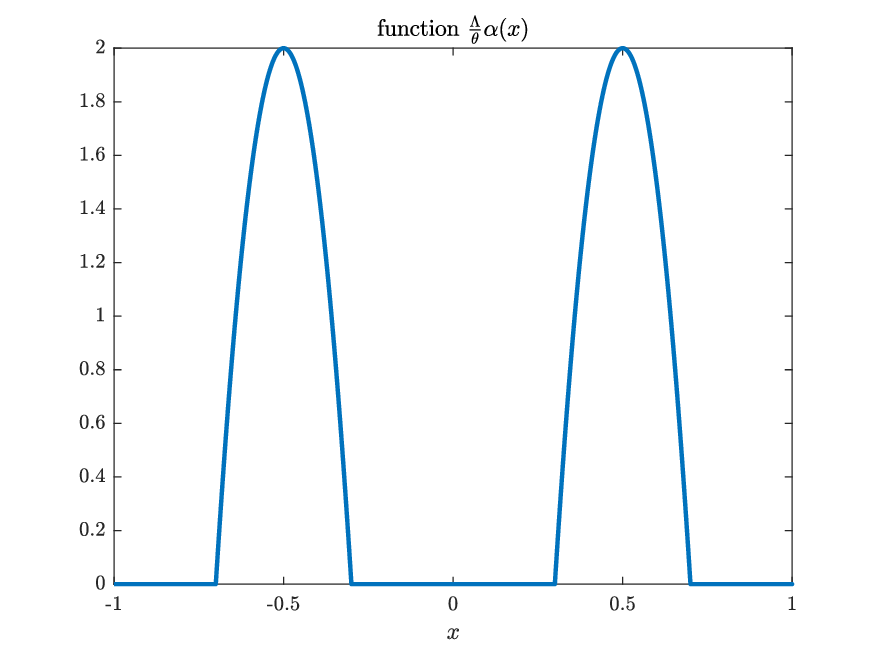}
	\end{minipage}
	\begin{minipage}{0.48\textwidth}
		\includegraphics[width=\textwidth]{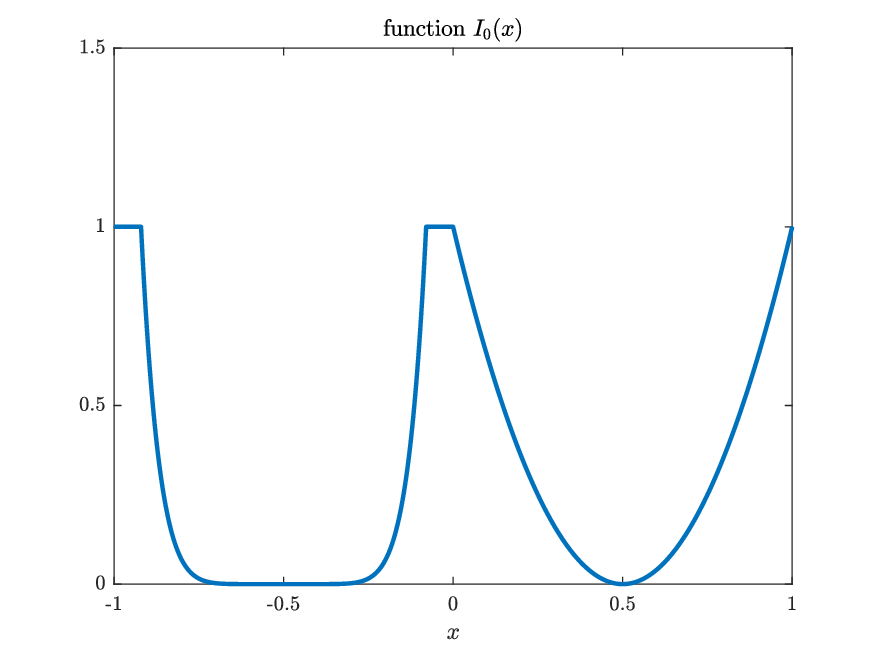}
	\end{minipage}

	\begin{minipage}{0.48\textwidth}
		\includegraphics[width=\textwidth]{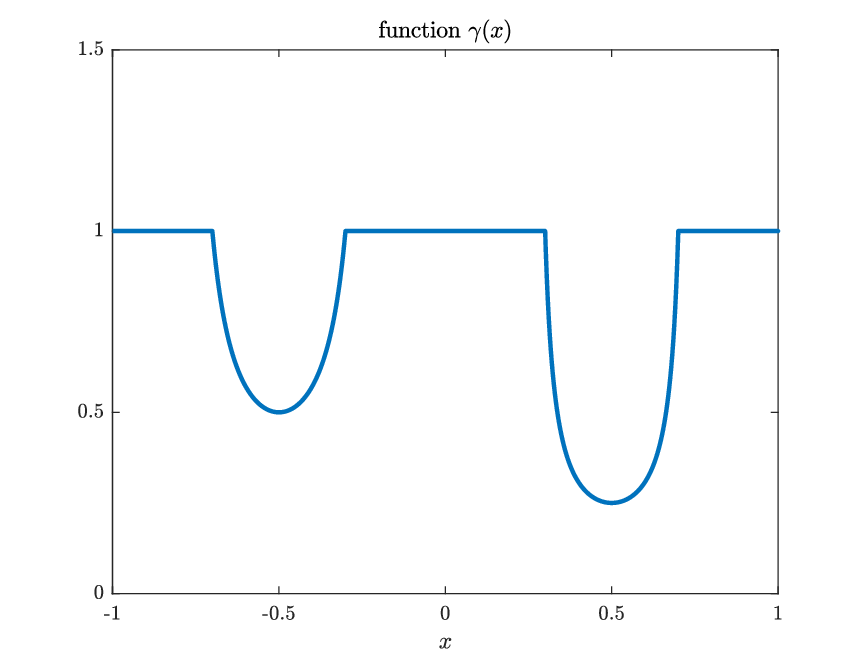}
	\end{minipage}
	\begin{minipage}{0.48\textwidth}
		\includegraphics[width=\textwidth]{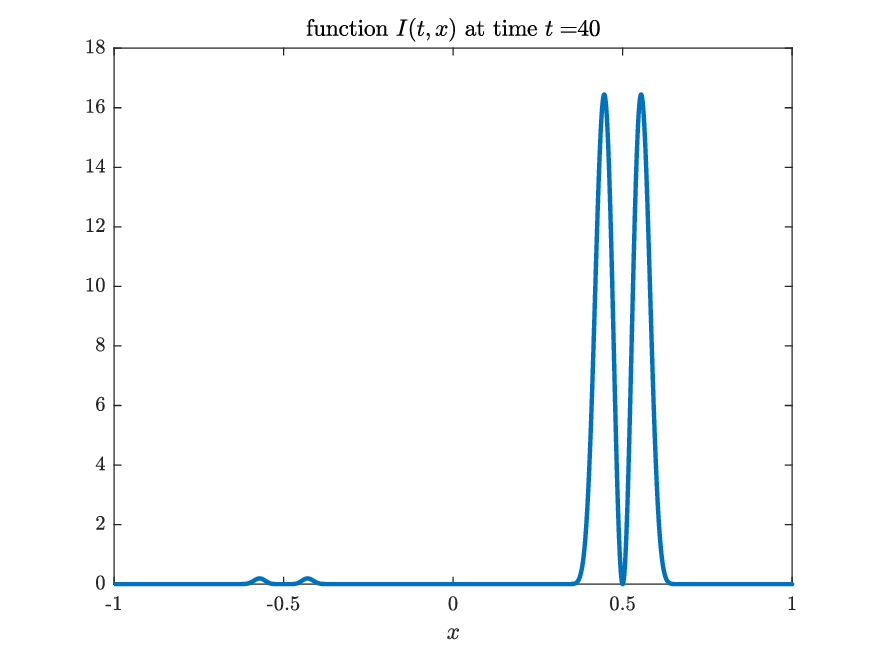}
	\end{minipage}

	\caption{Illustration of Theorem \ref{THEO-eta} and Corollary \ref{CORO-self-similar}. 
	Parameters of this simulation are: $\Lambda=2$, $\theta=1$, $\alpha(x) $ is given by \eqref{eq:example-alpha}, $I_0(x)$  by \eqref{eq:example-I_0} and
	$\gamma(x) $  by \eqref{eq:example-gamma}. In particular, $\alpha^*=1$, $\{\alpha(x)=\alpha^*\}=\{x_1,x_2\}$ with $x_1=-0.5$, $x_2=0.5$, $\kappa_1=8$, $\kappa_2=2$, $\gamma(x_1)=1/2$,  $\gamma(x_2)=1/4$, $\rho=6$ and $J=\{2\}$. The initial condition $I_0$ vanishes more rapidly around $x_1$ than  $x_2$ so that the solution $I(t,x)$ vanishes around $x_1$ as $t$ goes to $\infty$, even though $\gamma(x_1)> \gamma(x_2)$, while around $x_2$ it takes the shape given by expression \eqref{CORO-self-similar}.} 
	\label{Fig7}
\end{figure}
\begin{figure}
	\noindent\includegraphics[width=6cm]{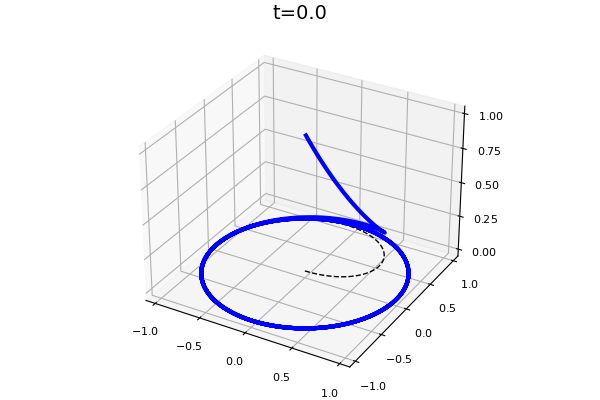}
	\includegraphics[width=6cm]{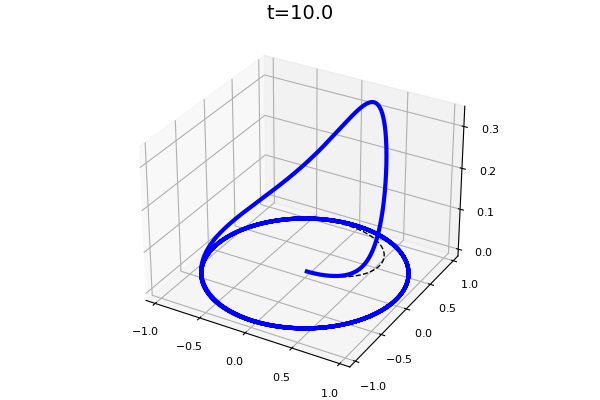}
	\includegraphics[width=6cm]{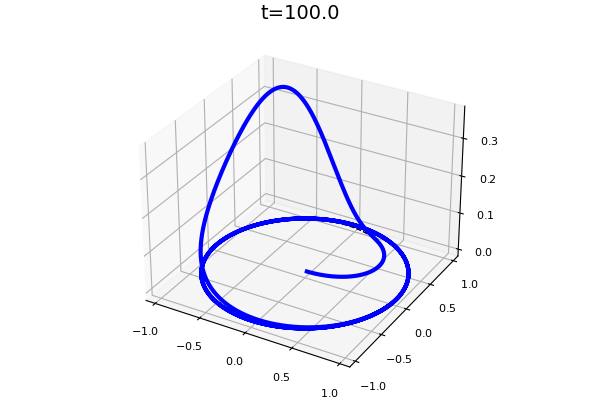}\bigskip\\
	\includegraphics[width=6cm]{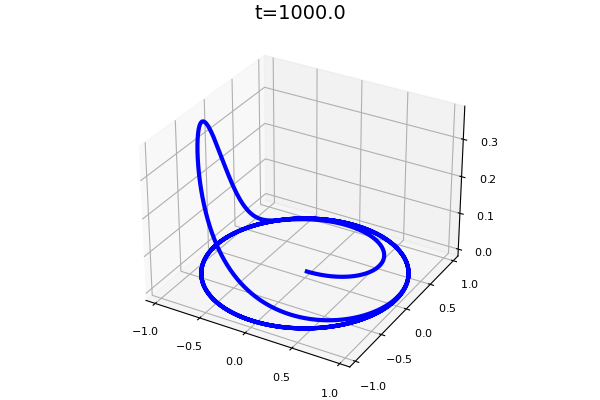}
	\includegraphics[width=6cm]{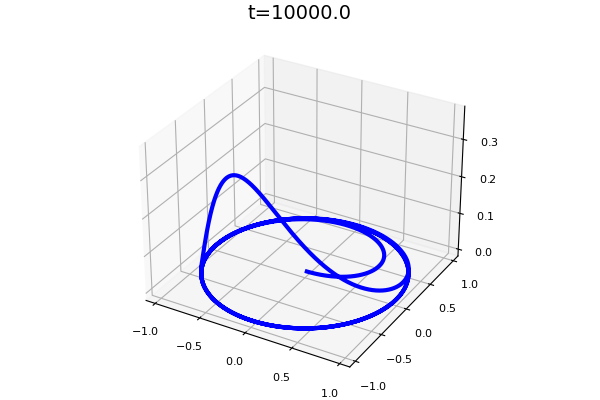}
	\includegraphics[width=6cm]{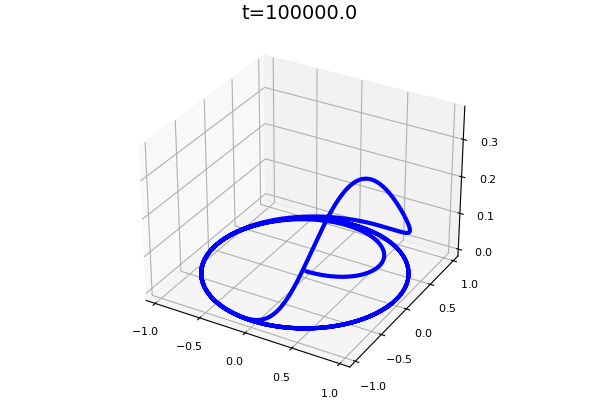}
	\caption{Numerical solution of the example provided in Section \ref{sec:no-convergence} of the main text with $L=1$. The $x, y$ axes correspond to the underlying physical space and we plot $I(t, \tau)$ on the $z$ axis at every $x=(1-e^{-\tau})\cos(\tau)$ and $y=(1-e^{-\tau})\sin(\tau)$. The observation times are spaced exponentially from one another ($t=0, 10, 10^2, 10^3, 10^4, 10^5$) to observe constant shifts of the fixed asymptotic profile. These plots are snapshots of the supplementary movie \texttt{spiraling.avi}. }\label{fig:osc}
\end{figure}
\end{document}